\documentclass[12pt]{amsart}

\usepackage{graphicx}
\usepackage{subcaption}
\usepackage{amsfonts}
\usepackage{epsf}
\usepackage{amssymb}
\usepackage{amsmath}
\usepackage{amscd}
\usepackage{hyperref}
\usepackage{color}

\newtheorem{theorem}{Theorem}[section]
\newtheorem{proposition}[theorem]{Proposition}
\newtheorem{lemma}[theorem]{Lemma}
\newtheorem{question}[theorem]{Question}
\newtheorem{corollary}[theorem]{Corollary} 
\newtheorem{conjecture}[theorem]{Conjecture}

\theoremstyle{remark}
\newtheorem{definition}[theorem]{Definition}
\newtheorem{example}[theorem]{Example}
\newtheorem{remark}{Remark}[section]

\newcommand{\eps}{\epsilon}

\begin{document}

\title{The non-orientable 4-genus for knots with 8 or 9 crossings}

\begin{abstract}    
The non-orientable 4-genus of a knot in the 3-sphere is defined as the smallest first Betti number of any non-orientable surface smoothly and properly embedded in the 4-ball, with boundary the given knot. We compute the non-orientable 4-genus for all knots with crossing number 8 or 9. As applications we prove a conjecture of Murakami's and Yasuhara's, and give computations of the clasp and slicing number of a knot.  
\end{abstract}


\author{Stanislav Jabuka}
\author{Tynan Kelly}

\email{jabuka@unr.edu}
\email{tbkelly@unr.edu}

\address{Department of Mathematics and Statistics, University of Nevada, Reno NV 89557}

\thanks{The first author gratefully acknowledges support from the Simons Foundation, Grant \#246123.}
\maketitle


\section{Introduction}
\subsection{Background} \label{Introduction}
Knots and the surfaces they bound have been intimately related from the origins of knot theory. The classification of surfaces has made it easy to impart a measure of complexity on the knots that bound them. For instance, the Seifert genus $g_3(K)$ of a knot $K$, defined as the minimal genus of any surface $S$ in $S^3$ with $\partial S=K$, was defined by Seifert \cite{Seifert} already in 1935. There are several other natural choices of surfaces to consider, leading to several flavors of knot genera.  

Work of Fox and Milnor \cite{Fox, FoxProblems, FoxMilnor} led to the definition of the {\em (smooth, oriented) 4-genus (or slice genus) $g_4(K)$} of a knot $K$ as the minimal genus of any smoothly and properly embedded surface $S$ in the 4-ball $D^4$ with $\partial S =K$. The {\em topological (oriented) 4-genus $g^{top}_4(K)$} is defined analogously by requiring that the embedding $S \hookrightarrow D^4$ be locally topologically flat instead of smooth. Note that $g^{top}_4(K) \le g_4(K) \le g_3(K)$. 

In another direction, Clark \cite{Clark} defined the {\em non-orientable 3-genus} or {\em 3-dimensional crosscap number} $\gamma_3(K)$ as the smallest first Betti number of any non-orientable surface $\Sigma \subset S^3$ with $\partial \Sigma =K$. The {\em non-orientable (smooth) 4-genus} or {\em 4-dimensional crosscap number} $\gamma_4(K)$ was defined by Murakami and Yasuhara \cite{MurakamiYasuhara2000} as the minimal first Betti number of any non-orientable surface $\Sigma$ smoothly and properly embedded in $D^4$ and with $\partial \Sigma  = K$. Some authors additionally define $\gamma_4(K)=0$ for any slice knot $K$, but in the interest of a more unifying treatment we adopt the definition from the previous sentence. Just as in the case of oriented surfaces, so too for non-orientable surfaces there is a topological version of this invariant denoted by $\gamma_4^{top}(K)$. The inequalities $\gamma_4^{top}(K) \le \gamma_4(K) \le \gamma_3(K)$ again hold in the non-orientable setting. The oriented and non-orientable genera are easily seen to compare as 
\begin{equation} \label{GeneraInequality}
\gamma_i(K) \le 2g_i(K)+1 \quad \text{ for } i=3, 4,
\end{equation}
with an analogous inequality holding for the topological 4-genera. Indeed if $K$ bounds a properly embedded, smooth, genus $g$ surface $S \subset D^4$, then the surface $\Sigma$ obtained from $S$ by removing a disk  neighborhood of an interior point and replacing it by a M\"obius band, has $\partial \Sigma=K$ and $b_1(\Sigma) = 2g+1$, demonstrating $\gamma_4(K) \le 2g_4(K) +1$.

The subject of study in this present work is the smooth non-orientable 4-genus $\gamma_4$. Having been introduced relatively recently, the literature available on $\gamma_4$ is relatively sparse too. First results go back to work of Viro \cite{Viro} who uses Witt classes of intersection forms of 4-manifolds to obstruct a knot $K$ from bounding a smoothly and properly embedded M\"obius band in $D^4$. He uses his findings to demonstrate that $\gamma_4(4_1)>1$.   

Let $\sigma (K)$ and Arf$(K)$ denote the signature and Arf invariant of $K$. Yasuhara \cite{Yasuhara} proves that if a knot $K$ bounds a M\"obius band in $D^4$, then there exists an integer $x$ such that  
\[
|8x+4\cdot \text{Arf}(K) - \sigma (K)| \le 2.
\]
This proves that $\gamma_4(K)>1$ for any knot $K$ with $\sigma (K) + 4\cdot \text{Arf}(K) \equiv 4\pmod{8}$, the knot $K=4_1$ being one example.

Gilmer--Livingston \cite{GilmerLivingston} use linking forms on the 2-fold branched cover of $K$, Heegaard Floer homology and Casson-Gordon invariants, to show, for instance, that $\gamma_4(4_1\#5_1)=3$, the largest known value for $\gamma_4$ at that time, and still the largest known value for $\gamma_4$ among alternating knots (see however Theorem \ref{8CrossingKnots} below). They also prove the following congruence relation 
\begin{equation} \label{YasuharasCongruence}
\sigma (K) + 4\cdot \text{Arf}(K) \equiv \sigma (W(\Sigma )) - \beta (D^4, \Sigma )\pmod{8}, 
\end{equation} 
valid for any knot $K$ that bounds a non-orientable surface $\Sigma$ smoothly and properly embedded in $D^4$. Here $\sigma (K)$ and Arf$(K)$ are as above, while $W(\Sigma)$ denotes the 2-fold cover of $D^4$ branched along $\Sigma$, and $\sigma (W(\Sigma))$ denotes its signature. Lastly, $\beta (D^4, \Sigma)$ is the Brown invariant \cite{GilmerLivingston, KirbyMelvin} of the pair $(D^4,\Sigma)$. It is easy to show that rk $H_2(W(\Sigma);\mathbb Z) = \text{rk }H_1(\Sigma;\mathbb Z)$ implying the bound $|\sigma (W(\Sigma))|\le \text{rk } H_1(\Sigma ;\mathbb Z)$, while work in \cite{KirbyMelvin} shows the same bound to hold for the Brown invariant (see also Corollary \ref{CorollaryOnTheBrownInvariantOfASurface}). The congruence \eqref{YasuharasCongruence}, along with the discussion of this paragraph, implies again that if $K$ is a knot with $\sigma (K) + 4\cdot \text{Arf}(K) \equiv 4\pmod{8}$, then $K$ cannot bound an embedded M\"obius band in $D^4$. Relation \eqref{YasuharasCongruence} makes frequent appearances throughout this work, and we shall refer to it as the {\em Gilmer--Livingston (congruence) relation}.

Using tools from Heegaard Floer homology, Batson \cite{Batson} is able to show that $\gamma_4$ is an unbounded function. He does so by proving the bound 
\begin{equation} \label{BatsonsBound}
\frac{\sigma(K)}{2} - d(S^3_{-1}(K)) \le \gamma_4(K),
\end{equation}
whose notation we proceed to explain now. For a rational number $r$ we let $S^3_r(K)$ be the manifold resulting from $r$-framed Dehn surgery on $K$, and for an integral homology 3-sphere $Y$, we use $d(Y)$ to denote its Heegaard Floer correction term \cite{OzsvathSzabo2003}. Batson shows that the left-hand side of \eqref{BatsonsBound} equals $k-1$ for the torus knot $K=T_{(2k,2k-1)}$, $k\in \mathbb N$, demonstrating the unboundedness of $\gamma_4$. 

Ozsv\'ath--Stipsicz--Szab\'o \cite{OzsvathStipsiczSzabo} in 2015 define a concordance invariant $v(K)$, derived from their family of concordance invariants $\Upsilon_K(t)$ \cite{OzsvathStipsiczSzabo2014}. They prove the lower bound 
\begin{equation} \label{OSSBound}
\left| \frac{\sigma (K)}{2}-v(K) \right| \le \gamma_4(K),
\end{equation}
and use it to provide another proof of the unboundedness of $\gamma_4$ by demonstrating that $\gamma _4( \#^nT_{(3,4)}) \ge n$, where $\#^nT_{(3,4)}$ is the $n$-fold connected sum of the $(3,4)$ torus knot $T_{(3,4)}$ with itself. The converse inequality $\gamma _4( \#^nT_{(3,4)}) \le n$ is easy to verify by finding an explicit M\" obius band bounded by $T_{(3,4)}$, leading to $\gamma _4( \#^nT_{(3,4)}) = n$.
\subsection{Results and Applications}
As of this writing, the KnotInfo \cite{Knotinfo} knot tables only contain values for $\gamma_4$ for knots with 7 or fewer crossings. Our goal and the main result of this work is to extend this tabulation to include all 70 knots with 8 and 9 crossings. 
\begin{theorem} \label{8CrossingKnots}
The values of $\gamma_4$ for the 21 knots with crossing number 8, are given as follows.  
\[
\begin{array}{lcl}
\gamma_4(K)=1  & \quad\text{ for } \quad & K = 8_3,\, 8_4,\, 8_5,\, 8_6,\, 8_7,\, 8_8,\, 8_9,\, 8_{10},\, 8_{11}, \, 8_{14},\, 8_{16}, \, 8_{19}, \, 8_{20}.  \cr  &&\cr
\gamma_4(K)=2  & \quad\text{ for } \quad & K = 8_1, \, 8_2,\, 8_{12},\, 8_{13},\, \, 8_{15}, \,8_{17}, \, 8_{21}.  \cr  &&\cr
\gamma_4(K)=3  & \quad\text{ for } \quad & K = 8_{18}.
\end{array}
\]
\end{theorem}
\begin{theorem} \label{9CrossingKnots}
The values of $\gamma_4$ for the 49 knots with crossing number 9, are given as follows.  
\[
\begin{array}{lcl}
\gamma_4(K)=1  & \text{ for }  & K =  9_1, \, 9_3, \, 9_4,\, 9_5,\, 9_6, \, 9_7,\, 9_8,\, 9_9, \, 9_{13},\, 9_{15}, \, 9_{17},\, 9_{19},\, 9_{21}, \, 9_{22}, \, 9_{23}, \cr  
&& \hspace*{9.5mm} 9_{25}, \, 9_{26},\, 9_{27}, \, 9_{28}, \, 9_{29}, \, 9_{31}, \, 9_{32},\, 9_{35}, \, 9_{36},\, 9_{41}, \, 9_{42}, \, 9_{43},\, 9_{44}, \, 9_{45}, \cr
&& \hspace*{9.5mm}  9_{46}, \, 9_{47},\, 9_{48}.\cr && \cr
\gamma_4(K)=2  & \text{ for } & K =  9_2,\, 9_{10}, \, 9_{11},\, 9_{12},\,  9_{14}, \, 9_{16},\, 9_{18}, \, 9_{20}, \, 9_{24},\, 9_{30}, \, 9_{33},\, 9_{34},\, 9_{37}, \cr
&& \hspace*{9.5mm} 9_{38},\,  9_{39},\, 9_{40}. \cr && \cr
\gamma_4(K)=3  & \quad\text{ for } \quad & K = 9_{49}.
\end{array}
\]
\end{theorem}
As already mentioned in Section \ref{Introduction}, the non-orientable slice genus $\gamma_4(K)$ was introduced by Murakami and Yasuhara in their work \cite{MurakamiYasuhara2000}, with the difference that in \cite{MurakamiYasuhara2000} $\gamma_4$ of a slice knot $K$ is defined to be zero. Murakami and Yasuhara observed the inequality $\gamma_4(K) \le 2g_4(K) +1$, the $i=4$ version of \eqref{GeneraInequality}. In Conjecture 2.10 from \cite{MurakamiYasuhara2000} they ask whether this inequality is the best possible bound relating $\gamma_4(K)$ and $g_4(K)$. 
\setcounter{section}{2}
\setcounter{theorem}{9}
\begin{conjecture}[Murakami and Yasuhara \cite{MurakamiYasuhara2000}] There exists a non-slice knot $K$ such that $\gamma_4(K) = 2g_4(K)+1$. 	
\end{conjecture}
\setcounter{section}{1}
\setcounter{theorem}{2}
Theorem \ref{8CrossingKnots} verifies the Murakami-Yasuhara Conjecture. 
\begin{corollary}
There exist non-slice knots $K$, for instance $K=8_{18}$, such that $\gamma_4(K) = 2g_4(K)+1$. Accordingly, the inequality $\gamma_4(K) \le 2g_4(K)+1$ is sharp for some knots and cannot be improved upon. 
\end{corollary}

Recall that the {\em unknotting number} $u(K)$ of a knot $K$ is the minimum number of crossing changes in any diagram of $K$ that renders $K$ unknotted. Similarly, the {\em slicing number} $u_s(K)$ of a knot $K$ is defined as the minimum number of crossing changes in any diagram of $K$ that transforms $K$ into a slice knot. These two quantities fit into the double inequality 
\begin{equation} \label{DoubleInequalityWithG4UsAndU}
g_4(K) \le u_s(K) \le u(K).
\end{equation}
Of these, a proof of the left inequality can be found in \cite{Scharlemann}, while the right inequality is obvious since the unknot is slice. The inequality $u_s(K)\le u(K)$ is a strict inequality for many knots $K$, for instance for any nontrivial slice knot $K$. To show that the inequality $g_4(K) \le u_s(K)$ may also be strict is rather more difficult. The first example of a knot $K$ where this occurs, namely $K=7_4$, was discovered by Livingston \cite{Livingston}. Owens \cite{Owens} and Owens--Strle \cite{OwensStrle}, by relying on gauge-theoretic techniques, are able to calculate $u_s(K)$ for all knots $K$ with 10 or fewer crossings, and find many more examples with $g_4(K)<u_s(K)$. In general however, both $u(K)$ and $u_s(K)$ remain difficult knot invariants to compute. 

The {\em 4-dimensional clasp number $c_4(K)$} of a knot $K$ is the smallest number of double points of any immersed disk in the 4-ball $D^4$, with boundary $K$. The clasp number also fits into a double inequality, namely
\begin{equation} \label{InequalityForG4C4AndUs}
g_4(K) \le c_4(K) \le u_s(K), 
\end{equation}
of which the left one is proved in \cite{Shibuya}, while the right one is obvious. The inequality $g_4(K)\le c_4(K)$ may be strict, an example is given in \cite{MurakamiYasuhara2000}, but we are not aware of a knot $K$ with $c_4(K) < u_s(K)$. The relation between the non-orientable 4-genus $\gamma_4(K)$ and the clasp number $c_4(K)$ was worked out by Murakami-Yasuhara  \cite{MurakamiYasuhara2000}.
\begin{proposition}[Proposition 2.3 in  \cite{MurakamiYasuhara2000}] \label{PropositionOnGamma4AndC4}
For any knot $K$ the following inequality holds.
\begin{equation} \label{BoundOnGamma4ByClaspNumber}
\gamma_4(K) \le 
\begin{cases}
c_4(K)\ ;  & \text{if } c_4(K) \text{ is even and } c_4(K)\ne 2, \\
c_4(K)+1\ ; & \text{otherwise}.
\end{cases}
\end{equation}
\end{proposition}
This inequality and its proof were independently communicated to us by Chuck Livingston, whose input we gratefully acknowledge. The reason for this detour into exploring $c_4(K)$ and $u_s(K)$ is to demonstrate in the next example that our computation of $\gamma_4(8_{18})$ in conjunction with Proposition \ref{PropositionOnGamma4AndC4} can be used to obtain a proof of the strict inequalities $g_4(8_{18}) < c_4(8_{18})$ and $g_4(8_{18}) < u_s(8_{18})$, facts that were first obtained by Owens-Strle \cite{OwensStrle} using rather different techniques.  
\begin{example} \label{CalculatingUsOf8_18}
The knot $K=8_{18}$ has $\gamma_4(K) = 3$ according to Theorem \ref{8CrossingKnots}. Proposition \ref{PropositionOnGamma4AndC4} implies that $2\le c_4(K)$, and thus $2\le u_s(K)$ by \eqref{InequalityForG4C4AndUs}. Since $u(K) = 2$ we obtain $c_4(K) = u_s(K)=2$ by \eqref{DoubleInequalityWithG4UsAndU} and \eqref{InequalityForG4C4AndUs}, while $g_4(K)=1$. 
\end{example}
\subsection{Organization} In Section \ref{SectionBackground} we provide needed background material. We remind the reader of the definition of the Brown invariant, introduce non-oriented band moves on knot diagrams, and review the Goeritz form of a knot and Donaldson's Diagonalization Theorem. The main results of this section are the obstruction   Theorems \ref{TheoremWithLowerBoundOnGamma4WhenSigmaPlus4ArfEqual2}, \ref{TheoremWithLowerBoundOnGamma4WhenSigmaPlus4ArfEqual4} and \ref{TheoremWithLowerBoundOnGamma4WhenSigmaPlus4ArfEqual0}. Section \ref{SectionWhereWeCalculateStuff} looks at all 70 knots with crossing number 8 or 9, and employs the techniques from Section \ref{SectionBackground} to compute their values of $\gamma_4$, providing proofs of Theorems \ref{8CrossingKnots} and \ref{9CrossingKnots}. 
Section \ref{SectionOnConcludingRemarks} concludes with some observations and open questions. 
\subsection{Acknowledgements}  We wish to thank Pat Gilmer and Chuck Livingston for helpful comments. The first author gratefully acknowledges support from the Simons Foundation, Grant \#246123. 
\section{Background} \label{SectionBackground}
This section describes the techniques used to determined the values of $\gamma_4$ for knots with 8 or 9 crossings. The techniques come in two flavors -- constructive and obstructive. The former takes the form of a non-oriented band move on knot diagrams (described in Section \ref{SectionOnNonOrientedBandMoves}) in such a way that if two knots are related by such a move, their values of $\gamma_4$ differ by at most 1; see Proposition \ref{NonOrientableBandMoveLemma}. The obstructive techniques use Donaldson's celebrated diagonalization theorem for definite 4-manifolds, in combination with a construction of Goeritz. These are described in Section \ref{SectionOnGoeritzAndDonaldson}. 
\subsection{The Brown invariant} This section recalls the definition of the Brown invariant $\beta (D^4, \Sigma)$ of a smoothly and properly embedded non-orientable surface $\Sigma \hookrightarrow D^4$. Our exposition follows that of \cite{KirbyMelvin}. 
 
Let $V$ be a finite dimensional $\mathbb Z_2$-vector space equipped with a nonsingular inner product $\cdot :V\times V\to \mathbb Z_2$, that is, an inner product for which $x\cdot y=0$ for all $y\in V$ implies $x=0$. We call $(V,\cdot)$ {\em even} if $x\cdot x =0$ for all $x\in V$, otherwise we say $(V,\cdot)$ is {\em odd}. Every such inner product space $(V,\cdot)$ can be decomposed as a direct sum of orthogonal subspaces isomorphic to 
\[
P=\mathbb Z_2 x \quad \text{ and } \quad T=\mathbb Z_2y\oplus \mathbb Z_2z
\]
with  $x\cdot x=1=y\cdot z$ and $y\cdot y=0=z\cdot z$. These two irreducible spaces satisfy the isomorphism $P\oplus T \cong P\oplus P\oplus P$, and there are no other relations among them. Accordingly, every inner product space $(V,\cdot)$ is isomorphic to either $mP$ (the $m$-fold orthogonal sum of $P$) or $nT$ (the $n$-fold orthogonal sum of $T$). The former are the odd inner product spaces, the latter the even ones.  

A {\em quadratic form on $(V,\cdot)$} is a function $q:V\to \mathbb Z_4$ with $q(x+y) = q(x) + q(y) +2x \cdot y$ for all $x,y \in V$. Here $\cdot 2:\mathbb Z_2 \to \mathbb Z_4$ is the unique homomorphism sending 1 to 2. Restricting $q$ to the irreducible summands of $(V,\cdot)$ gives a decomposition of $q$ as a sum of quadratic forms on $P$ or $T$. 

The space $P=\mathbb Z_2x$ admits two quadratic forms $q_{-1}$ and $q_1$, defined by $q_i(x) = i$. Similarly, the space $T= \mathbb Z_2 y \oplus \mathbb Z_2z$ admits exactly 4 quadratic forms $q_{0,0}$, $q_{0,2}$, $q_{2,0}$ and $q_{2,2}$, given by $q_{i,j}(y) = i$ and $q_{i,j}(z) = j$. Of these the first three are mutually isomorphic, but are not isomorphic to the fourth one, giving precisely two isomorphism classes of quadratic forms on $T$. 

The relation $P\oplus T \cong 3P$ of inner product spaces induces relations among the quadratic forms $q_{i,j}$ and $q_k$ as: $q_{\pm 1}\oplus q_{0,0} \cong q_{\pm 1}\oplus q_{-1}\oplus q_1$ and $q_{\pm 1}\oplus q_{2,2} \cong 3 q_{\mp 1}$. Of course since $q_{0,0}$ is isomorphic to both $q_{0,2}$ and $q_{2,0}$, we may replace $q_{0,0}$ in the first relation above by either of $q_{0,2}$ or $q_{2,0}$. These relations further imply the relations 
\begin{equation}  \label{RelationsAmongEnhancements}
2q_{0,0}\cong 2q_{2,2} \quad \text{ and } \quad 4q_{-1}\cong 4q_1,
\end{equation}
which lead to the following unique decomposition of a quadratic form $(V,\cdot,q)$:  
\[
q \cong \begin{cases}
\text{Direct sums of } q_{0,0}\text{s and } q_{2,2}\text{s with at most 1 copy of } q_{2,2}\text{ ;} & \text{if } (V,\cdot) \text{ is even,}\\
\text{Direct sums of } q_{-1}\text{s and }q_{1}\text{s with at most 3 copies of } q_{-1}\text{ ;} & \text{if } (V,\cdot) \text{ is odd.}
\end{cases}
\]

We define the {\em Brown invariant $\beta (q)\in \mathbb Z_8$ of $(V,\cdot , q)$} by setting
\[
\beta (q_{0,0}) = \beta(q_{0,2}) = \beta(q_{2,0}) = 0,\quad  \beta (q_{2,2}) = 4,\quad  \beta (q_{-1}) = -1, \quad \beta (q_1) = 1,
\]
and by imposing additivity $\beta(q'\oplus q'') = \beta (q') + \beta (q'')$ under the direct sum of the quadratic forms $q'$ and $q''$. The relations \eqref{RelationsAmongEnhancements} show that the Brown invariant is well defined modulo 8.  

For the next lemma we define the norm $|x|$ for $x\in \mathbb Z_8$ as the smallest absolute value $|y|$ with $x\equiv y\pmod{8}$. For example $|7| = 1$.
\begin{lemma} \label{LemmaOnSizeOfBrownInvariant}
For an odd quadratic inner product space $(V,\cdot, q)$, the inequality 
\[
|\beta(q)|\le \dim _{\mathbb Z_2}V
\]
holds.  
\end{lemma}
\begin{proof}
Since $(V,\cdot)$ is odd we can write $(V,\cdot)\cong nP$ with $n=\dim _{\mathbb Z_2}V$. Then $q$ is isomorphic to an $n$-fold direct sum of copies of $P_{-1}$ and $P_1$, and its Brown invariant $\beta(q)$ is therefore an $n$-fold sum whose summands are either $-1$ or $1$. It follows that $|\beta(q)|\le n$ as claimed.  	
\end{proof}

Given a non-orientable surface $\Sigma \subset D^4$, smoothly and properly embedded, Guillou and Marin \cite{GuillouMarin} define an odd form $q_\Sigma :H_1(\Sigma ;\mathbb Z_2)\to \mathbb Z_2$ that is quadratic with respect to the linking pairing $\cdot$ on $H_1(\Sigma ;\mathbb Z_2)$. We omit the details of the definition of $q_\Sigma$ as they are not relevant to our subsequent discussion. The {\em Brown invariant $\beta (\Sigma , D^4)\in \mathbb Z_8$} of the embedding $\Sigma \subset D^4$ is defined as the Brown invariant $\beta (q_\Sigma)$ of the quadratic inner product space $(H_1(\Sigma ;\mathbb Z_2), \cdot, q_\Sigma)$. The following is a direct consequence of Lemma \ref{LemmaOnSizeOfBrownInvariant}.
\begin{corollary} \label{CorollaryOnTheBrownInvariantOfASurface}
For a non-orientable surface $\Sigma \subset D^4$, smoothly and properly embedded, the inequality $|\beta (\Sigma , D^4)|\le b_1(\Sigma)$ holds.
\end{corollary}

\subsection{Non-oriented band moves} \label{SectionOnNonOrientedBandMoves}
We describe here a move on knots that will be one of our fundamental tools in seeking concrete non-orientable surfaces $\Sigma$, smoothly and properly embedded in $D^4$, with boundary a given knot $K$.  

\begin{definition}
A {\em non-oriented band move} on an oriented knot $K$ is the operation of attaching an oriented band $h=[0,1]\times [0,1]$ to $K$ along $[0,1]\times \partial [0,1]$ in such a way that the orientation of the knot agrees with that of $[0,1]\times \{0\}$ and disagrees with that of $[0,1]\times \{1\}$ (or vice versa), and then performing surgery on $h$, that is replacing the arcs $[0,1]\times \partial [0,1]\subseteq K$ by the arcs $\partial [0,1]\times [0,1]$.  

The resulting knot $K'$ shall be said to have been {\em obtained from $K$ by a non-oriented band move}, and we write $K' = K\#h$ to indicate this operation. 
\end{definition}
Note that if $K'$ was obtained from $K$ by a non-oriented band move and $K'=K\#h$, then the knot $K$ is also obtained from $K'$ by a non-oriented band move and $K=K'\#h'$ where $h'$ is the ``dual band''  of $h$, see Figure~\ref{DualBands}. 
\begin{figure}
\includegraphics[width=12cm]{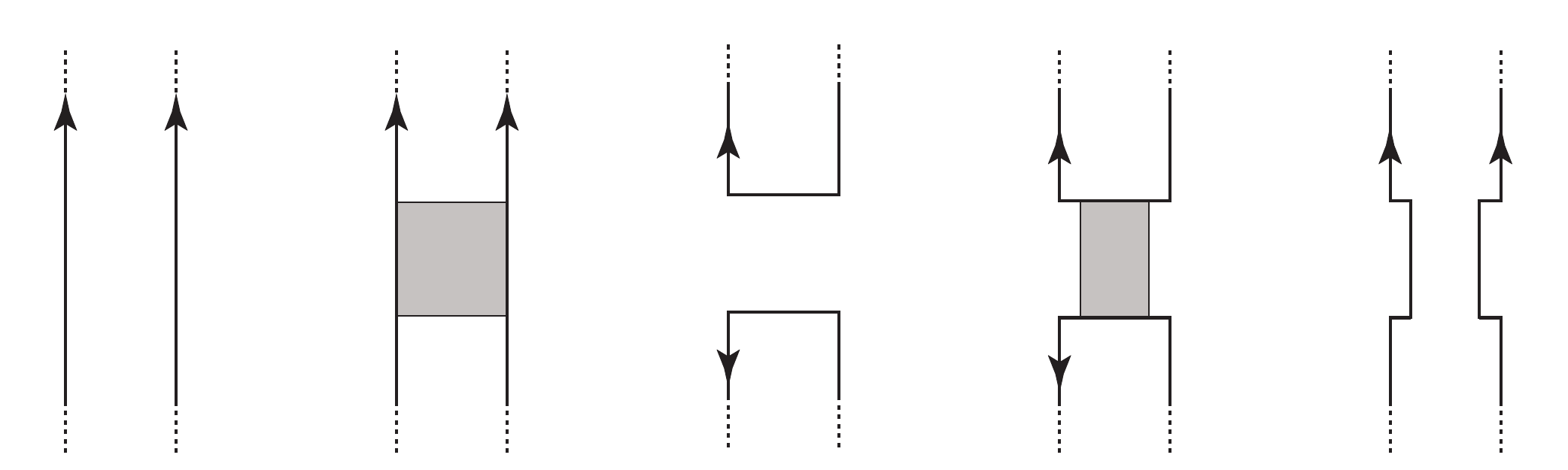}
\put(-320,-15){$K$}
\put(-247,43){$h$}
\put(-200,-15){$K'=K\#h$}
\put(-102,43){$h'$}
\put(-53,-15){$K=K'\#h'$}
\caption{Non-orientable band moves are symmetric: If a knot $K'$ is obtained from a knot $K$ by a non-orientable band move using the band $h$, then $K$ is also obtained from $K'$ by a non-orientable band move using the ``dual band'' $h'$ of $h$. }
\label{DualBands}
\end{figure}
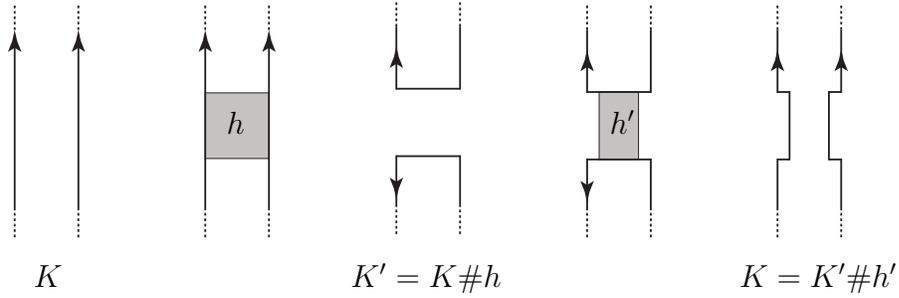

\begin{remark}[On band-move notation]
Before proceeding, we pause to introduce some pictorial notation for band moves. We shall represent a band $h$ in a knot diagram of $K$ by drawing a dotted line representing the core $\{\frac{1}{2}\}\times [0,1]$ of $h$. We shall then use an integer $n$ to indicate the number of half-twists to be introduced into $h$ with respect to the blackboard (or paper) framing, where as is usual $n>0$ corresponds to $n$ right-handed half-twists and $n<0$ corresponds to $|n|$ left-handed half-twists. This framing shall appear in the caption of the figure, where we write $K\stackrel{n}{\longrightarrow} K'$ to indicate that $K'=K\#h$ and $h$ is the band obtained from its core $\{\frac{1}{2}\} \times [0,1]$ by adding the framing $n$. Figure~\ref{HandleLabelingConvention} illustrates this convention. 
\end{remark}
\begin{remark}

We note that in writing $K\stackrel{n}{\longrightarrow} K'$ we mean that the knot $K$ under the indicated non-oriented band move transforms into either the knot $K'$ or its reverse mirror knot $-K'$. In all of our computations we determined $K'$ from its crossing number and its Alexander polynomial, two data points which do not differentiate between $K'$ and $-K'$. Since $\gamma_4(K') = \gamma_4(-K')$ this does not affect our claims. 
\end{remark}
\begin{figure}
\centering
\begin{subfigure}[b]{0.4\textwidth}
        \includegraphics[width=\textwidth]{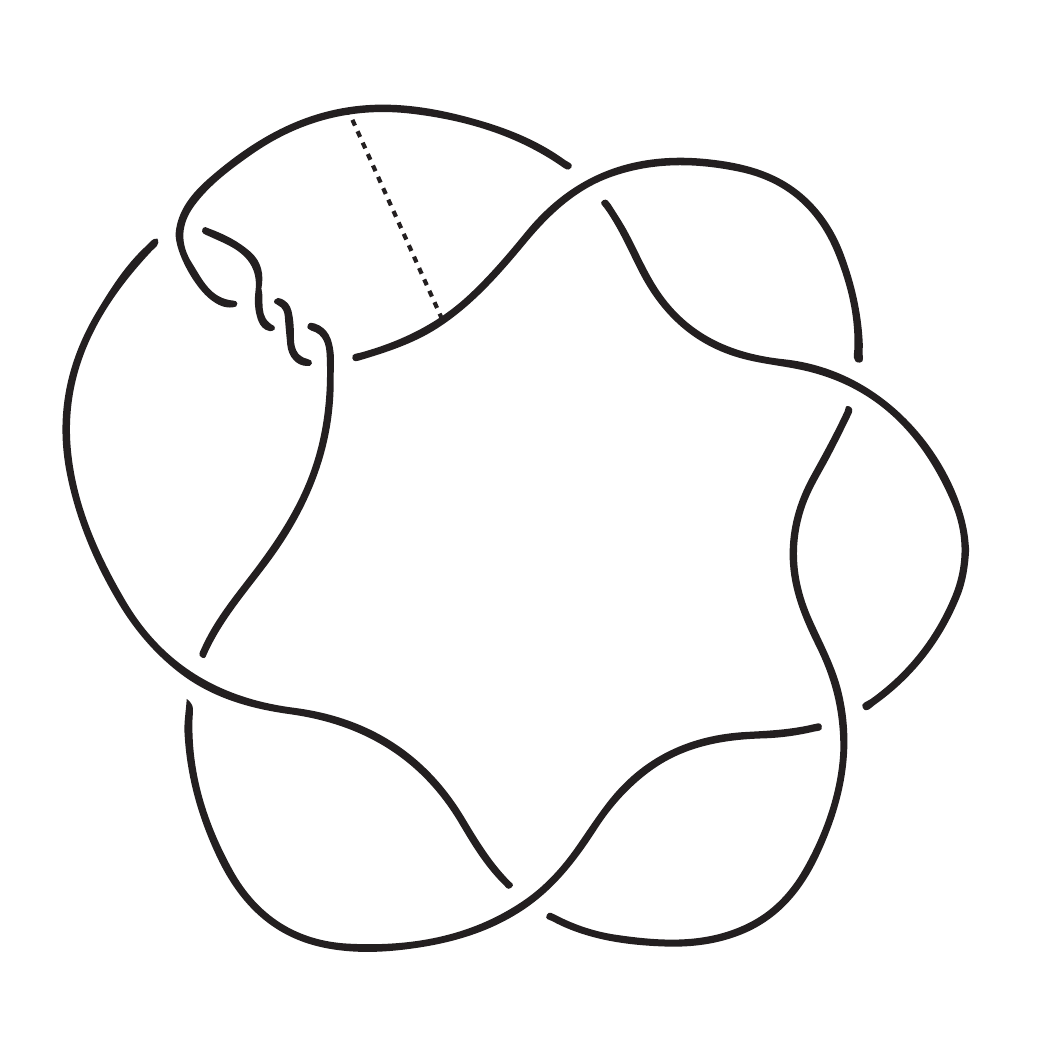}
        \caption{$9_{4}\stackrel{1}{\longrightarrow} 10_3$}
        \label{BandNotation1}
\end{subfigure}
\quad \quad \quad 
\begin{subfigure}[b]{0.4\textwidth}
        \includegraphics[width=\textwidth]{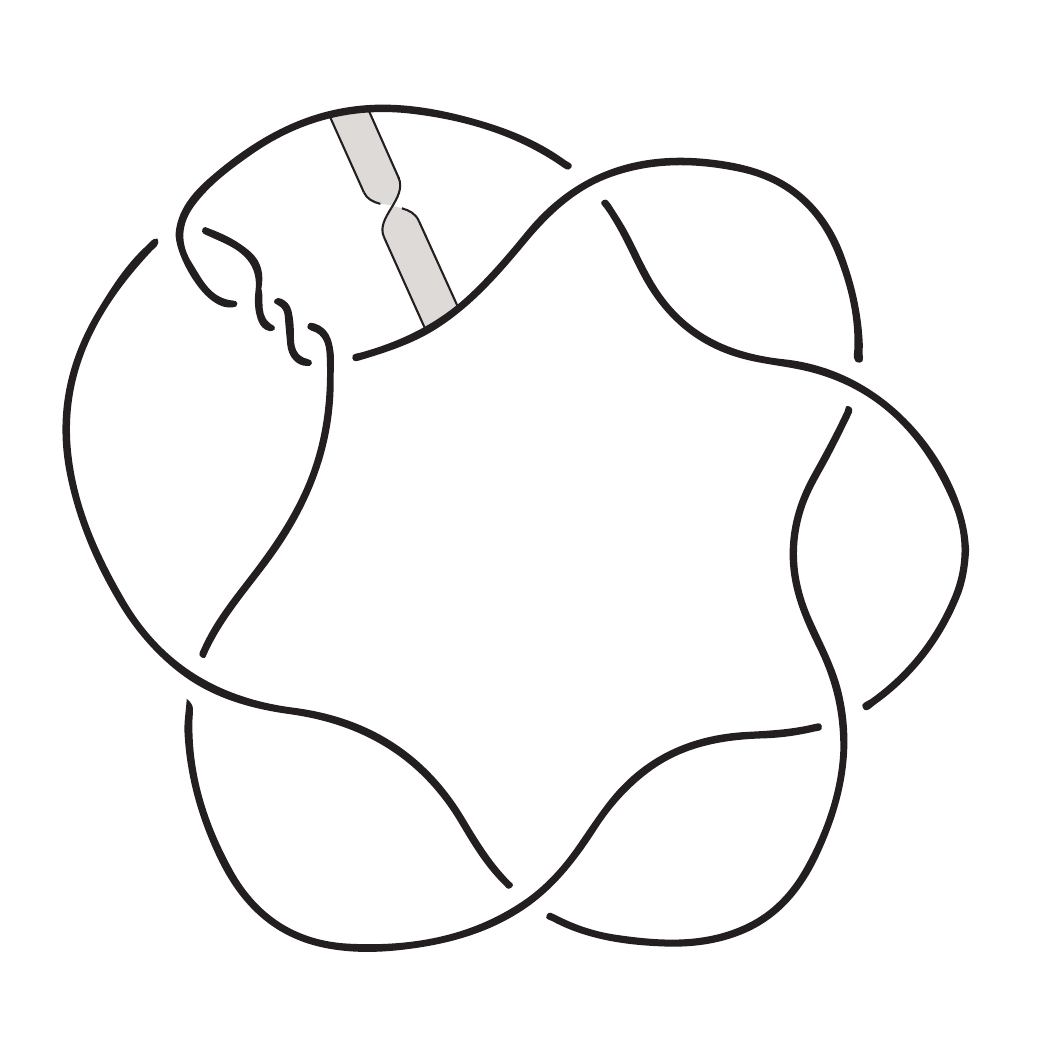}
        \put(-103,140){$h$}
        \caption{$10_{3} = 9_4\#h$}
        \label{BandNotation2}
\end{subfigure}
\vskip3mm
\caption{Our convention for labeling non-oriented bands is illustrated in Subfigure~\ref{BandNotation1}, where the handle $h$ is represented by a dotted line, and its ``framing'' of $1$ is indicated in the caption of that subfigure. The band is fully drawn, including its right-handed half-twist, in Subfigure~\ref{BandNotation2}. The caption of Subfigure~\ref{BandNotation1} is shorthand notation for the caption in Subfigure~\ref{BandNotation2}. As convention dictates, ``positive framings'' correspond to right-handed half-twists, and ``negative framings''  to left-handed half-twists.   }\label{HandleLabelingConvention}
\end{figure}

The following proposition is an easy but very useful observation. 
\begin{proposition} \label{NonOrientableBandMoveLemma}
If the knots $K$ and $K'$ are related by a non-oriented band move, then 
\[
\gamma_4(K) \le \gamma_4(K') +1.
\]
If a knot $K$ is related to a slice knot $K'$ by a non-oriented band move, then $\gamma_4(K)=1$. 
\end{proposition}
\begin{proof}
Let $\Sigma'$ be a non-orientable smoothly embedded surface in $D^4$ with $\partial \Sigma' =K'$ and with $b_1(\Sigma' ) = \gamma_4(K')$. Let $h$ be a band such that $K$ is obtained from $K'$ by a non-oriented band move on $h$, and let $\Sigma$ be the surface in $D^4$ obtained by attaching the band $h$ to $\Sigma'$ along $[0,1]\times \partial [0,1] \subseteq K'$, and pushing the interior of $h$ into $D^4$ so as to make $\Sigma$ properly (and smoothly) embedded in $D^4$. Then $\Sigma $ is a non-orientable surface with $\partial \Sigma =K$ and with $b_1(\Sigma) = b_1(\Sigma ')+1$, and so 
\[
\gamma_4(K) \le b_1(\Sigma ) = b_1(\Sigma ')+1 = \gamma_4(K')+1,
\]
as needed. If $K'$ is slice, the above construction can be repeated by using a slice disk for $\Sigma '$, rendering $\Sigma$ a M\"obius band. 
\end{proof}
\subsection{Goeritz forms and Donaldson's Diagonalization Theorem} \label{SectionOnGoeritzAndDonaldson}
Associated to a projection $D$ of knot $K$ are two ``black-and-white'' checkerboard colorings. Each is a coloring of the regions of the knot projection with black and white colors, so that no two regions sharing an edge receive the same color. There are exactly two such colorings, one in which the unbounded region is colored white, the other in which it is black. 

Associated to either checkerboard coloring of the knot projection $D$ is a bilinear form first described by Goeritz \cite{Goeritz}. Our exposition follows that given by Gordon and Litherland \cite{GordonLitherland}. 
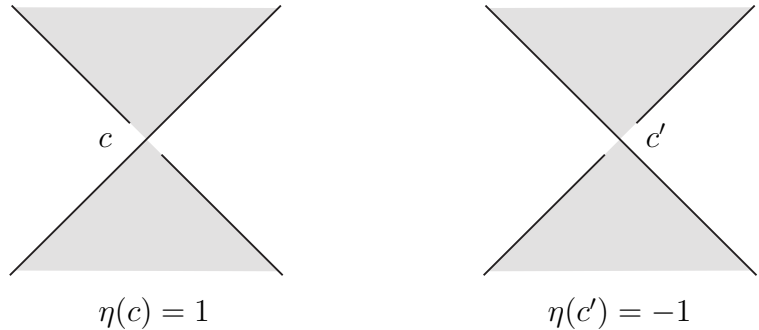
\begin{figure}
\includegraphics[width=10cm]{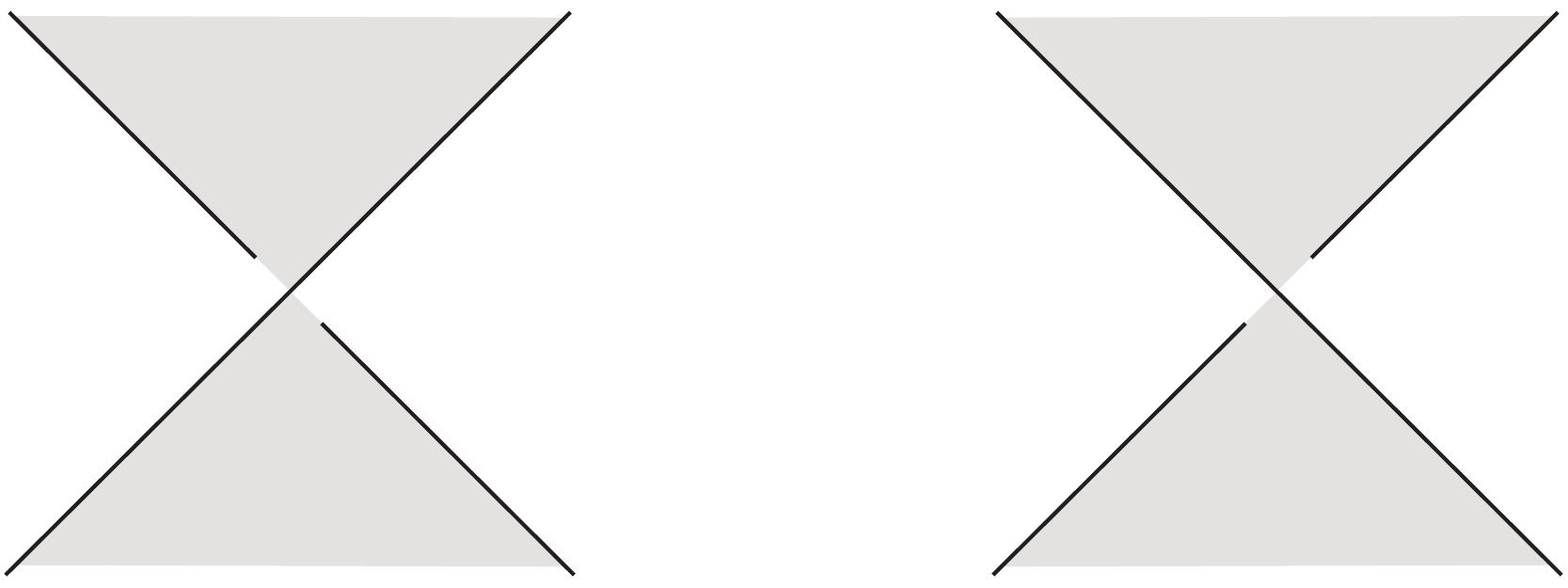}
\put(-250,51){$c$}
\put(-43,51){$c'$}
\put(-250,-15){$\eta(c) = 1$}
\put(-80,-15){$\eta(c') = -1$}
\caption{The weights $\pm 1$ associated to the two different types of crossings $c$ and $c'$.}
\label{CrossingTypesWithWeights}
\end{figure}

Let $X_0, X_1,\dots, X_n$ denote the white regions in the checkerboard coloring of $D$. We associate to every crossing $c$ in $D$ a weight $\eta (c) = \pm 1$ as described in Figure~\ref{CrossingTypesWithWeights}.  Let $P_{i,j}$ be the set of double points in $D$ that are incident to both $X_i$ and $X_j$. For $i,j\in \{0,\dots, n\}$ let $g_{ij}$ be the integer obtained as
\begin{equation} \label{PreGoeritzMatrixCoefficients}
g_{ij} = \begin{cases}
\displaystyle -\sum_{c\in P_{i,j}} \eta(c)\text{ ;} &  i\ne j, \\[20pt]
\displaystyle -\sum _{k\ne i} g_{ik}\text{ ;} & i=j.
\end{cases}
\end{equation} 
Let $G' = [g_{ij}]$ be the $(n+1)\times (n+1)$ matrix comprised of the coefficients $g_{ij}$, we refer to $G'$ as the {\em pre-Goeritz matrix} associated to the above choice of checkerboard coloring of $D$. The {\em Goeritz matrix} $G=[g_{ij}]_{i,j=1,\dots, n}$ is the $n\times n$ matrix obtained from $G'$ by deleting its 0-th row and column. The bilinear form $(\mathbb Z^n, G)$ is symmetric and non-degenerate, indeed $\det G = \det K$.

Let $F$ be a smoothly and properly embedded surface in $D^4$ with $\partial F  = K$, and let $W=W(F)$ be the twofold cover of $D^4$ with branching set $F$. The surface $F$ may be chosen to be either oriented or non-orientable.  Note that the boundary of $W$ is $Y=Y(K)$ -- the twofold cover of $S^3$ with branching set $K$. We denote by $Q_W$ the intersection form on $H_2(W;\mathbb Z)/Tor$. The following result is Theorem 3 in \cite{GordonLitherland}.

\begin{theorem}[Gordon--Litherland, \cite{GordonLitherland}]  \label{GordonLitherlandTheorem}
Let $K$ be a knot and $D$ a projection of $K$. Pick a checkerboard coloring of $D$ and let $n+1$ be the number of white regions. Let $F'$ be the  surface with boundary equal to $K$, obtained from the black regions (with twisted bands added to connect the black disks) and let $F$ be obtained from $F'$ by pushing its interior into $D^4$. With $W=W(F)$ described as above, there is an isomorphism    
\[
(H_2(W;\mathbb Z)/Tor, Q_W)\cong (\mathbb Z^n, G)
\]
of integral, symmetric, bilinear forms. 
\end{theorem}
\begin{corollary}\label{CorollaryOnBoundingDefiniteFourManifolds}
If a knot $K$ has a projection $D$ with a (positive or negative) definite Goeritz matrix, $G$, then its twofold branched cover $Y(K)$ bounds a smooth, compact (positive or negative) definite 4-manifold $W$.  
\end{corollary}
If $D$ is an alternating knot projection of a knot $K$ then the Goeritz matrices $G_{\pm}$ associated to either of the two possible checkerboard colorings of $D$ are definite, the subscript in $G_{\pm}$ indicating the type of definiteness for each case. Indeed, by a beautiful result of Greene's \cite{Greene}, this property characterizes alternating knots. We note that all knots with 8 or 9 crossings have alternating diagrams, with the exception of the 11 knots 
\begin{equation} \label{TheNonAlternatinKnots}
8_{19}, 8_{20}, 8_{21}, 9_{42}, 9_{43}, 9_{44}, 9_{45}, 9_{46}, 9_{47}, 9_{48}, 9_{49}.
\end{equation}

Theorem \ref{GordonLitherlandTheorem} points to the importance of understanding the 4-manifold $W(\Sigma)$, and we pause to elucidate some of its algebro-topological properties before continuing. It is not hard to show that $b_2(W(\Sigma)) = b_1(\Sigma)$; see Lemma 1 in \cite{GilmerLivingston}. More difficult is the proof of $b_1(W(\Sigma))=0$, which is given in Lemma 2 in \cite{Matsumoto}. The next proposition, whose proof can be found in \cite{Gilmer, OwensStrle2006}, describes other relevant aspects of the algebraic topology of $W(\Sigma)$.
\begin{proposition}[Gilmer \cite{Gilmer}, Owens-Strle \cite{OwensStrle2006}] \label{AlgebraicTopologicalTheorem}
Let $W$ be a smooth, oriented, compact 4-manifold with $b_1(W)=0$, and with $Y=\partial W$ a rational homology 3-sphere. Let $\ell$ denote the determinant of the intersection pairing 
\[
Q_W:H_2(W;\mathbb Z)/Tor \otimes H_2(W;\mathbb Z)/Tor \to \mathbb Z,
\]
and let $n$ be the order of 
\[
\text{Im}(Tor(H_2(W;\mathbb Z))) \to Tor(H_2(W,Y;\mathbb Z)).
\]
Then $|H_1(Y)| = \ell\cdot n^2$. 
\end{proposition}
\begin{corollary} \label{AboutSquareInCaseOfMobiusBand}
Let $K$ be a knot bounding a M\"obius band $\Sigma \subset D^4$ and let $W=W(\Sigma)$. Then the absolute value of the square of the single generator of $H_2(W;\mathbb Z)/Tor$ equals a natural number $\ell$ that divides $\det K$ with quotient a square. In particular, if $\det K$ is square-free then $\ell=\pm \det K$. 
\end{corollary}

We conclude this section by quoting a beautiful result of Donaldson's. 
\begin{theorem}[Donaldson, \cite{Donaldson}] \label{TheoremA}
	Let $X$ be a smooth, closed, oriented 4-manifold whose intersection form $(H_2(X;\mathbb Z)/Tor, Q_X)$ is definite, then $Q_X$ is diagonalizable over $\mathbb Z$. 
\end{theorem}
\subsection{Lower bounds on $\gamma_4(K)$}
A combination of the Gilmer--Livingston congruence relation \eqref{YasuharasCongruence} and Theorems \ref{GordonLitherlandTheorem} and  \ref{TheoremA} can be used to obtain lower bounds on $\gamma_4(K)$ for certain knots $K$. We distinguish three cases according to whether $\sigma (K) +4\cdot \text{Arf}(K)$ is congruent modulo 8 to 2, 4 or 0. The corresponding lower bounds on $\gamma_4(K)$ are stated in Theorems \ref{TheoremWithLowerBoundOnGamma4WhenSigmaPlus4ArfEqual2}, \ref{TheoremWithLowerBoundOnGamma4WhenSigmaPlus4ArfEqual4} and \ref{TheoremWithLowerBoundOnGamma4WhenSigmaPlus4ArfEqual0} respectively. 

To begin, let $K$ be a knot such that $\sigma (K) +4\cdot \text{Arf}(K) \equiv 2\pmod{8}$ and suppose that $K$ bounds a smoothly and properly embedded M\"obius band $\Sigma$ in $D^4$. Let $W=W(\Sigma)$ be the twofold cover of $D^4$ with branching set $\Sigma$, and note that  Gilmer--Livingston's congruence relation implies that $W$ is positive definite. Corollary \ref{AboutSquareInCaseOfMobiusBand} dictates that the square of the sole generator of $H_2(W(\Sigma))/Tor$ is $\ell$ for some $\ell \in \mathbb N$, such that $\ell$ is a divisor of $\det K$ and with $\det K/\ell$ a square. We capture this statement by writing $Q_{W(\Sigma)}=[\ell]$.

Assume additionally that $K$ has a checkerboard coloring whose associated Goeritz matrix $G$ is negative definite, and let $W(F)$ be the 4-manifold as in Theorem \ref{GordonLitherlandTheorem}. We can then create the smooth, closed, oriented 4-manifold $X$ by gluing $W(\Sigma)$ to $W(F)$ along their boundaries:
\[
X=W(F)\cup _{Y(K)}(-W(\Sigma)).
\]
Then $X$ is negative definite and so by Theorem \ref{TheoremA} its intersection form $Q_X$ must be diagonalizable over $\mathbb Z$. The direct sum $Q_{W(F)}\oplus Q_{W(\Sigma)} = Q_{W(F)}\oplus [-\ell]$ of the intersection forms of $W(F)$ and $W(\Sigma)$ clearly embeds into $Q_X$, a condition that can be explicitly checked for a concrete knot $K$. Conversely, if   $Q_{W(F)}\oplus Q_{W(\Sigma)}$ does not embed into a diagonal form of equal rank, then $K$ cannot bound a M\"obius band in $D^4$ and we conclude that $\gamma_4(K) \ge 2$. We summarize this conclusion in the next theorem where we use the term {\em $1$-definite} as a synonym for positive definite, and similarly $-1$-definite as a substitute for negative definite. 
\begin{theorem}[Case of $\sigma (K) +4\cdot \text{Arf}(K) \equiv \pm 2\pmod{8}$]  \label{TheoremWithLowerBoundOnGamma4WhenSigmaPlus4ArfEqual2}
	Let $K$ be a knot with $\sigma(K) + 4\cdot \text{Arf}(K) \equiv 2\eps\pmod{8}$ for a choice of $\eps \in \{\pm1\}$. Assume that $K$ admits a checkerboard coloring for which the associated Goeritz form $G$ is $-\eps$-definite.
	
	If no embedding exists of $G\oplus [-\ell]$ into the $\eps$-definite diagonal form $(\mathbb Z^{\text{rank}(G)+1},\eps \text{Id})$ for any divisor $\ell\in \mathbb N$ of $\det K$ with $\det K/\ell$ a square, then $\gamma_4(K) \ge 2$.  
\end{theorem}

If $K$ is a knot with $\sigma (K) +4\cdot \text{Arf}(K) \equiv 4\pmod{8}$ then  Gilmer--Livingston's relation \eqref{YasuharasCongruence} reduces to 
\[
 \sigma (W(\Sigma )) - \beta (D^4, \Sigma ) \equiv 4\pmod{8},
 \]
for any non-orientable surface $\Sigma \subset D^4$ with $\partial \Sigma =K$, and implies that $\gamma_4(K) \ge 2$. If a non-orientable surface $\Sigma$ with $b_1(\Sigma) = 2$  existed, then the above relation would force $\sigma (W(\Sigma)) = \pm 2$ according to Corollary \ref{CorollaryOnTheBrownInvariantOfASurface}. Since $b_2(W(\Sigma)) =2$, we conclude that $W(\Sigma)$ is either positive or negative definite. If $K$ is an alternating knot so that both its Goeritz forms $G_{\pm}$ are definite, then one can again form a smooth, oriented, closed and definite 4-manifold $X$ as 
\begin{equation} \label{XInCaseOfSigmaPlus4ArfCongruent4Or0}
X = \begin{cases}
W(F_-) \cup_{Y(K)} (-W(\Sigma))\text{ ;} & \sigma (W(\Sigma)) >0, \\[5pt]
W(F_+) \cup_{Y(K)} (-W(\Sigma))\text{ ;} & \sigma (W(\Sigma)) <0.
\end{cases}
\end{equation}
The surfaces $F_{\pm}$ are the surfaces formed by the black regions of the checkerboard coloring used to create the Goertiz form $G_{\pm}$; see Theorem \ref{GordonLitherlandTheorem}. In either case, Donaldson's Theorem \ref{TheoremA} implies that the intersection form $Q_X$ of $X$ must be diagonalizable. The difficulty faced in this case is that it is generally hard to determine the intersection form $Q_{W(\Sigma)}$ of $W(\Sigma)$, given that $\Sigma$ is a hypothetical surface. Nevertheless, if $\Sigma$ existed, we would still obtain an embedding of $Q_{W(F_{\pm})}$ into a diagonal definite form of rank two larger, since $b_2(X) = b_2(W(F_{\pm}))+2$. We summarize this in the next theorem.  
\begin{theorem}[Case of $\sigma (K) +4\cdot \text{Arf}(K) \equiv 4\pmod{8}$]       \label{TheoremWithLowerBoundOnGamma4WhenSigmaPlus4ArfEqual4}
	Let $K$ be a knot with $\sigma(K) + 4\cdot \text{Arf}(K) \equiv 4\pmod{8}$ and assume that the Goeritz matrices $G_{\pm}$ of $K$, associated to the two possible checkerboard colorings of a knot projection $D$ of $K$, are positive and negative definite respectively (with the subscript $\pm$ indicating the definiteness type of $G_{\pm}$). 
	
	If no embedding exists of $G_{+}$ into the positive-definite  form $(\mathbb Z^{\text{rank}(G_{+})+2}, \text{Id})$, and no embedding exists of $G_{-}$ into the negative-definite form $(\mathbb Z^{\text{rank}(G_{-})+2}, -\text{Id})$, then $\gamma_4(K) \ge 3$. 
\end{theorem}

Lastly, we turn to the case of a knot $K$ with  $\sigma (K) +4\cdot \text{Arf}(K) \equiv 0\pmod{8}$. For such knots  Gilmer--Livingston's relation offers no information about $\gamma_4(K)$. Suppose that $K$ is alternating so that its Goeritz forms $G_{\pm}$ are definite. If $K$ bounded a M\"obius band $\Sigma\subset D^4$, then $W(\Sigma)$ is either positive or negative definite, with intersection form $[\ell]$ for some non-zero integer $\ell$ dividing $\det K$ and with $\det K/|\ell|$ a square, according to Corollary \ref{AboutSquareInCaseOfMobiusBand}. We then form again the definite, smooth, compact 4-manifold $X$ as in \eqref{XInCaseOfSigmaPlus4ArfCongruent4Or0}. By Donaldson's Theorem, $X$ must have a diagonalizable intersection form $Q_X$. 
\begin{theorem}[Case of $\sigma (K) +4\cdot \text{Arf}(K) \equiv 0\pmod{8}$]       \label{TheoremWithLowerBoundOnGamma4WhenSigmaPlus4ArfEqual0}
	Let $K$ be a knot with $\sigma(K) + 4\cdot \text{Arf}(K) \equiv 0\pmod{8}$ and assume that the Goeritz matrices $G_{\pm}$ of $K$, associated to the two possible checkerboard colorings of a knot projection $D$ of $K$, are positive and negative definite respectively (with the subscript $\pm$ indicating the definiteness type of $G_{\pm}$). 
	
		If no embedding exists of $G_+\oplus [\ell]$ into the positive-definite form $(\mathbb Z^{\text{rank}(G_+)+1},\text{Id})$, and no embedding exists of $G_-\oplus [-\ell]$ into the negative-definite form $(\mathbb Z^{\text{rank}(G_-)+1},-\text{Id})$, for any divisor $\ell\in \mathbb N$ of $\det K$ with $\det K/\ell$ a square, then $\gamma_4(K) \ge 2$.  
\end{theorem}
This last theorem is stated only for completeness and possible future applications. By happenstance, all of the knots $K$ with 8 or 9 crossings with $\sigma(K) + 4\cdot \text{Arf}(K) \equiv 0\pmod{8}$ admit non-oriented band moves to slice knots, and accordingly all such knots have $\gamma_4$ equal to 1, see Section \ref{SectionOnNullYasuharaKnots}.  
\section{Computations of $\gamma_4$}  \label{SectionWhereWeCalculateStuff}
This sections computes the value $\gamma_4(K)$ for all knots $K$ with crossing number equal to 8 or 9, thereby proving Theorems \ref{8CrossingKnots} and \ref{9CrossingKnots}. The computations are organized into 4 subsections: Section \ref{SectionOnSliceAndConcordantKnots} considers those 8- and 9-crossing knots that are either slice (and hence have $\gamma_4$ equal to 1), are concordant to a knot with a known value of $\gamma_4$, or admit a single non-orientable band move resulting in a slice knot. Sections \ref{SectionOnYasuharaKnots}, \ref{SectionOnHalfYasuharaKnots} and \ref{SectionOnNullYasuharaKnots} consider knots $K$ with $\sigma (K) +4\cdot \text{Arf}(K)$ congruent to 4, 2 and 0 modulo 8 respectively, and rely on Theorems \ref{TheoremWithLowerBoundOnGamma4WhenSigmaPlus4ArfEqual2} and  \ref{TheoremWithLowerBoundOnGamma4WhenSigmaPlus4ArfEqual4} to work out the their value of $\gamma_4$. We note that while these theorems only apply to alternating knots, by happenstance most of the non-alternating knots \eqref{TheNonAlternatinKnots} with 8 or 9 crossings, are already addressed in Section \ref{SectionOnSliceAndConcordantKnots} as they are all either slice (in the case of $8_{20}$ and $9_{46}$), or admit a single non-oriented band move to a slice knot (in the case of $8_{19}$, $9_{42}$, $9_{43}$, $9_{44}$, $9_{45}$, $9_{47}$ and $9_{48}$). The only remaining non-alternating knots are $8_{21}$ and $9_{49}$, which are addressed in Section \ref{SectionOnHalfYasuharaKnots} and Section \ref{SectionOnYasuharaKnots}, respectively, by relying on Proposition \ref{NonOrientableBandMoveLemma}, which in turn does not use the assumption of the knot being alternating.  

\subsection{Slice knots, concordant knots, and band moves to slice knots}
\label{SectionOnSliceAndConcordantKnots}

Among knots with crossing number 8 or 9, the smoothly slice knots are precisely \cite{Knotinfo} 
\begin{equation} \label{BatchOneSliceKnots} 
8_8, \, 8_9,\, 8_{20} \quad \text{ and } \quad 9_{27}, \, 9_{41}, \, 9_{46}. 
\end{equation}
For each of these 6 knots, $\gamma_4$ equals 1. 

Additionally, there is a smooth concordance \cite{Conway, CollinsKirkLivingston} between the knots $8_{10}$ and $-3_1$ (with $-K$ referring to the reverse mirror of  $K$).
Since $\gamma_4(3_1)=1$, and $\gamma_4$ is a invariant of smooth concordance, we obtain
\begin{equation} \label{BatchTwoConcordances}
\gamma_4(8_{10}) = 1. 
\end{equation}
%
Among knots with 8 or 9 crossings, the 38 knots 
\begin{gather} \label{BatchThreeBoundingMobiusBands}
8_3,\, 8_4,\, 8_5,\, 8_6,\, 8_7, \, 8_{11},\, 8_{14},\, 8_{16},\, 8_{19}, \cr \text{ and } \cr
9_1,\, 9_3,\, 9_4,\,  9_5,\, 9_6,\, 9_7,\, 9_8, \, 9_9,\, 9_{13},\,9_{15},\,9_{17},\,9_{19},\, 9_{21},\, 9_{22},\, 9_{23}, \cr
9_{25}, \, 9_{26},\, 9_{28},\, 9_{29},\, 9_{31},\, 9_{32},\, 9_{35}, \, 9_{36},\, 9_{42}, \, 9_{43},\, 9_{44},\, 9_{45},\, 9_{47},\, 9_{48},
\end{gather}
bound smoothly and properly embedded M\"obius bands in $D^4$. This is seen in 
Subfigures~\ref{FigureFor8_3} -- \ref{FigureFor8_19} of Figure~\ref{Figure1ToSliceKnots}, 
Subfigures~\ref{FigureFor9_1} -- \ref{FigureFor9_15} of Figure~\ref{Figure2ToSliceKnots}, 
Subfigures~\ref{FigureFor9_17} -- \ref{FigureFor9_29} of Figure~\ref{Figure3ToSliceKnots}, and in  
Subfigures~\ref{FigureFor9_31} -- \ref{FigureFor9_48} of Figure~\ref{Figure4ToSliceKnots}
%
where we exhibit band moves from each of the knots in \eqref{BatchThreeBoundingMobiusBands} to a slice knot. The claim about bounding M\"obius bands then follows from Proposition \ref{NonOrientableBandMoveLemma}. The only knot from the list \eqref{BatchThreeBoundingMobiusBands} not found in the above mentioned figures, is the knot $K=9_4$ which was addressed in Figure \ref{HandleLabelingConvention}. 
\subsection{Knots with $\sigma (K) + 4\cdot \text{Arf}(K) \equiv 4\pmod{8}$} \label{SectionOnYasuharaKnots}
Among knots $K$ with crossing number 8 or 9, those that satisfy the congruence $\sigma (K) + 4\cdot \text{Arf}(K) \equiv 4\pmod{8}$ are precisely the 19 knots 
\begin{align} \label{ListOfTheYasuharaKnots}
8_1,\, 8_2,\, 8_{12}, \, 8_{13},\, 8_{15},\, 8_{17},\, 8_{18},\, \quad \text{ and } \cr
9_{10},\, 9_{11}, \, 9_{14}, \, 9_{18}, \, 9_{20}, \, 9_{24}, \, 9_{30}, \, 9_{33}, \, 9_{34}, \, 9_{37},\, 9_{38},\, 9_{49}.   
\end{align}
 Gilmer--Livingston's relation \eqref{YasuharasCongruence} implies that $\gamma_4(K) \ge 2$ for any such knot. We verify that $\gamma_4(K) = 2$ for all knots $K$ in \eqref{ListOfTheYasuharaKnots} with the exception of $K=8_{18}$ and $K=9_{49}$, by constructing an explicit non-orientable smooth and properly embedded surface $\Sigma$ in $D^4$ with $b_1(\Sigma) = 2$ and $\partial \Sigma =K$. The existence of such a surface follows from Proposition \ref{NonOrientableBandMoveLemma} by finding a non-orientable band move from $K$ to $K'$ for a  knot $K'$ with $\gamma_4(K')=1$. Such band moves  are described in Subfigures~\ref{FigureFor8_1}--\ref{FigureFor9_14} of Figure~\ref{8_1, 8_2, 8_12,8_13,8_15,8_17,9_10,9_11,9_14}, and in Subfigures~\ref{FigureFor9_18} -- \ref{FigureFor9_38} in Figure~\ref{9_18,9_20,9_24,9_30,9_34,9_37,9_38,9_49}. The two exceptional knots $8_{18}$ and $9_{49}$ are addressed below, where they are both shown to have $\gamma_4$ equal to 3. 

\begin{proposition} \label{PropositionAbout8_18}
$\gamma_4(8_{18}) = 3$.
\end{proposition}
\begin{proof}
Let $K=8_{18}$. Theorem \ref{TheoremWithLowerBoundOnGamma4WhenSigmaPlus4ArfEqual4} implies that $\gamma_4(K)\ge 3$ provided we can prove that neither of the two Goeritz matrices $G_{\pm}$ of $K$ embed into a diagonal form of equal definiteness (positive or negative) and of rank two larger.  

We start by considering $G_-$, the negative definite Goeritz matrix associated to the checkerboard coloring of $K=8_{18}$ as given in Figure~\ref{8_18}. In that figure, $G_-$ is the incidence matrix of the given graph, whereby each of the vertices $e_1,\dots, e_4$  (corresponding to generators of $\mathbb Z^4$) has square $-3$ (that is $G_-(e_i,e_i)=-3)$ and any pair of vertices sharing an edge pairs to 1, while vertices not sharing an edge pair to 0.  
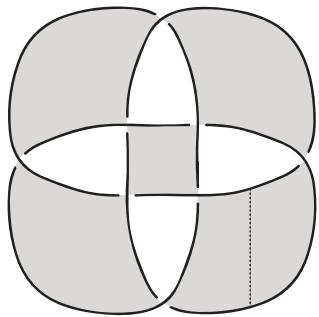
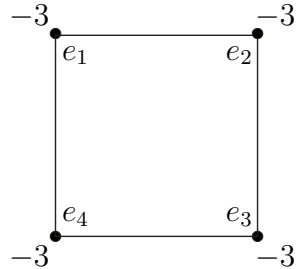
\begin{figure}[h]
	\centering
	\begin{subfigure}[b]{0.30\textwidth}
		\includegraphics[width=\textwidth]{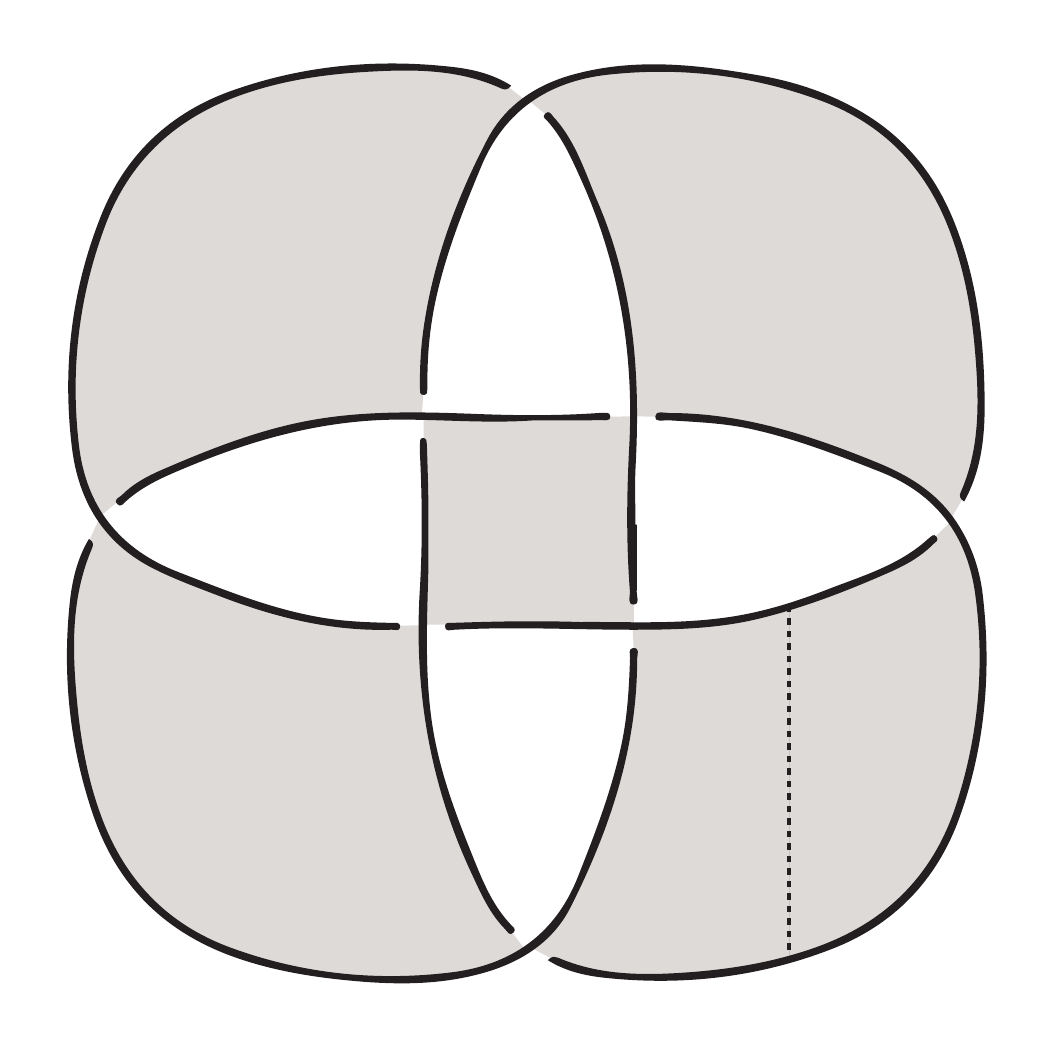}
		\caption{$8_{18}\stackrel{0}{\longrightarrow} 7_7$}
		\label{8_18Checkerboard}
	\end{subfigure}
	\qquad \qquad \qquad \qquad
	\begin{subfigure}[b]{0.3\textwidth}
		\includegraphics[width=\textwidth]{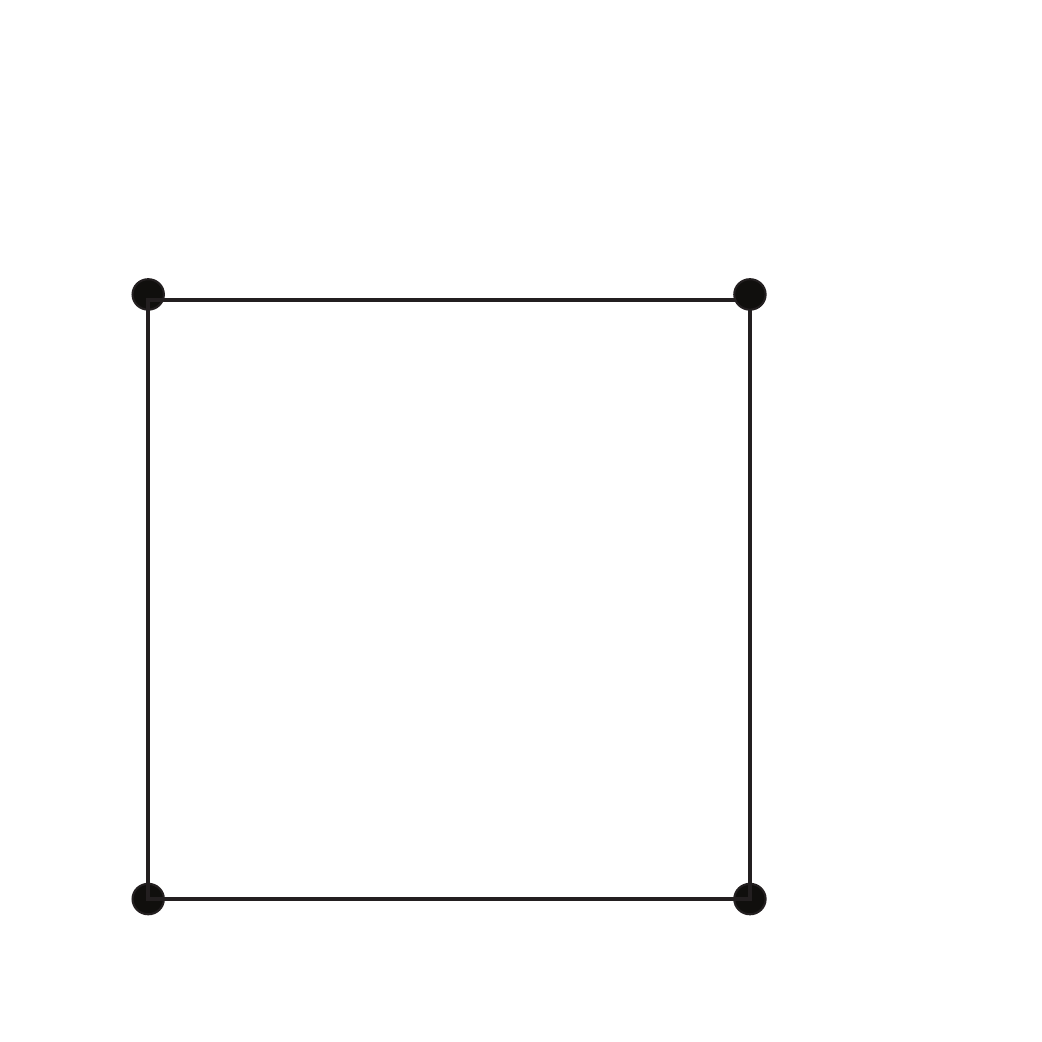}   
		\put(-111,85){$e_1$}
		\put(-131,98){$-3$}        
		\put(-38,98){$-3$}        
		\put(-49,85){$e_2$}
		\put(-38,7){$-3$}
		\put(-49,24){$e_3$}
		\put(-131,7){$-3$}
		\put(-111,24){$e_4$}
		\caption{The negative definite Goeritz form $G_-$ for $8_{18}$.}
		\label{8_18NegativeDefiniteIntersectionForm}
	\end{subfigure}
	\vskip3mm
	\caption{Case of $K=8_{18}$}\label{8_18}
\end{figure}

An embedding 
\[
\varphi:(\mathbb Z^4, G_-) \hookrightarrow (\mathbb Z^6, -\text{Id})
\]
is a monomorphism $\varphi:\mathbb Z^4 \to \mathbb Z^6$ such that $\varphi(a)\cdot \varphi(b) = G_-(a,b)$ for any pair $a,b\in \mathbb Z^4$, where the dot product refers to the product $-\text{Id}$ on $\mathbb Z^6$. Let $\{e_i\}_{i=1}^4$ be the basis for $(\mathbb Z^4, G_-)$ as described by Figure~\ref{8_18NegativeDefiniteIntersectionForm}, and let $\{f_i\}_{i=1}^6$ be the standard basis for $(\mathbb Z^6, -\text{Id})$, that is the basis with $f_i\cdot f_j = -\delta_{ij}$. In a further simplification of notation we shall also write $e_i\cdot e_j$ to mean $G_-(e_i,e_j)$, the nature of the vectors engaging in the dot product determines the particular dot product being used. 

If an embedding $\varphi$ existed, it would have to send $e_1$ (up to a change of basis of $(\mathbb Z^6, -\text{Id})$) to  
\[
\varphi(e_1) = f_1+f_2+f_3.
\]
Since $e_1 \cdot e_3=0$, then $\varphi(e_3)$ must share an even number of basis elements $\{f_i\}_{i=1}^6$ with the basis elements $f_1, f_2, f_3$ occurring in the formula for $\varphi(e_1)$. Thus that shared number is either 0 or 2. If it is 0, then $\varphi(e_3) = f_4+f_5+f_6$. However, since $e_1\cdot e_2=e_3\cdot e_2=1$, $\varphi(e_2)$ must share an odd number of basis elements $\{f_i\}_{i=1}^6$ with those occurring in each of the formulas for $\varphi(e_1)$ and $\varphi(e_3)$. That odd number cannot be 3 and thus must be 1, which is impossible as all six basis elements $\{f_i\}_{i=1}^6$ occur in $\varphi(e_1)$ and $\varphi(e_3)$. 

It follows that $\varphi(e_3)$ must have 2 basis elements in common with $\varphi(e_1)$, and so, again up to a change of basis of $(\mathbb Z^6, -\text{Id})$, it must be that  
\[
\varphi(e_3) = f_1-f_2+f_4.
\]
Suppose that $\varphi(e_2) = \sum_{i=1}^6 \lambda _i f_i$. Then the values of $e_2\cdot e_i$ for $i=1,2,3$ lead to these equations in the integer coefficients $\lambda _1,\dots, \lambda _6$: 
\begin{align*}
-\lambda _1-\lambda _2-\lambda _3 & = 1, \cr
-\lambda _1 +\lambda _2-\lambda _4 & = 1, \cr
\lambda _1^2+\dots + \lambda _6^2 & = 3.
\end{align*}
Writing $\lambda _3 = -\lambda _1-\lambda _2-1$ and $\lambda _4 = -\lambda _1+\lambda _2-1$ by using the first two equations, and plugging these into the third equation, yields 
\[
\lambda _1^2+ 3\lambda _2^2+2(\lambda _1+1)^2+\lambda _5^2 +\lambda _6^2 = 3.
\]
It follows that $\lambda _2=0$ and that either $\lambda _1=0$ or $\lambda _1=-1$, leading to two possibilities for $\varphi(e_2)$: 
\[
\varphi(e_2) = -f_3-f_4+f_5 \quad \text{ or } \quad \varphi(e_2) = -f_1+f_5+f_6.
\]

We suppose first that $\varphi(e_2) = -f_3-f_4+f_5$ and write $\varphi(e_4) = \sum _{i=1}^6 \mu_i f_i$ for some integers coefficients $\mu_1, \dots, \mu_6$ subject to the equations 
\begin{align*}
-\mu_1-\mu_2-\mu_3& = 1, \cr
\mu_3+\mu_4-\mu_5 & = 0, \cr
-\mu_1+\mu_2-\mu_4 & = 1, \cr
\mu_1^2+\dots +\mu_6^2 & = 3.
\end{align*} 
The first three of these equations lead to $\mu_3 = -\mu_1-\mu_2-1$, $\mu_4=-\mu_1+\mu_2-1$ and $\mu_5 = -2(\mu_1+1)$, which when plugged into the fourth equation yield 
\[
\mu_1^2+ 3\mu_2^2+6(\mu_1+1)^2+\mu_6^2 = 3.
\]
It follows immediately that $\mu_1=-1$ and $\mu_2=0$ and $\mu_6^2=2$, and the latter equation of course has no integral solution $\mu_6$. Therefore the choice of $\varphi(e_2) = -f_3-f_4+f_5$ does not lead to an embedding $\varphi$. 

Secondly, suppose that $\varphi(e_2) = -f_1+f_5+f_6$, the only remaining possibility for $\varphi(e_2)$, and write again $\varphi(e_4)=\sum_{i=1}^6 \eta_i f_i$ for integers $\eta_1, \dots, \eta_6$ this time subject to 
\begin{align*}
-\eta_1-\eta_2-\eta_3 & = 1, \cr
\eta_1 -\eta_5-\eta_6 & = 0, \cr
-\eta_1+\eta_2-\eta_4 & = 1, \cr
\eta_1^2+\dots +\eta_6^2 & = 3.
\end{align*}
The first three of these equations imply $\eta _3 = -\eta_1-\eta_2-1$, $\eta_4 = -\eta_1+\eta_2-1$ and $\eta_6 = \eta_1-\eta_5$, which when inserted into the fourth equation lead to 
\[
 \eta_1^2 + 3\eta_2^2+2(\eta_1+1)^2 + \eta_5^2+(\eta_1-\eta_5)^2 = 3.
 \]
We are forced to conclude that $\eta_2=0$ and that either $\eta_1=0$ or $\eta_1=-1$. The case of $\eta_1=0$ forces the equation $2\eta_5^2=1$, while the case of $\eta_1=-1$ leads to $\eta_5+(1+\eta_5)^2=2$, neither of which has integral solutions. 

We find that both possibilities for $\varphi(e_2)$ lead to equations for the coefficients of $\varphi(e_4)$ that have no integral solutions, and thus the embedding $\varphi : (\mathbb Z^4, G_-) \hookrightarrow (\mathbb Z^6, -\text{Id})$ cannot exist. 

It is easy to see that $G_+ = - G_-$ and so an embedding $(\mathbb Z^4, G_+) \hookrightarrow (\mathbb Z^6, \text{Id})$ would lead to an embedding $(\mathbb Z^4, G_-) \hookrightarrow (\mathbb Z^6, -\text{Id})$ which was already shown not to exist. Therefore the conditions of Theorem \ref{TheoremWithLowerBoundOnGamma4WhenSigmaPlus4ArfEqual4} are met and it follows that $\gamma_4(8_{18})\ge 3$. The equality $\gamma_4(8_{18})=3$ follows from the non-oriented band move indicated in Figure~\ref{8_18Checkerboard} which transforms $8_{18}$ into the knot $7_7$, and the fact that $\gamma_4 (7_7)=2$, see \cite{Knotinfo}.  
\end{proof}

{\bf Case of $K=9_{49}$. } Subfigure~\ref{FigureFor9_49} gives a non-orientable band move from the knot $9_{49}$ to the knot $8_{21}$. We show in Section~\ref{SectionOnHalfYasuharaKnots} that $\gamma_4(8_{21})=2$, which implies that $\gamma_4(9_{49})\le 3$. To see that $\gamma_4(9_{49})$ cannot equal 2, we rely on Theorem 4 from \cite{GilmerLivingston}. To state the theorem, let $K$ be a knot such that $H_1(M(K);\mathbb Z) \cong \mathbb Z_p \oplus \mathbb Z_p$ for some prime $p$. Then, by said theorem, if $K$ bounds a puncture Klein bottle in the 4-ball, the discriminant of the linking pairing of $M(K)$ is $\pm 1 \in \mathbb Z_p^\ast /(\mathbb Z_p^\ast)^2$. For the choice of $K=9_{49}$ one obtains $H_1(M(K);\mathbb Z) \cong \mathbb Z_5 \oplus \mathbb Z_5$ and that the discriminant of its liking pairing is $\pm 2 \in \mathbb Z_5^\ast/(\mathbb Z_5^\ast)^2$, which is different from $\pm 1$ (seeing as $\pm 2$ is a nonsquare in $\mathbb Z_5$ while $\pm 1$ is a square). It follows that $\gamma_4(9_{49})>2$ and thus that $\gamma_4(9_{49}) = 3$, as claimed.
\subsection{Knots with $\sigma (K) + 4\cdot \text{Arf}(K) \equiv 2\pmod{8}$} \label{SectionOnHalfYasuharaKnots}
In this section we consider the 8- and 9-crossing knots $K$ that satisfy the congruence relation $\sigma (K) + 4 \cdot \text{Arf}(K) \equiv 2\pmod{8}$. These are precisely the 34 knots 
\begin{gather} \label{ListOfHalfYasuharaKnots}
-\underline{8_4},\, -\underline{8_6},\,  \underline{8_7},\, -\underline{8_{10}}, \,   \underline{8_{11}}, \, -\underline{8_{14}},\, -\underline{8_{16}},\, -\underline{8_{19}},\, -8_{21}, \,\cr \text{ and } \cr
-9_2,\, -\underline{9_3},\, -\underline{9_5},\, -\underline{9_6},\, -\underline{9_8},\,  \underline{9_9},\, 9_{12},\, \underline{9_{15}},\, 9_{16},\, -\underline{9_{17}},\, -\underline{9_{21}},\,   \underline{9_{22}},\, -\underline{9_{25}},\, \cr
\underline{9_{26}}, \,\underline{9_{28}}, \, -\underline{9_{29}},\, -\underline{9_{31}},\, -\underline{9_{32}},\, \underline{9_{35}},\, 9_{39}, \,9_{40},\, \underline{9_{42}},\, \underline{9_{45}},\, \underline{9_{47}}, \,  -\underline{9_{48}}. 
\end{gather}
The 8- and 9-crossing knots $K$ that satisfy the opposite congruence relation $\sigma (K) + 4 \cdot \text{Arf}(K) \equiv -2\pmod{8}$ are the mirror knots of those appearing in \eqref{ListOfHalfYasuharaKnots}. The 29 underlined knots (or their mirror knots) in this list have been shown in Section \ref{SectionOnSliceAndConcordantKnots} to have $\gamma_4$ equal to 1, leaving us only to deal with the remaining 6 knots. 

{\bf Case of $K=8_{21}$. } To compute $\gamma_4(8_{21})$ one can use Corollary 3 from \cite{GilmerLivingston}. Let $K$ be a knot with $\det K=n$ and with $n$ a product of primes all with odd exponent. Suppose that $H_1(M(K);\mathbb Z) \cong \mathbb Z_n$ where $M(K)$ is the 2-fold cover of $S^3$ branched along $K$. Then by the aforementioned corollary, if $K$ bounds a M\"obius band in the 4-ball, the linking form 
\[
\ell k :H_1(M(K);\mathbb Z)\times H_1(M(K);\mathbb Z)\to \mathbb Q/\mathbb Z
\]
has the property that $\ell k(x,x) = \pm 1/n$ for some generator $x\in H_1(M(K);\mathbb Z)$.

For $K=8_{21}$ the linking form $\ell k:\mathbb Z_{15}\times \mathbb Z_{15}\to \mathbb Q/\mathbb Z$ is, up to isomorphism, given by multiplication by $13/15$. It is easy to check that $\pm 1/15$ does not occur as an output value of $\ell k(x,x)$ for any generator $x\in \mathbb Z_{15}$, proving that $2\le \gamma_4(8_{21})$. To see that $\gamma_4(8_{21}) \le 2$  it suffices to exhibit a non-orientable band move that changes $K=8_{21}$ into a knot $K'$ with $\gamma_4(K') = 1$. This is accomplished in Figure \ref{8_21} with $K'=5_2$. 
\begin{figure}[h]
\centering
\includegraphics[width=5cm]{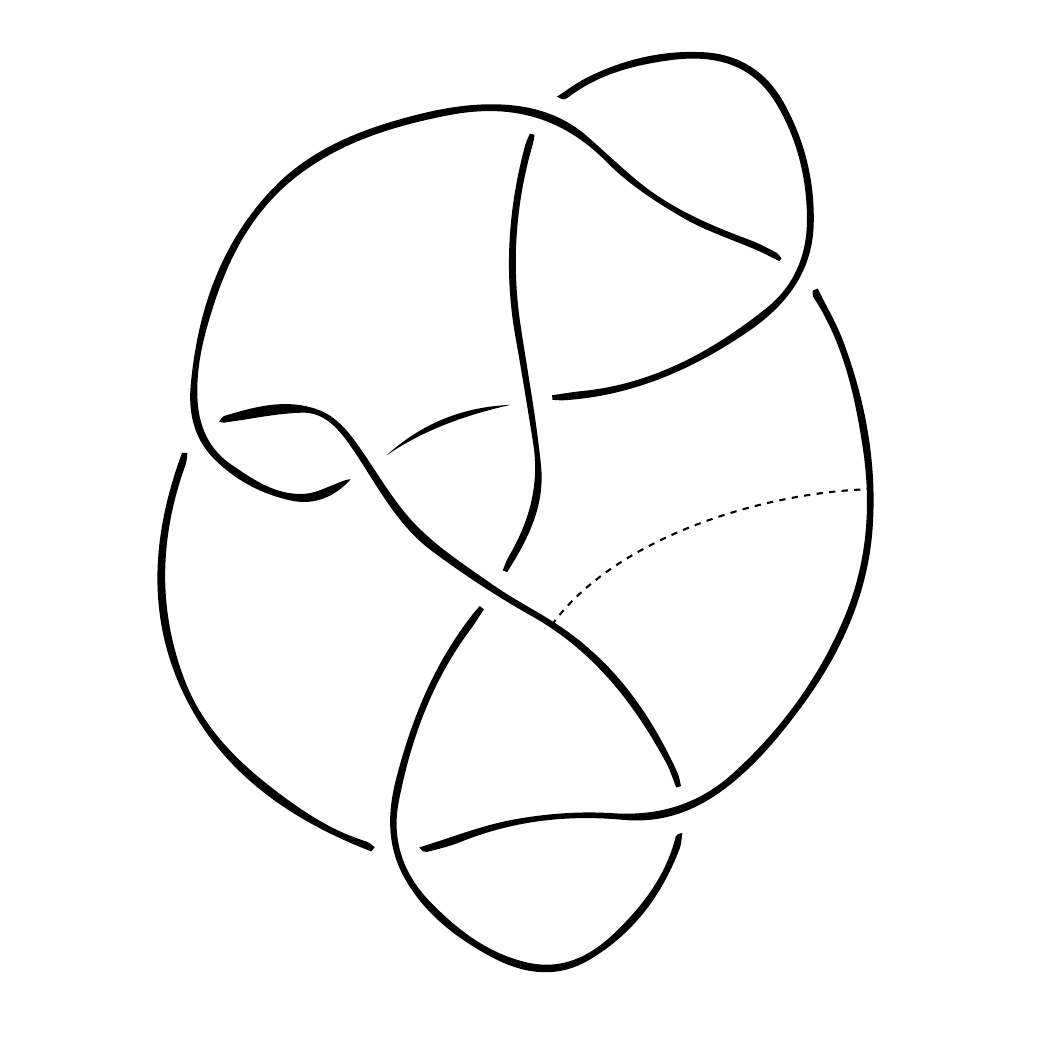}
\vskip2mm
\caption{$8_{21}\stackrel{0}{\longrightarrow} 5_2$.} 
\label{8_21}
\end{figure}

We will show that each of the remaining 5 non-underlined knot $K$ from \eqref{ListOfHalfYasuharaKnots} meets the assumptions of Theorem \ref{TheoremWithLowerBoundOnGamma4WhenSigmaPlus4ArfEqual2}, thereby proving that $\gamma_4(K)\ge 2$. We will  then show that $\gamma_4(K)=2$ by finding a non-oriented band move from $K$ to a knot $K'$ with $\gamma_4(K')=1$. 

{\bf Notational convention } We will represent the Goeritz forms $(\mathbb Z^n, G)$ of the various non-underlined knots from \eqref{ListOfHalfYasuharaKnots} as incidience matrices of weighted graphs. Recall that in such a presentation the generators $e_1,\dots, e_n$ of $\mathbb Z^n$ correspond to the $n$ vertices of the weighted graph, which in turn correspond to the white regions in the checkerboard coloring of the diagram of $K$. Moreover, $G(e_i,e_i)$ is given by the weight of the vertex $e_i$, $i=1,\dots, n$, and if $n_{i,j}$ is the number of edges between the vertices $e_i$ and $e_j$, then $G(e_i, e_j) = n_{ij}$. For simplicity of notation we shall write $e_i\cdot e_j$ to mean $G(e_i,e_j)$. This weighted graph approach to describing $(\mathbb Z^n, G)$ is merely a graphical tool that encodes the form \eqref{PreGoeritzMatrixCoefficients}.     

An embedding of 
\begin{equation} \label{TheGoertizEmbeddingForSigmaPlus4ArfEqual2}
\varphi:(\mathbb Z^{n+1},G\oplus [-d]) \hookrightarrow (\mathbb Z^{n+1},-\text{Id})
\end{equation}
is a monomorphism $\varphi:\mathbb Z^{n+1} \to \mathbb Z^{n+1}$ with the property that $-\text{Id}(\varphi(e_i),  \varphi(e_j)) = e_i\cdot e_j$ for $i,j = 1, \dots, n+1$. Here $e_{n+1}$ is the basis element of $\mathbb Z^n\oplus \mathbb Z$ corresponding to the last summand, and it has the properties: 
\[
e_{n+1}\cdot e_{n+1} = -d \quad \text{ and } \quad e_{n+1}\cdot e_i=0, \quad \text{for } i=1,\dots, n.
\]
Recall that $d\in \mathbb N$ is a divisor of $\det K$ with $\det K/d$ a square; see Corollary \ref{AboutSquareInCaseOfMobiusBand}. With the exception of $K=9_{40}$, all non-underlined knots $K$ from \eqref{ListOfHalfYasuharaKnots} have square-free determinant, forcing $d=\det K$. 

We let $\{f_i\}_{i=1}^{n+1}$ be the standard basis for $(\mathbb Z^{n+1},-\text{Id})$, that is the basis with $-\text{Id}(f_i,f_j) = -\delta_{ij}$, and shall also write $f_i\cdot f_j$ to mean $-\text{Id}(f_i,f_j)$. While the dot notation is used for both $G$ and $-\text{Id}$, the nature of the vectors involved in the dot product makes clear which form is meant. 

Our approach to showing that the embedding $\varphi$ does not exist is to take advantage of the ``rigidities'' presented by vertices $e_i$ with square $-2$ or $-3$. Any such vertex $e_i$ under $\varphi$ maps to either $f_1-f_2$ or $f_1+f_2+f_3$, up to a change of basis of $(\mathbb Z^{n+1},-\text{Id})$. Moreover since each basis element $f_i$ has square $-1$, then a pair of vertices $e_i$ and $e_j$ with $e_i\cdot e_j=1$ must have the property $\varphi(e_i)$ and $\varphi(e_j)$ share an odd number of basis elements $\{f_i\}_{i=1}^{n+1}$, and similarly if  $e_i\cdot e_j=0$ then $\varphi(e_i)$ and $\varphi(e_j)$ must share an even number of basis elements $\{f_i\}_{i=1}^{n+1}$. These requirements are restrictive enough to show that $\varphi$ cannot exist for the Goertiz forms of the non-underlined knots in \eqref{ListOfHalfYasuharaKnots}. 

The non-existence of the embedding \eqref{TheGoertizEmbeddingForSigmaPlus4ArfEqual2} shows that $\gamma_4(K)\ge 2$ for the corresponding knot $K$. The equality $\gamma_4(K) =2$ is derived by finding a non-oriented band move from $K$ to a knot $K'$ with $\gamma_4(K') = 1$. 

\begin{remark}
We would like to emphasize that in each of the following arguments, our computations are valid up to a change of basis of $(\mathbb{Z}^{n+1}, -\text{Id})$ and we will usually take this fact for granted to simplify the exposition.
\end{remark}

\vskip3mm
\begin{figure}[h]
	\centering
	\begin{subfigure}[b]{0.30\textwidth}
		\includegraphics[width=\textwidth]{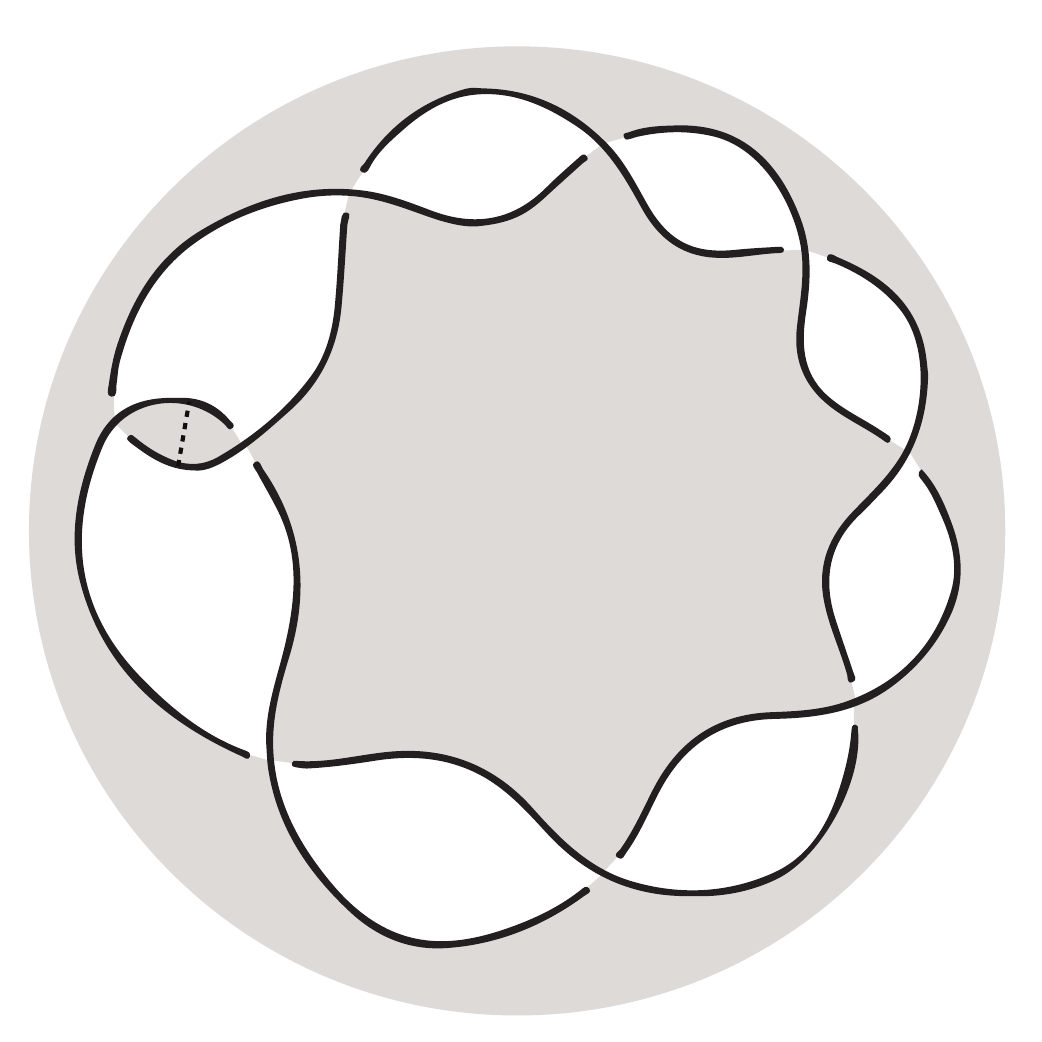}
		\caption{$-9_{2}\stackrel{0}{\longrightarrow} 7_1$}
		\label{9_2mirror}
	\end{subfigure}
	\qquad \qquad \qquad \qquad
	\begin{subfigure}[b]{0.3\textwidth}
		\includegraphics[width=\textwidth]{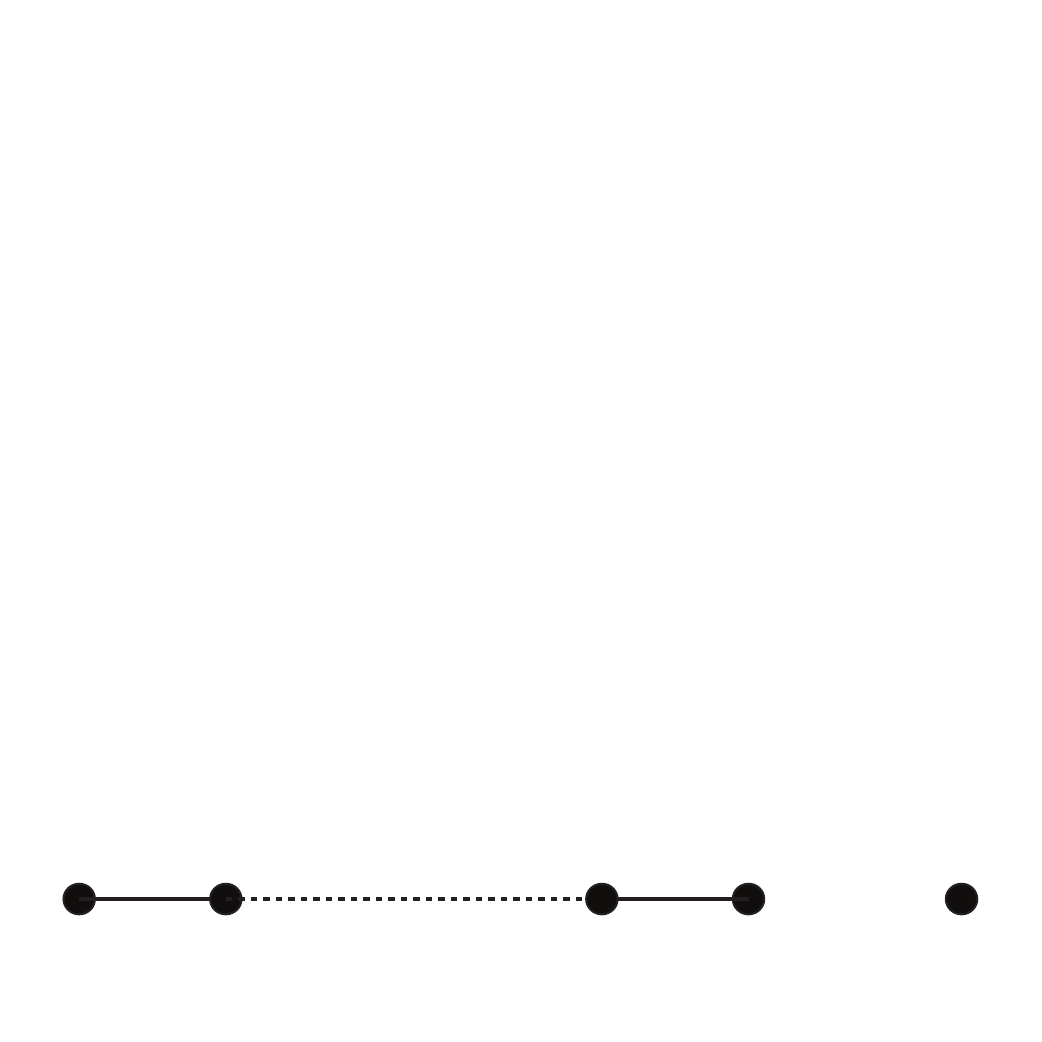}   
		\put(-136,26){$-2$}        
		\put(-128,6){$e_1$}
		\put(-116,26){$-2$}
		\put(-107,6){$e_2$}
		\put(-68,26){$-2$}
		\put(-59,6){$e_6$}
		\put(-48,26){$-3$}
		\put(-39,6){$e_7$}
		\put(-24,26){$-15$}
		\put(-13,6){$e_8$}
		\caption{Goeritz form for $-9_{2}$.}
		\label{9_2MirrorIntersectionform}
	\end{subfigure}
	\vskip3mm
	\caption{Case of $K=-9_{2}$}\label{-9_2}
\end{figure}
{\bf Case of $K=-9_{2}$. } The negative definite Goeritz matrix $G$ associated to the checkerboard coloring of the knot $K=-9_{2}$ from Figure~\ref{9_2mirror}, is given as the incidence matrix of the weighted graph in Figure~\ref{9_2MirrorIntersectionform}, where all the missing vertices, indicated by the dotted line, have weights $-2$. Since $\det 9_2= 15$ is square-free, we are seeking to obstruct the existence of an embedding 
\[
\varphi:(\mathbb Z^8, G\oplus [-15]) \hookrightarrow (\mathbb Z^8, -\text{Id}).
\]
If $\varphi$ existed, we would have to have
\[
\varphi(e_i) = f_i-f_{i+1}, \quad \text{ for } i =1,\dots, 6.
\]
Let $\varphi(e_7) = \sum _{i=1}^8 \mu_i f_i$, then since $e_7\cdot e_6=1$ and $e_7\cdot e_j=0$ for $j=1,\dots, 5$, it follows that 
\[
\mu_1=\mu_2=\mu_3=\mu_4=\mu_5=\mu_6 \quad \text{ and } \quad -\mu_6+\mu_7 = 1.
\]
Since at most 3 of the coefficients $\mu_j$ are nonzero, we conclude that $\mu_j=0$ for $j=1,\dots, 6$ and $\mu_7 = 1$. It follows that $\varphi(e_7) = f_7+\mu_8f_8$ forcing the relation $1+\mu_8^2=3$, which has no integral solution. Thus $\varphi$ cannot exist leading to $\gamma_4(-9_{2})\ge 2$. The equality $\gamma_4(-9_{2}) = 2$ follows from the non-orientable band move in Figure~\ref{9_2mirror} which transforms $-9_{2}$ to $7_1$, and given that $\gamma_4(7_1) = 1$. 
\vskip3mm
\begin{figure}[h]
	\centering
	\begin{subfigure}[b]{0.33\textwidth}
		\includegraphics[width=\textwidth]{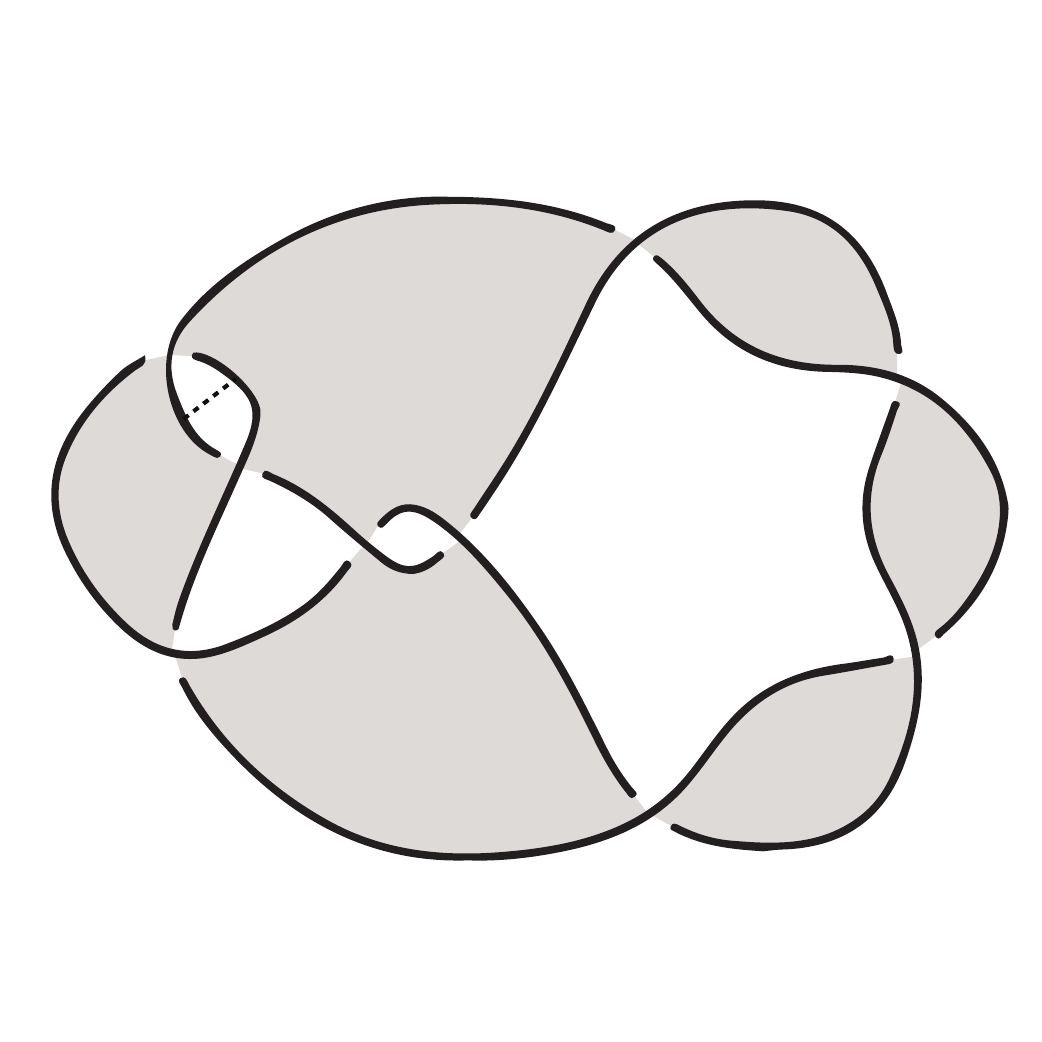}
		\caption{$9_{12}\stackrel{0}{\longrightarrow} 7_3$}
		\label{9_12Subfigure}
	\end{subfigure}
	\qquad \qquad \qquad \qquad
	\begin{subfigure}[b]{0.33\textwidth}
		\includegraphics[width=\textwidth]{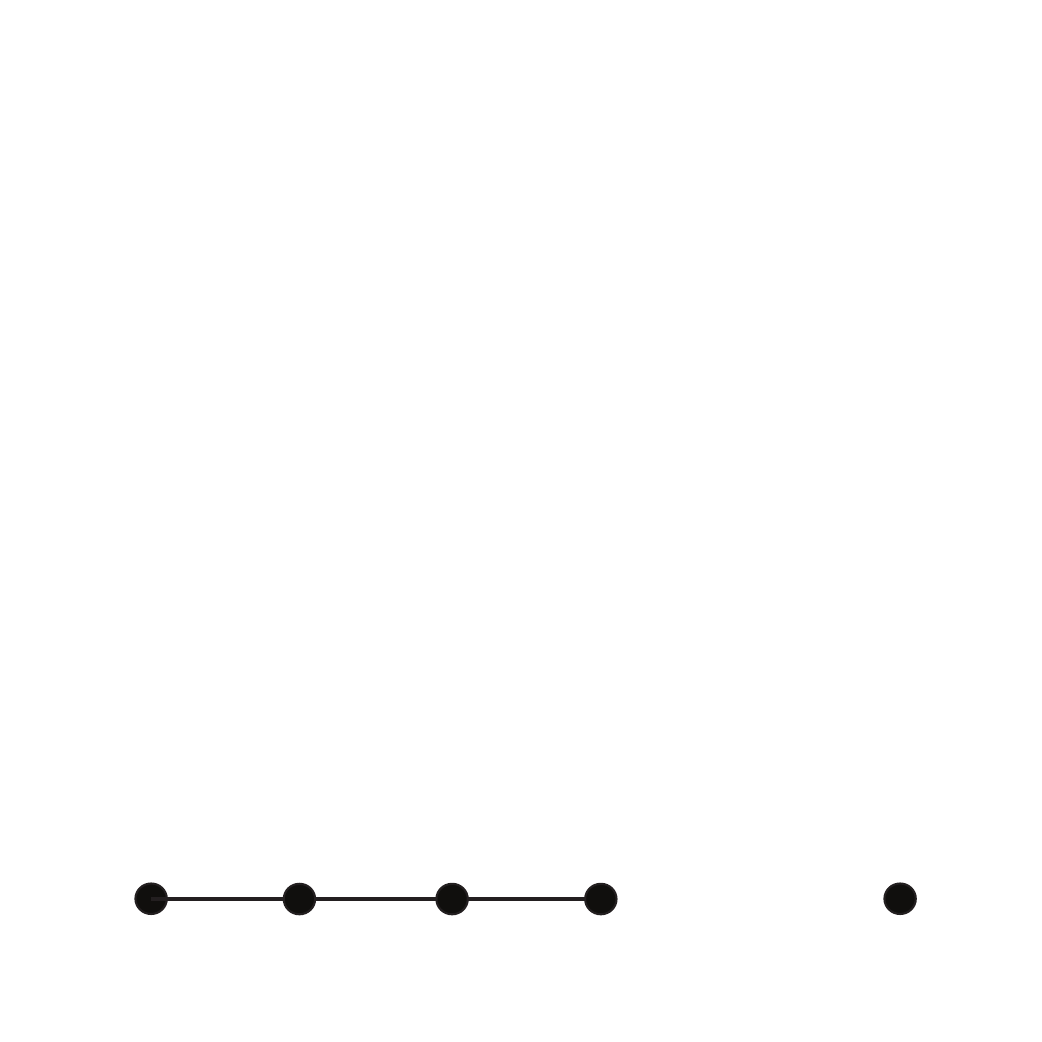}   
		\put(-127,6){$e_1$}
		\put(-137,26){$-2$}        
		\put(-107,6){$e_2$}
		\put(-117,26){$-3$}
		\put(-86,6){$e_3$}
		\put(-95,26){$-2$}
		\put(-68,6){$e_4$}
		\put(-74,26){$-5$}
		\put(-24,6){$e_5$}
		\put(-36,26){$-35$}
		\caption{Goeritz form for $9_{12}$.}
		\label{9_12Intersectionform}
	\end{subfigure}
	\vskip3mm
	\caption{Case of $K=9_{12}$.}\label{9_12}
\end{figure}
{\bf Case of $K=9_{12}$. } The negative definite Goeritz matrix $G$ associated to the checkerboard coloring of the knot $K=9_{12}$ from Figure~\ref{9_12Subfigure} is given by the incidence matrix in Figure~\ref{9_12Intersectionform}. Since $\det 9_{12} = 35$ is square-free, we seek to obstruct an embedding 
\[
\varphi:(\mathbb Z^5,G\oplus [-35]) \hookrightarrow (\mathbb Z^5, -\text{Id}).
\]
If $\varphi$ existed, we would have 
\[
\varphi(e_4) = f_1+f_2+f_3+f_4+f_5 \quad \text{ or } \quad \varphi(e_4) = f_1+2f_2.
\]
\begin{itemize}
\item[(a)] Case of $\varphi(e_4) = f_1+f_2+f_3+f_4+f_5.$ Write $\varphi(e_5) = \sum _{i=1}^5 \lambda _i f_i$ for integers $\lambda _1, \dots, \lambda _5$ to be determined. Since $e_5\cdot e_5=-35$ it follows that $35 = \sum _{i=1}^5 \lambda _i^2$, and $e_4\cdot e_5=0$ implies that $\lambda _1+\lambda _2+\lambda _3+\lambda _4+\lambda _5 = 0$. These two relations are in contradiction with one another because 
$$35 = \lambda _1^2+\dots + \lambda _5^2 \equiv (\lambda _1+\dots +\lambda _5)^2  \pmod{2} \equiv 0 \pmod{2}, $$ 
showing that the choice of $\varphi(e_4) = f_1+f_2+f_3+f_4+f_5 $ does not extend to an embedding $\varphi$. 

\item[(b)] Case of $\varphi(e_4) = f_1+2f_2$. Since $e_3\cdot e_4=1$ and $\varphi(e_3)$ is a sum of only two basis elements $\{f_i\}_{i=1}^5$, it must be that $\varphi(e_3) = -f_1+f_3$ or $\varphi(e_3) = f_1-f_2$. 

We first pursue the case of $\varphi(e_3) = -f_1+f_3$. Since $e_2\cdot e_3=1$, $\varphi(e_2)$ must share exactly one basis element with $\varphi(e_3)$. This shared element cannot be $f_1$ since this would force $\varphi(e_2) \cdot \varphi(e_4)\ne 0$, showing that $\varphi(e_3)$ and $\varphi(e_2)$ must share $f_3$. Note the $\varphi(e_2)$ cannot contain $f_2$ either since this would lead yet again to $\varphi(e_2) \cdot \varphi(e_4) \ne 0$. We are thus forced to conclude that $\varphi(e_2) = -f_3+f_4+f_5$. Write $\varphi(e_1) = \sum _{i=1}^5 \mu_i f_i$, then $\mu_1+2\mu_2=0$, $\mu_1=\mu_3$, $1=\mu_3-\mu_4-\mu_5$ and $\sum _{i=1}^5 \mu_i^2 = 2$. The three linear equations lead to $\mu_1= \mu_3 =  -2\mu_2$, $\mu_5 = -2\mu_2-\mu_4-1$ which when plugged into the quadratic equation gives 
\[
9\mu_2^2 +\mu_4^2 +(2\mu_2+\mu_4+1)^2 = 2.
\]
This forces $\mu_2=0$ and $\mu_4^2+(\mu_4+1)^2=2$, the latter of which has no solution $\mu_4\in \mathbb Z$. 

Next we turn to the only remaining possibility of $\varphi(e_3) = f_1-f_2$. Since $e_2\cdot e_3=1$, $\varphi(e_3)$ and $\varphi(e_2)$ must share exactly one basis element $f_i$, $i=1,\dots, 5$. Accordingly, we must have $\varphi(e_2) = \pm f_i \pm f_j$ with $i\in \{1, 2\}$ and $j\in \{3, 4, 5\}$. However each of these cases leads to $\varphi(e_2)\cdot \varphi(e_4) \ne 0$, contradicting $e_2\cdot e_4=0$. 
\end{itemize}

It follows that the embedding $\varphi$  cannot exist, implying that $\gamma_4(9_{12})\ge 2$. The non-oriented band move from $9_{12}$ to $7_3$ in Figure~\ref{9_12Subfigure} shows that $\gamma_4(9_{12}) = 2$, seeing as $\gamma_4(7_3) = 1$. 
\begin{figure}[h]
	\centering
	\begin{subfigure}[b]{0.34\textwidth}
		\includegraphics[width=\textwidth]{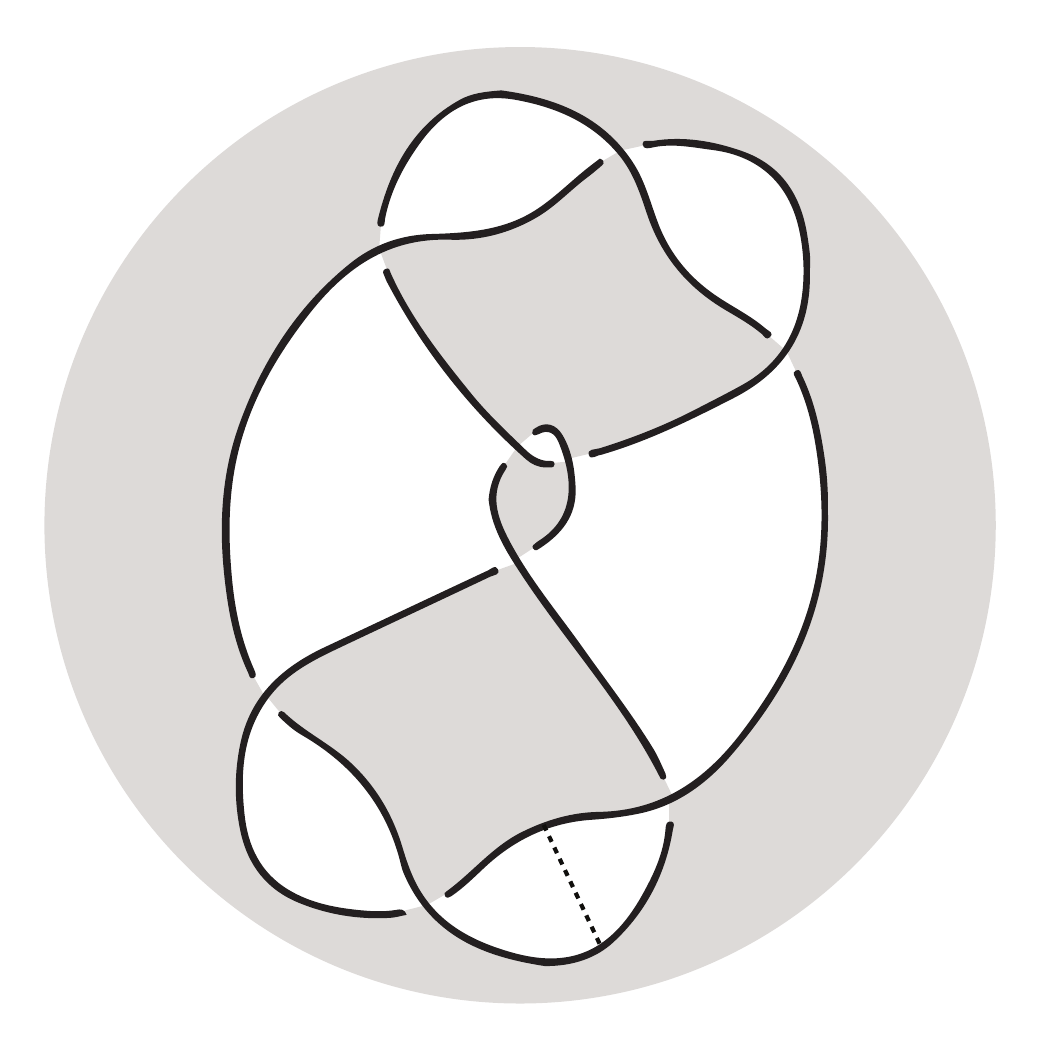}
		\caption{$9_{16}\stackrel{0}{\longrightarrow} 6_2$}
		\label{9_16Subfigure}
	\end{subfigure}
	\qquad \qquad \qquad 
	\begin{subfigure}[b]{0.38\textwidth}
		\includegraphics[width=\textwidth]{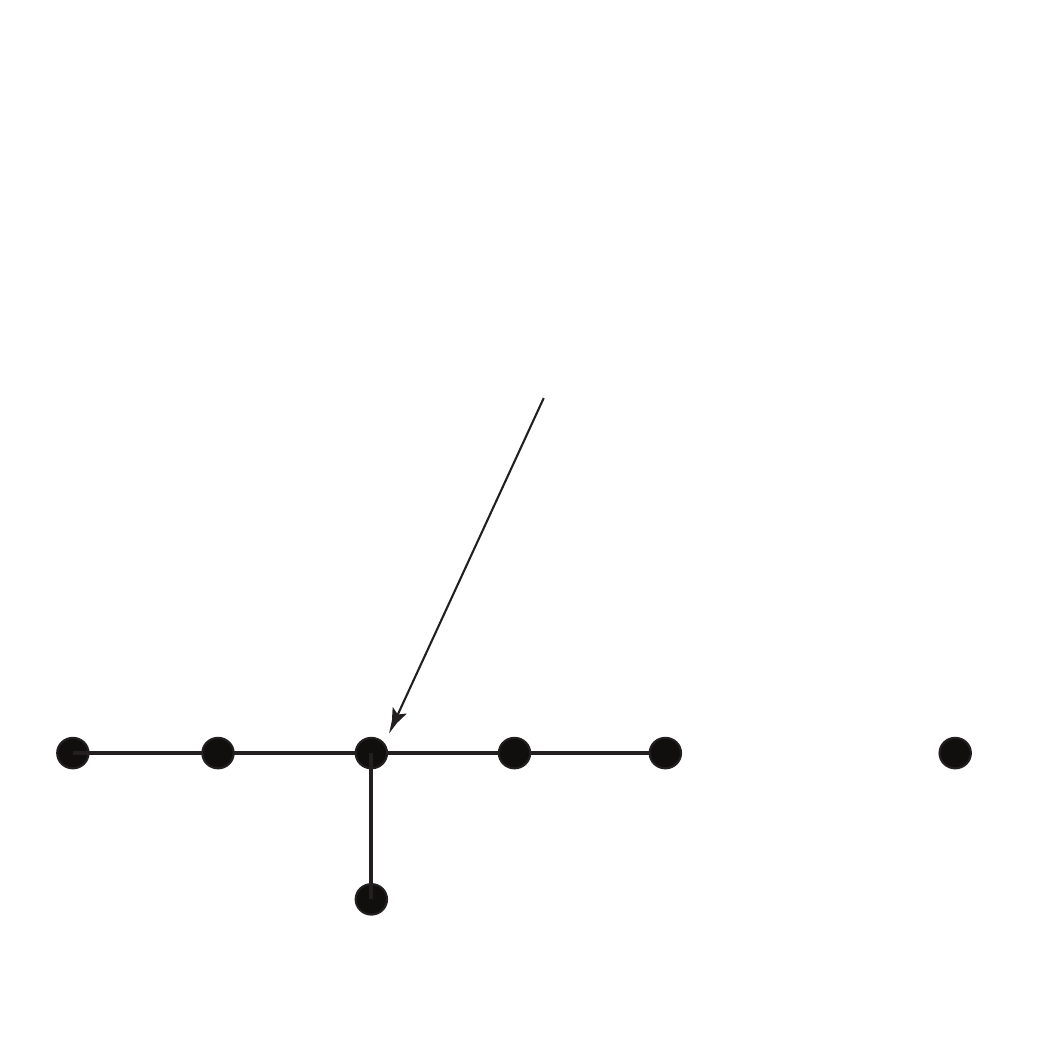}           
		\put(-158,34){$e_1$}
		\put(-168,56){$-2$}
		\put(-136,34){$e_2$}
		\put(-145,56){$-2$}
		\put(-77,105){$e_3$}
		\put(-121,56){$-4$}
		\put(-100,20){$e_4$}
		\put(-129,20){$-2$}
		\put(-88,34){$e_5$}
		\put(-98,56){$-2$}
		\put(-64,34){$e_6$}
		\put(-72,56){$-2$}
		\put(-18,34){$e_7$}
		\put(-29,56){$-39$}
		\caption{Goeritz form for $9_{16}$.}
		\label{9_16Intersectionform}
	\end{subfigure}
	\vskip3mm
	\caption{Case of $K=9_{16}$.}\label{9_16}
\end{figure}
\vskip3mm
{\bf Case of $K=9_{16}$. } The negative definite Goeritz matrix $G$ associated to the checkerboard coloring of the knot $K=9_{16}$ from Figure~\ref{9_16Subfigure} is given in Figure~\ref{9_16Intersectionform}. Since $\det 9_{16} = 39$ is square-free, we wish to obstruct the existence of an embedding 
\[
\varphi :(\mathbb Z^{7},G\oplus [-39]) \hookrightarrow (\mathbb Z^7, -\text{Id}).
\]
Any such $\varphi$ would have
\[
\varphi(e_3) = f_1 + f_2 +f_3 + f_4.
\]
Since $e_i\cdot e_4=1$ and $e_i\cdot e_i=-2$ for $i=2, 4, 5$, then each $\varphi(e_i)$ has exactly one basis element in common with $\varphi(e_3)$, and that common basis element is different for each $i=2, 4, 5$. Indeed if we had for instance $-f_1$ be common to $\varphi(e_2)$ and $\varphi(e_4)$ then we would be forced to have $\varphi(e_2) = -f_1+f_5$ and $\varphi(e_4) = -f_1-f_5$, which would make it impossible to satisfy the two relations $\varphi(e_1)\cdot \varphi(e_2) = 1$ and $\varphi(e_1)\cdot \varphi(e_4)=0$ simultaneously. A similar argument shows that neither of the other two pairs $\{\varphi(e_2), \varphi(e_5)\}$ and $\{\varphi(e_4), \varphi(e_5)\}$ can share the same basis element with $\varphi(e_3)$. Thus we conclude that 
\[
\varphi(e_2) = -f_1+f_5, \quad \varphi(e_4) = -f_2+f_6, \quad \varphi(e_5) = -f_3+f_7.
\]
Since $e_1\cdot e_2=1$, $\varphi(e_1)$ shares exactly one basis element with $\varphi(e_2)$. This shared element cannot be $f_5$ since the other basis element for $\varphi(e_1)$ would have to come from $\{f_2, f_3, f_4, f_6, f_7\}$, each choice of which would lead to $\varphi(e_1)\cdot \varphi(e_i)\ne 0$ for some $i\ne 1,2$. This leaves $\varphi(e_1) = f_1-f_4$ as the only possibility.  Lastly, $e_6\cdot e_5=1$ says that $\varphi(e_6)$  must contain one and only one of $f_3$ or $f_7$. However either choice for the other basis element in $\varphi(e_6)$ leads to one of $\varphi(e_6)\cdot \varphi(e_i)$, $i=1, 2, 3, 4$ being nonzero, a contradiction. Thus $\varphi$ cannot exist and so $\gamma_4(9_{16})\ge 2$, showing that $\gamma_4(9_{16})=2$ given the non-oriented band move from $9_{16}$ to $6_2$ in Figure~\ref{9_16Subfigure}, and seeing as $\gamma_4(6_2) = 1$. 

\begin{figure}[h]
	\centering
	\begin{subfigure}[b]{0.34\textwidth}
		\includegraphics[width=\textwidth]{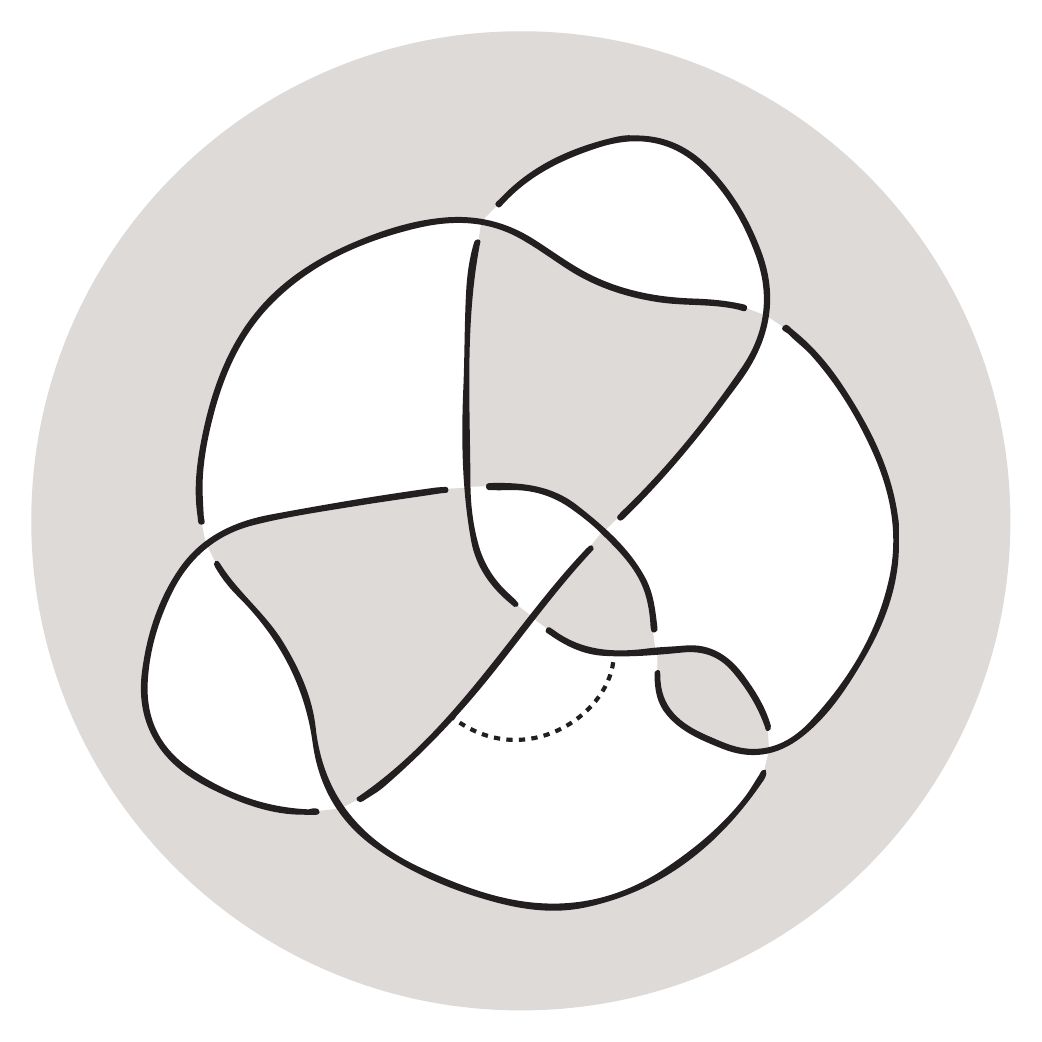}
		\caption{$9_{39}\stackrel{1}{\longrightarrow} 8_{11}$}
		\label{9_39Subfigure}
	\end{subfigure}
	\qquad \qquad \qquad
	\begin{subfigure}[b]{0.38\textwidth}
		\includegraphics[width=\textwidth]{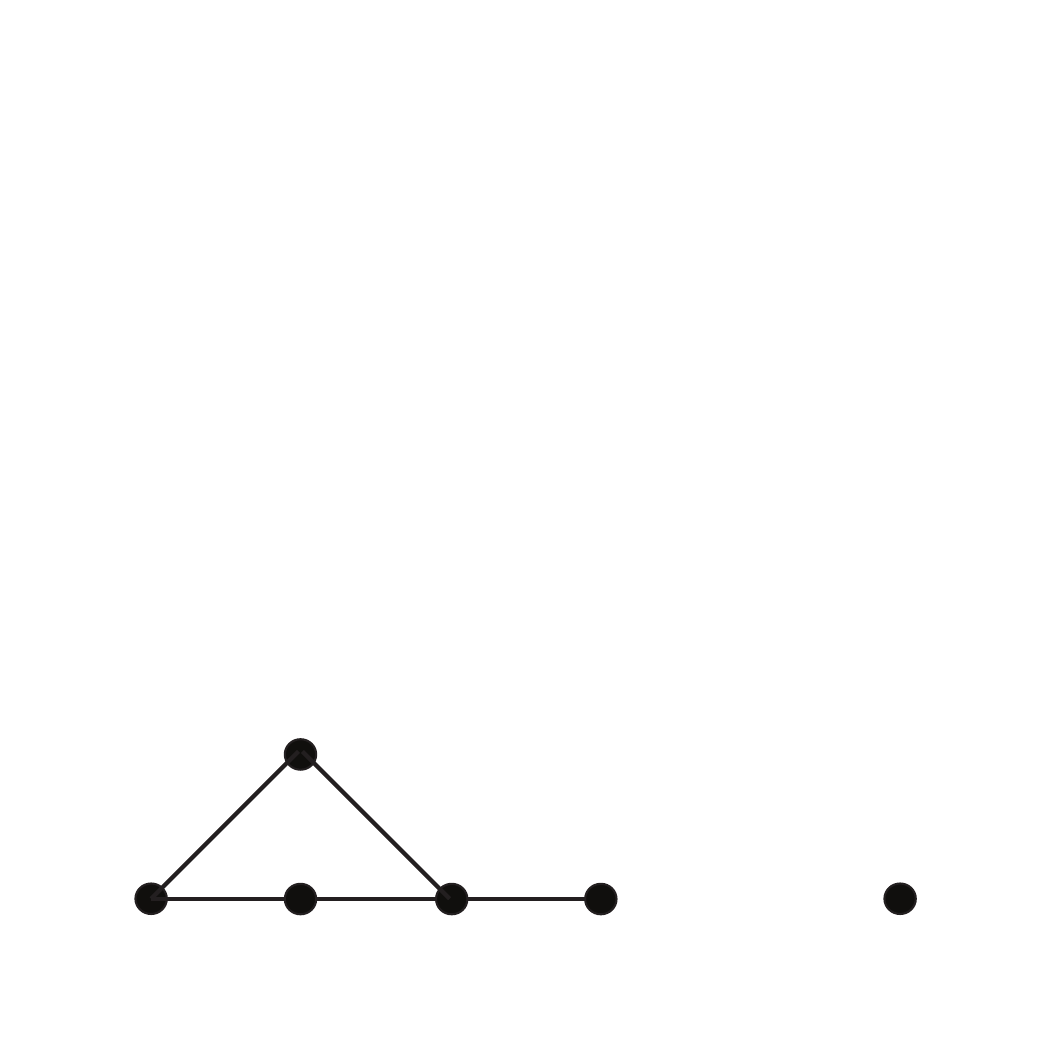}           
		\put(-147,8){$e_1$}
		\put(-155,30){$-4$}
		\put(-123,8){$e_2$}
		%
		\put(-99,8){$e_3$}
		\put(-103,30){$-3$}
		\put(-125,56){$e_4$}
		%
		\put(-75,8){$e_5$}
		\put(-83,30){$-2$}
		\put(-26,8){$e_6$}
		\put(-38,30){$-55$}
		\caption{Goeritz form for $9_{39}$. Here $e_2$ has square $-2$ and $e_4$ has square $-3$.}
		\label{9_22Intersectionform}
	\end{subfigure}
	\vskip3mm
	\caption{Case of $K=9_{39}$.}\label{9_39}
\end{figure}
\vskip3mm
{\bf Case of $K=9_{39}$. } The negative definite Goeritz matrix $G$ associated to the checkerboard coloring of the knot $K=9_{39}$ in Figure~\ref{9_39Subfigure} is the incidence matrix of the weighted graph in Figure~\ref{9_22Intersectionform}. Since $\det 9_{39} = 55$ is squre-free, we aim to show that no embedding 
\[
\varphi :(\mathbb Z^6, G\oplus[-55]) \hookrightarrow (\mathbb Z^6, -\text{Id})
\]
exists. Any such $\varphi$ would have 
\[
\varphi(e_1) = f_1 + f_2 +f_3+f_4 \quad \text{ and } \quad \varphi(e_2) = -f_4+f_5.
\]
Note that $\varphi(e_1) = 2f_1$ is not possible because $e_1\cdot e_2=1$. Since $e_2\cdot e_3=1$ and $e_1\cdot e_3=0$, then $\varphi(e_3)$ shares exactly one basis element with $\varphi(e_2)$ and an even number of basis elements with $\varphi(e_1)$. The shared element among $\varphi(e_3)$ and $\varphi(e_2)$ may be either $f_4$ or $f_5$ leading to the two possibilities $\varphi(e_3) = f_4 -f_1-f_6$ or  $\varphi(e_3) = -f_5-f_1+f_2$. 

\begin{itemize}
\item[(a)] Case of $\varphi(e_3) = f_4 -f_1-f_6$.

Since $e_3\cdot e_4=1=e_1\cdot e_4$ then $\varphi(e_4)$ shares an odd number of basis elements with each of $\varphi(e_1)$ and $\varphi(e_3)$. This shared number of basis elements between $\varphi(e_4)$ and $\varphi(e_3)$ cannot be 3 since if it were then we would obtain $\varphi(e_5)\cdot \varphi(e_4) \equiv \varphi(e_5)\cdot \varphi(e_3) \pmod{2}$ which is not a valid congruence. Thus $\varphi(e_4)$ shares one basis element with $\varphi(e_3)$. Note also that $\varphi(e_4)$ shares an even number of basis elements with $\varphi(e_2)$. 

\begin{itemize}
\item[(i)] Case of $\varphi(e_3)$ and $\varphi(e_4)$ sharing only $f_4$. In this case $f_5$ also appears as a summand in $\varphi(e_4)$ and we are led to $\varphi(e_4) = -f_4-f_5\pm f_i$ for some $i\in \{2, 3\}$. No matter which $i\in \{2,3\}$ we pick, we arrive at an even number of shared basis elements between $\varphi(e_4)$ and $\varphi(e_1)$, a contradiction. 

\item[(ii)] Case of $\varphi(e_3)$ and $\varphi(e_4)$ sharing only $f_1$. In this case $\varphi(e_4)$ cannot contain $f_4$ or $f_6$, and therefore cannot contain $f_5$ either since $e_2\cdot e_4=0$. Thus we are forced to conclude that in this case $\varphi(e_4) = f_1-f_2-f_3$. Moving on to $\varphi(e_5)$, the relation $e_5\cdot e_3=1$ shows that $\varphi(e_5)$ shares with $\varphi(e_3)$ exactly one of $f_1$, $f_4$ or $f_6$. 
\begin{itemize}
\item[($\alpha$)] Case of $\varphi(e_5)$ and $\varphi(e_3)$ sharing $f_1$. In this case we find that $\varphi(e_5) = f_1\pm f_i$ for some $i\in\{2,3,5\}$. The relation $e_5\cdot e_2=0$ shows that $i\ne 5$. Each of the possibilities $\varphi(e_5) = f_1\pm f_2$ or $\varphi(e_5) = f_1\pm f_3$ leads to one of $\varphi(e_5) \cdot \varphi(e_1)$ or $\varphi(e_5) \cdot \varphi(e_4)$ having the wrong value, a contradiction.  

\item[($\beta$)] Case of $\varphi(e_5)$ and $\varphi(e_3)$ sharing $f_4$. Here $\varphi(e_5) = -f_4\pm f_i$ for some $i\in\{2,3,5\}$. The relation $e_5\cdot e_2=0$ forces $i = 5$ and $\varphi(e_5) = -f_4-f_5$. However this leads to the incorrect value of 1 for $\varphi(e_5) \cdot \varphi(e_1)$, a contradiction. 

\item[($\gamma$)] Case of $\varphi(e_5)$ and $\varphi(e_3)$ sharing $f_6$. Here $\varphi(e_5) = f_6\pm f_i$ for some $i\in\{2,3,5\}$. The relation $e_5\cdot e_2=0$ forces $i \ne 5$, and each of the remaining possibilities $\varphi(e_5) = f_6\pm f_2$ and $\varphi(e_5) = f_6\pm f_3$ leads to the incorrect value of $\pm 1$ for $\varphi(e_5) \cdot \varphi(e_1)$, another contradiction.  
\end{itemize}

\item[(iii)] Case of $\varphi(e_3)$ and $\varphi(e_4)$ sharing only $f_6$. Since $e_4\cdot e_2=0$, then $\varphi(e_4)$ cannot contain $f_5$ either, leaving us with the possibility of $\varphi(e_4) = f_6\pm f_2\pm f_3$. No matter the choice of signs, this leads to an even number of shared basis elements between $\varphi(e_4)$ and $\varphi(e_1)$, contradicting $e_4\cdot e_1=1$. We conclude that the case of $\varphi(e_3) = f_4 -f_1-f_6$ does not lead to an embedding $\varphi$ .
\end{itemize}

\item[(b)] Case of $\varphi(e_3) = -f_5-f_1+f_2$. 

As in the previous case, we find that $\varphi(e_4)$ shares one basis element with $\varphi(e_3)$, shares an even number of basis elements with $\varphi(e_2)$, and an odd number with $\varphi(e_1)$. 

\begin{itemize}
\item[(i)] Case of $\varphi(e_3)$ and $\varphi(e_4)$ only sharing $f_1$. In this case $\varphi(e_4)$ cannot contain $f_2$ or $f_5$ and thus also not $f_4$, since $\varphi(e_2) = -f_4+f_5$. This leaves us with $\varphi(e_4) = f_1\pm f_3\pm f_6$, leading to an even number of shared basis elements between $\varphi(e_4)$ and  $\varphi(e_1)$, a contradiction to the relation $e_4\cdot e_1=1$. 

\item[(ii)] Case of $\varphi(e_3)$ and $\varphi(e_4)$ only sharing $f_2$. In this setup $\varphi(e_4)$ cannot contain $f_1$ or $f_5$ and thus also not $f_4$ (again since $e_2\cdot e_4=0$). Similarly to the previous subcase we are left with $\varphi(e_4) = -f_2\pm f_3 \pm f_6$ leading to the same contradiction as in the previous subcase. 

\item[(iii)] Case of $\varphi(e_3)$ and $\varphi(e_4)$ only sharing $f_5$. Here $\varphi(e_4)$ cannot contain $f_1$ or $f_2$, while the relation $e_2\cdot e_4=0$  implies that $\varphi(e_4)$ must contain $f_4$. This implies that $\varphi(e_4) = f_5+f_4 \pm f_i$ for $i\in \{3,6\}$. Since $\varphi(e_4)$ shares an odd number of elements with $\varphi(e_1)$ we conclude that $i=6$ and that $\varphi(e_4) = f_5+f_4 \pm f_6$. This, regardless of the sign choice, implies that $\varphi(e_4) \cdot \varphi(e_1) = -1$,  a contradiction. Having exhausted all possibilities and having been led to a contradiction in each, we conclude that the embedding $\varphi$ cannot exist. 
\end{itemize}
\end{itemize}

It follows that $\gamma_4(9_{39})\ge 2$. Figure~\ref{9_39Subfigure} shows a band move from $9_{39}$ to $8_{11}$ and since $\gamma_4(8_{11}) = 1$, it follows that $\gamma_4(9_{39}) = 2$. 
%
\begin{figure}[h]
	\centering
	\begin{subfigure}[b]{0.34\textwidth}
		\includegraphics[width=\textwidth]{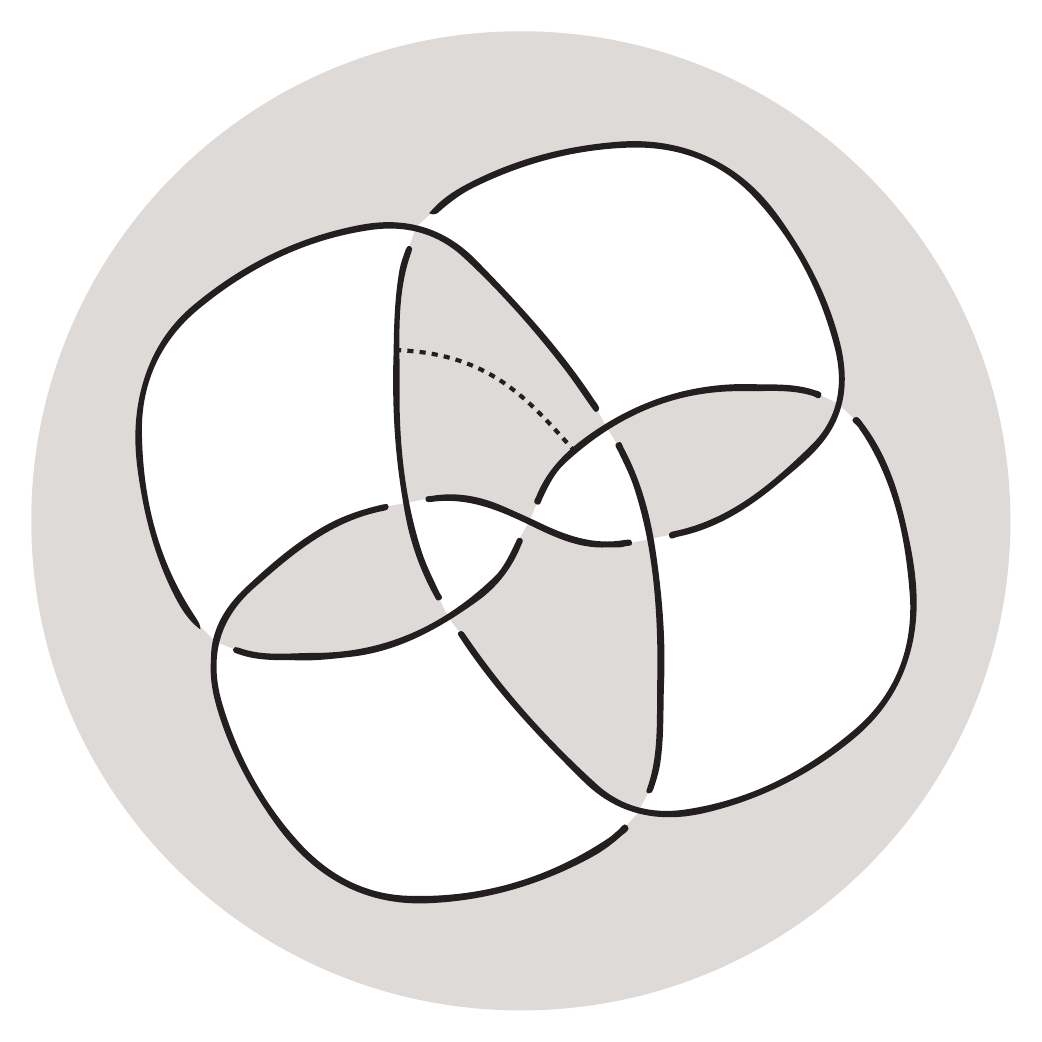}
		\caption{$9_{40}\stackrel{0}{\longrightarrow} 9_{31}$}
		\label{9_40Subfigure}
	\end{subfigure}
	\qquad \qquad \qquad
	\begin{subfigure}[b]{0.38\textwidth}
		\includegraphics[width=\textwidth]{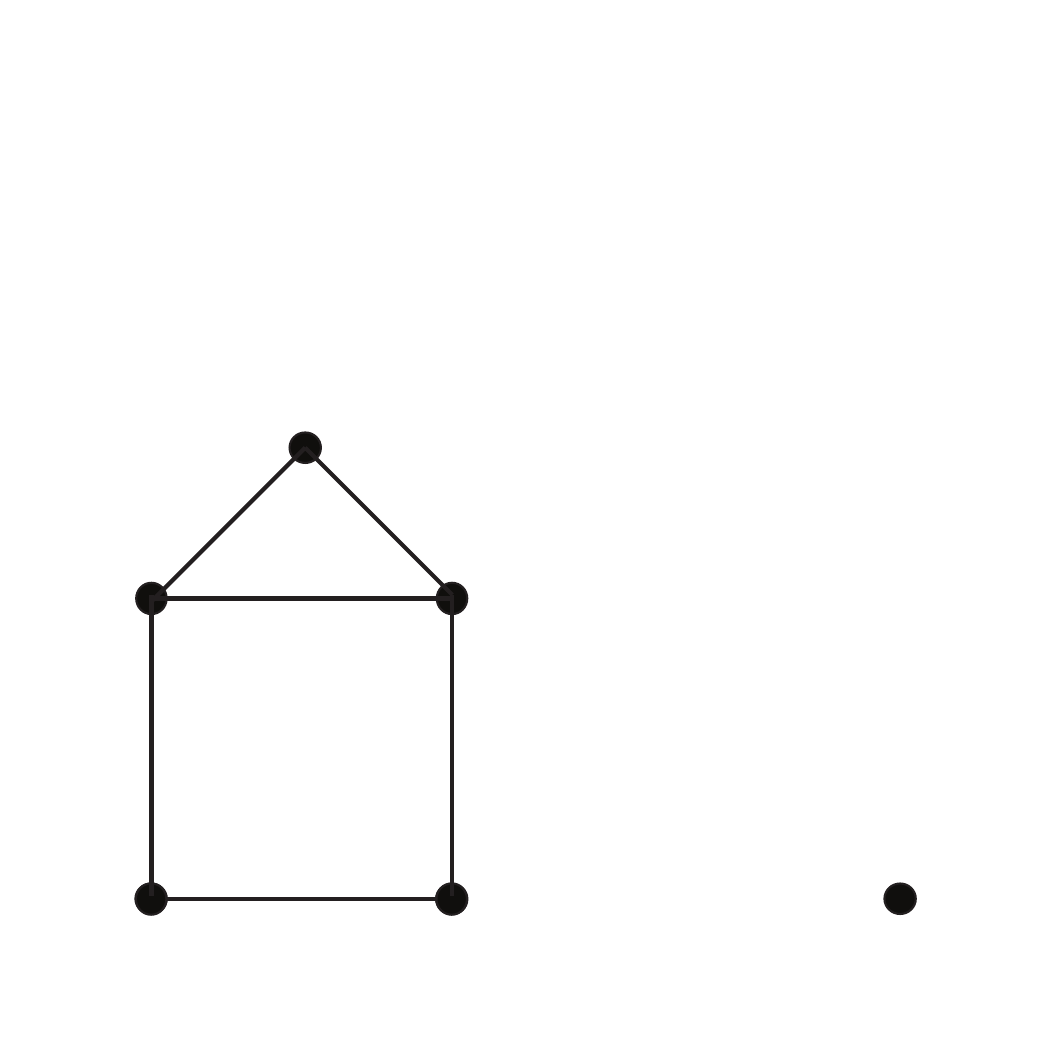}           
		\put(-139,32){$e_1$}
		\put(-156,7){$-3$}
		\put(-109,32){$e_2$}
		\put(-108,7){$-3$}
		\put(-109,60){$e_3$}
		\put(-88,68){$-3$}
		\put(-110,95){$e_5$}
		\put(-141,93){$-3$}
		\put(-139,60){$e_4$}
		\put(-164,68){$-3$}
		\put(-26,8){$e_6$}
		\put(-53,30){$-3 \text{ or } -75$}
		\caption{Goeritz form for $9_{40}$.}
		\label{9_40Intersectionform}
	\end{subfigure}
	\vskip3mm
	\caption{Case of $K=9_{40}$.}\label{9_40}
\end{figure}
\vskip3mm
{\bf Case of $K=9_{40}$. } The negative definite Goeritz matrix $G$ associated to the checkerboard coloring of the knot $K=9_{40}$ in Figure~\ref{9_40Subfigure} is the incidence matrix of the weighted graph in Figure~\ref{9_40Intersectionform}. Since $\det 9_{40}=75 = 3\cdot 5^2$ we wish to obstruct the existence of an embedding 
\[
\varphi :(\mathbb Z^6, G\oplus[-d]) \hookrightarrow (\mathbb Z^6, -\text{Id})
\]
for $d=3$ and for $d=75$. If such a $\varphi$ existed, then its restriction to $V:=\text{Span}(e_1,e_2,e_3,e_4)$ would be an embedding of $(V,G|_{V\times V})$ into $(\mathbb Z^6, -\text{Id})$. However the form $(V,G|_{V\times V})$ is isomorphic to the form considered in Proposition \ref{PropositionAbout8_18}, where it was shown not to embed into $(\mathbb Z^6, -\text{Id})$.
It follows that $\gamma_4(9_{40})\ge 2$, and since the non-oriented band move in Figure~\ref{9_40Subfigure} turns $9_{40}$ into $9_{31}$, a knot with $\gamma_4$ equal to 1, we conclude that $\gamma_4(9_{40}) = 2$. 

\subsection{Knots with $\sigma (K) + 4\cdot \text{Arf}(K) \equiv 0\pmod{8}$} \label{SectionOnNullYasuharaKnots}
The 17 knots $K$ with crossing number 8 or 9, that satisfy the congruence relation $\sigma (K) + 4 \cdot \text{Arf}(K) \equiv  0\pmod{8}$ are, up to passing to mirrors, given by 
\begin{gather} \label{ListOfNullYasuharaKnots}
\underline{8_3},\, \underline{8_5},\, \underline{8_8},\, \underline{8_9},\, \underline{8_{20}},\, \cr
\underline{9_1},\, \underline{9_4},\, \underline{9_7},\, \underline{9_{13}},\, \underline{9_{19}},\, \underline{9_{23}},\, \underline{9_{27}},\, \underline{9_{36}}, \, \underline{9_{41}},\, \underline{9_{43}}, \, \underline{9_{44}}, \, \underline{9_{46}}.
\end{gather}

\begin{figure}[!htbp]
	\centering
	\begin{subfigure}[b]{0.27\textwidth}
		\includegraphics[width=\textwidth]{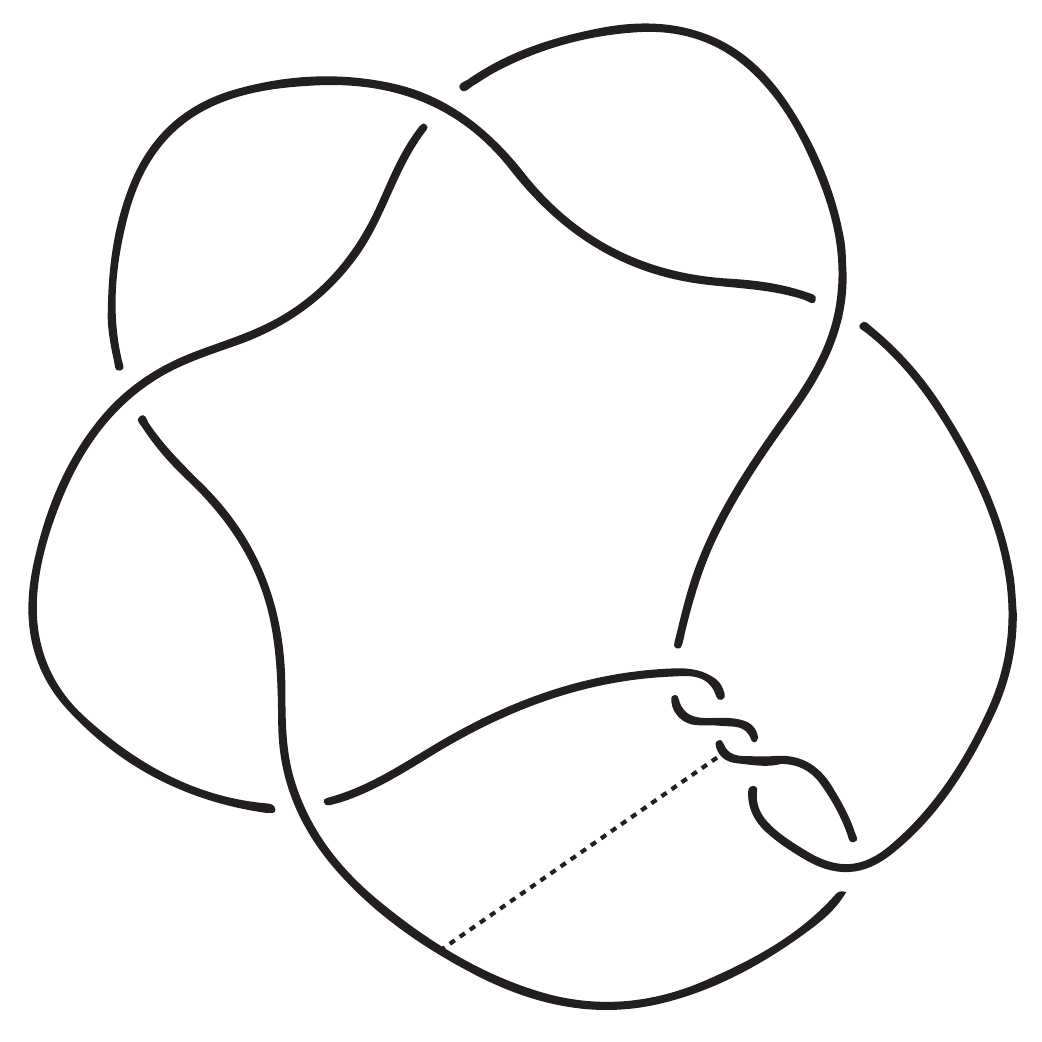}
		\caption{$8_{3}\stackrel{1}{\longrightarrow} 0_1$}
		\label{FigureFor8_3}
	\end{subfigure}
	~
	\begin{subfigure}[b]{0.3\textwidth}
		\includegraphics[width=\textwidth]{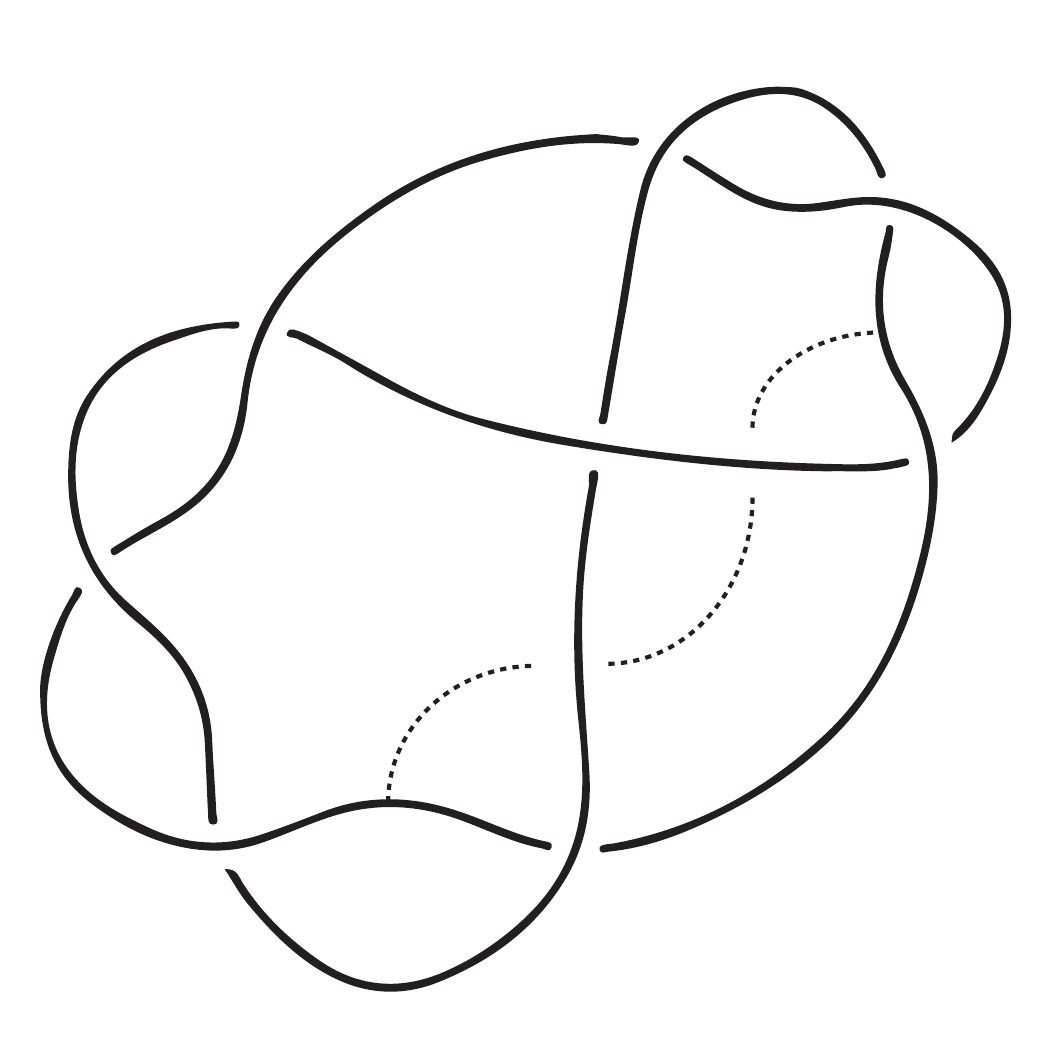}
		\caption{$8_{4}\stackrel{1}{\longrightarrow} 0_1$}
		\label{FigureFor8_4}
	\end{subfigure}
	~
	\begin{subfigure}[b]{0.3\textwidth}
		\includegraphics[width=\textwidth]{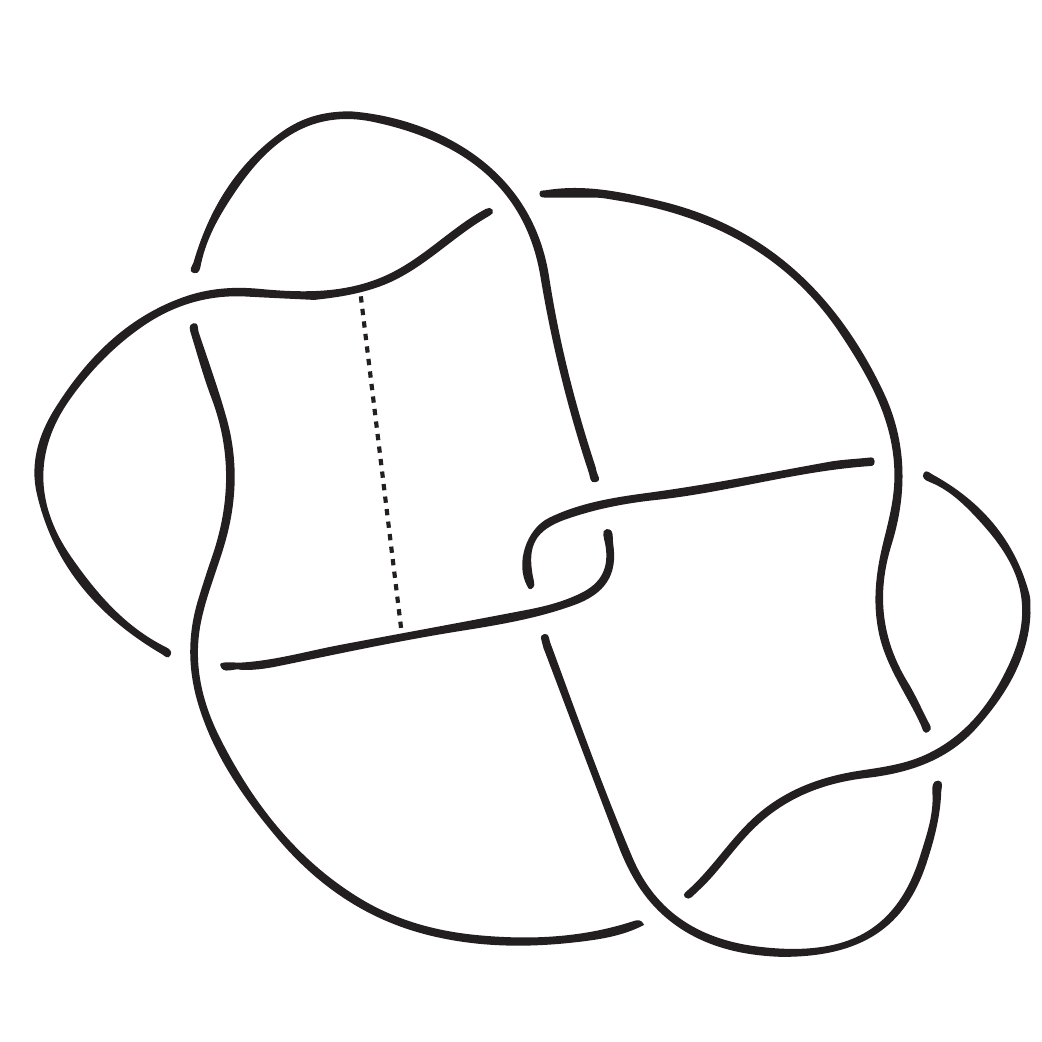}
		\caption{$8_{5}\stackrel{-1\phantom{i}}{\longrightarrow} 0_1$}
		\label{FigureFor8_5}
	\end{subfigure}
	\vskip3mm
	\begin{subfigure}[b]{0.3\textwidth}
		\includegraphics[width=\textwidth]{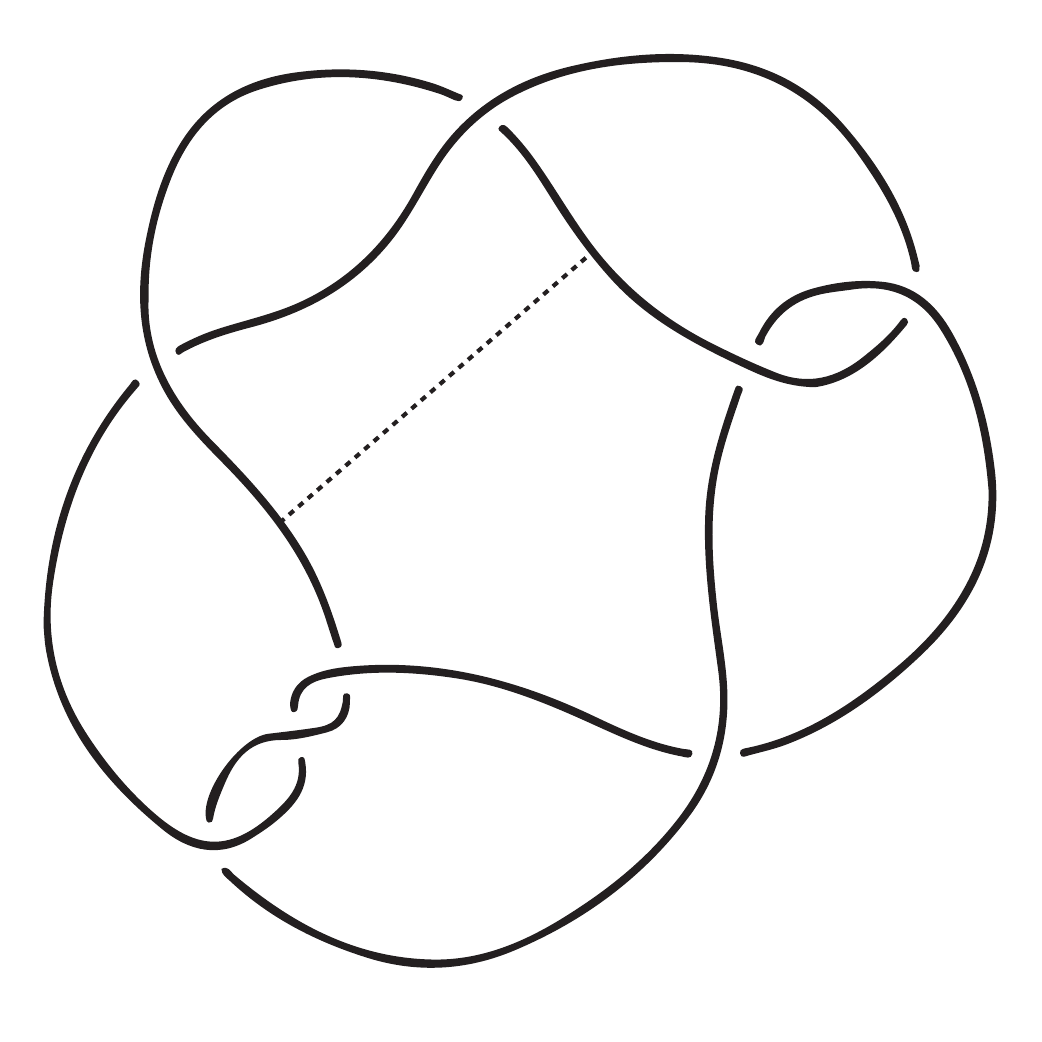}
		\caption{$8_{6}\stackrel{1}{\longrightarrow} 0_1$}
		\label{FigureFor8_6}
	\end{subfigure}
	~
	\begin{subfigure}[b]{0.3\textwidth}
		\includegraphics[width=\textwidth]{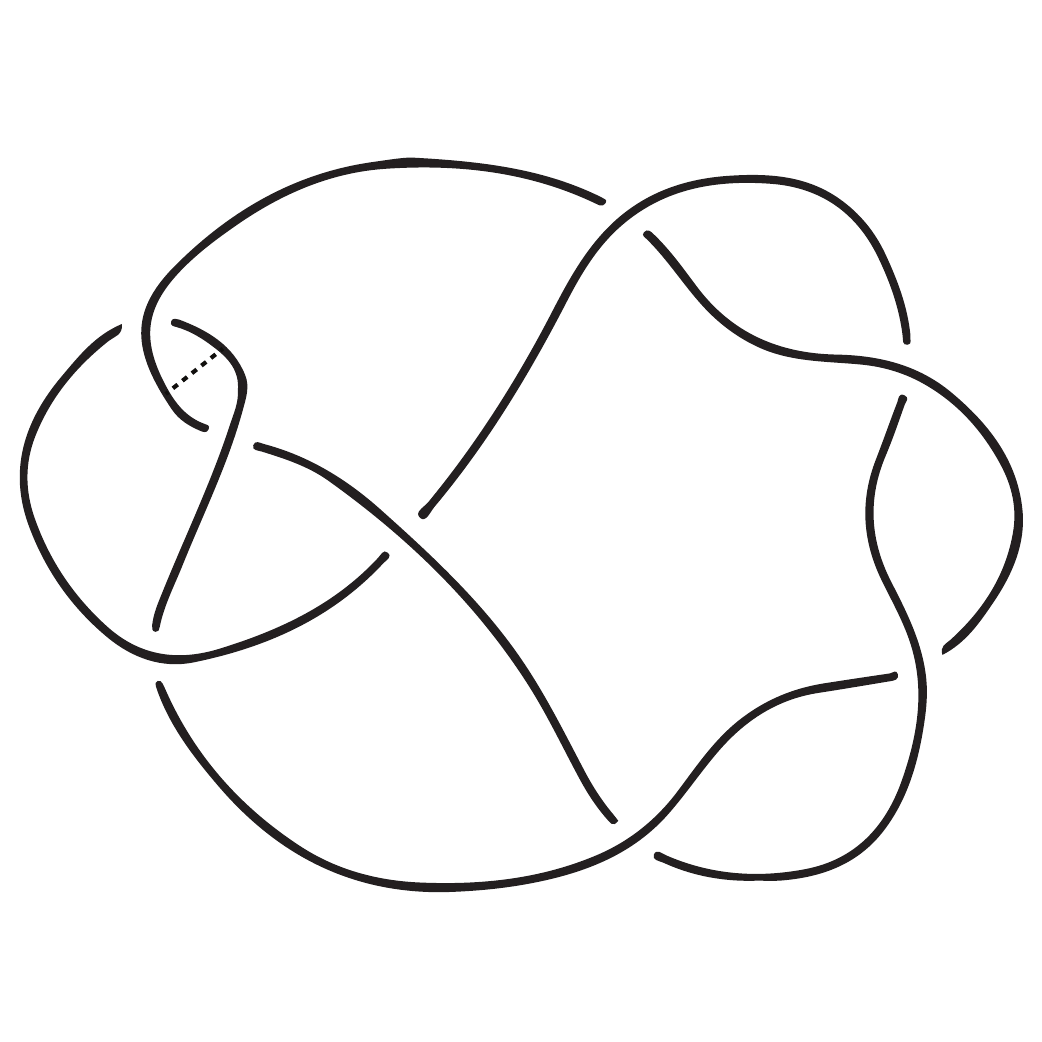}
		\caption{$8_{7}\stackrel{0}{\longrightarrow} 6_1$}
		\label{FigureFor8_7}
	\end{subfigure}
	~
	\begin{subfigure}[b]{0.27\textwidth}
		\includegraphics[width=\textwidth]{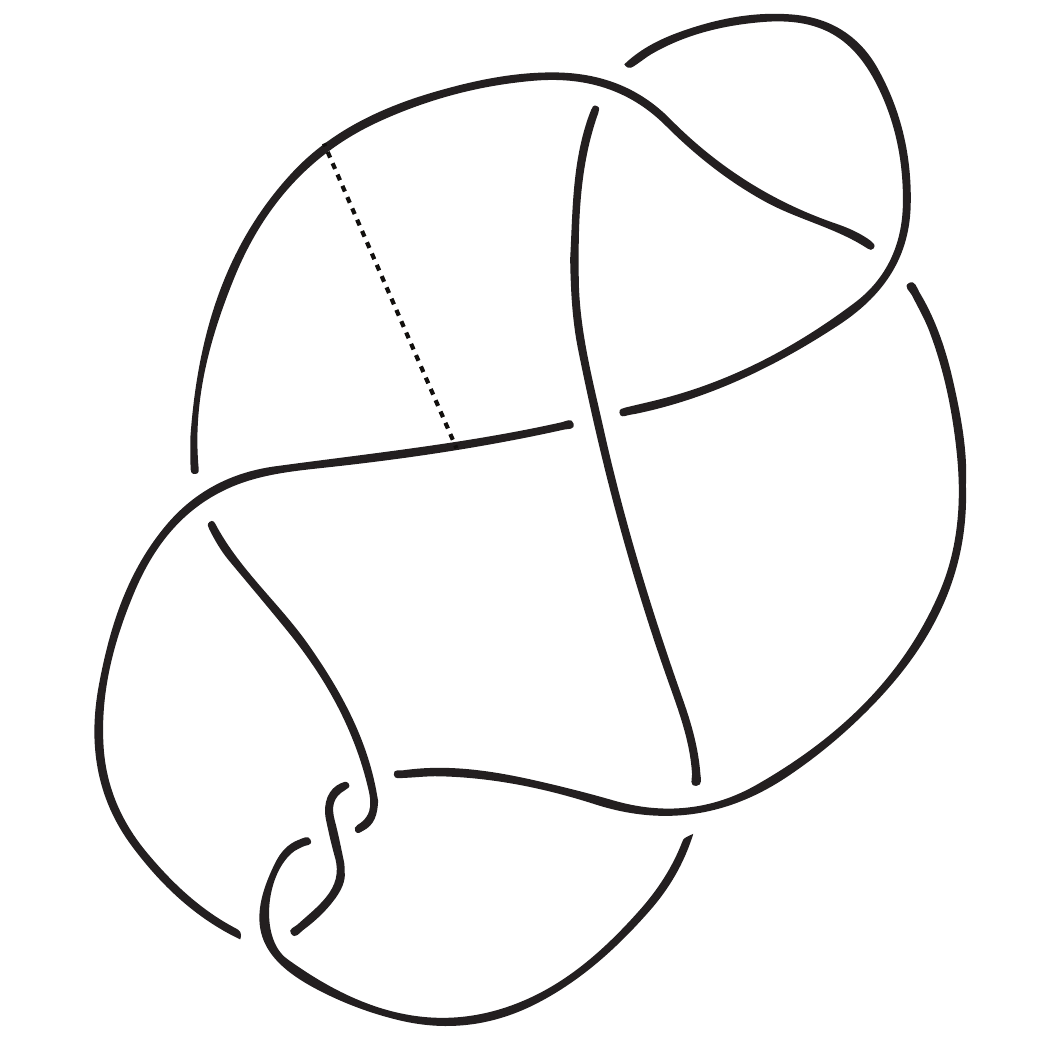}
		\caption{$8_{11}\stackrel{0}{\longrightarrow} 3_1\#(-3_1)$}
		\label{FigureFor8_11}
	\end{subfigure}
	\vskip3mm
	\begin{subfigure}[b]{0.27\textwidth}
		\includegraphics[width=\textwidth]{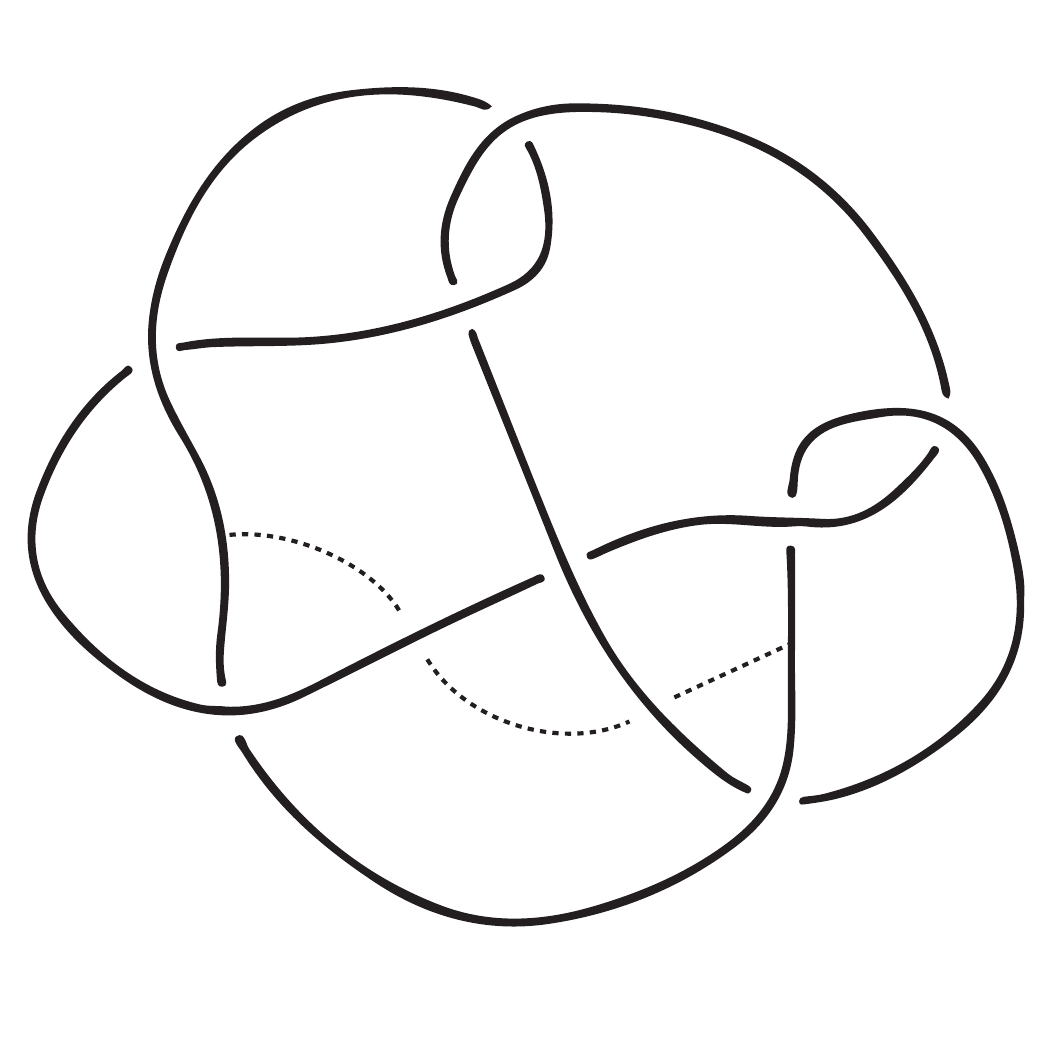}
		\caption{$8_{14}\stackrel{1}{\longrightarrow} 8_8$}
		\label{FigureFor8_14}
	\end{subfigure}
	~
	\begin{subfigure}[b]{0.27\textwidth}
		\includegraphics[width=\textwidth]{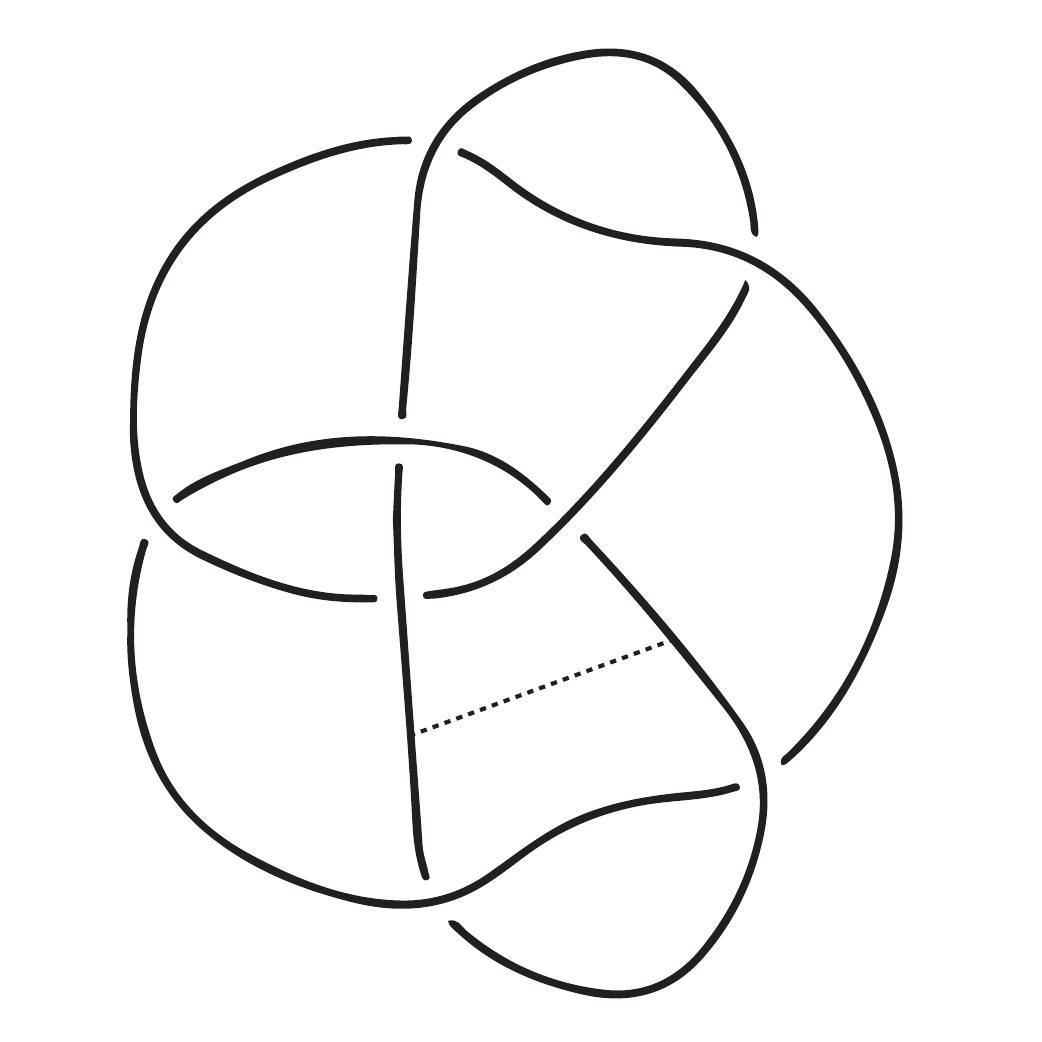}
		\caption{$8_{16}\stackrel{1}{\longrightarrow} 8_{20}$}
		\label{FigureFor8_16}
	\end{subfigure}
	~
	\begin{subfigure}[b]{0.3\textwidth}
		\includegraphics[width=\textwidth]{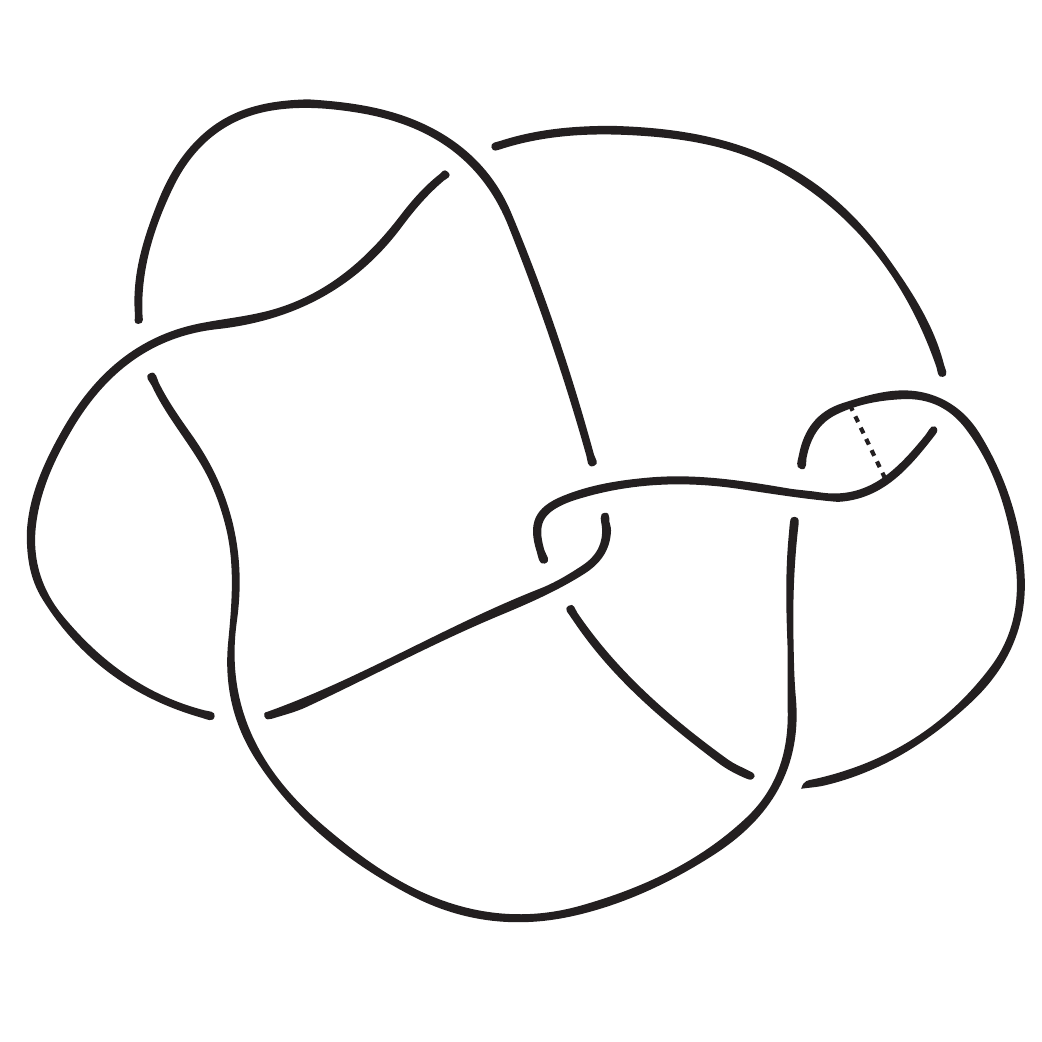}
		\caption{$8_{19}\stackrel{0}{\longrightarrow} 0_1$}
		\label{FigureFor8_19}
	\end{subfigure}
	\vskip3mm
	\caption{Non-oriented band moves from the knots $8_{3}$, $8_4$, $8_{5}$, $8_{6}$, $8_{7}$, $8_{11}$, $8_{14}$, $8_{16}$, $8_{19}$ to smoothly slice knots.}\label{Figure1ToSliceKnots}
\end{figure}
\begin{figure}[!htbp]
	\centering
	\begin{subfigure}[b]{0.27\textwidth}
		\includegraphics[width=\textwidth]{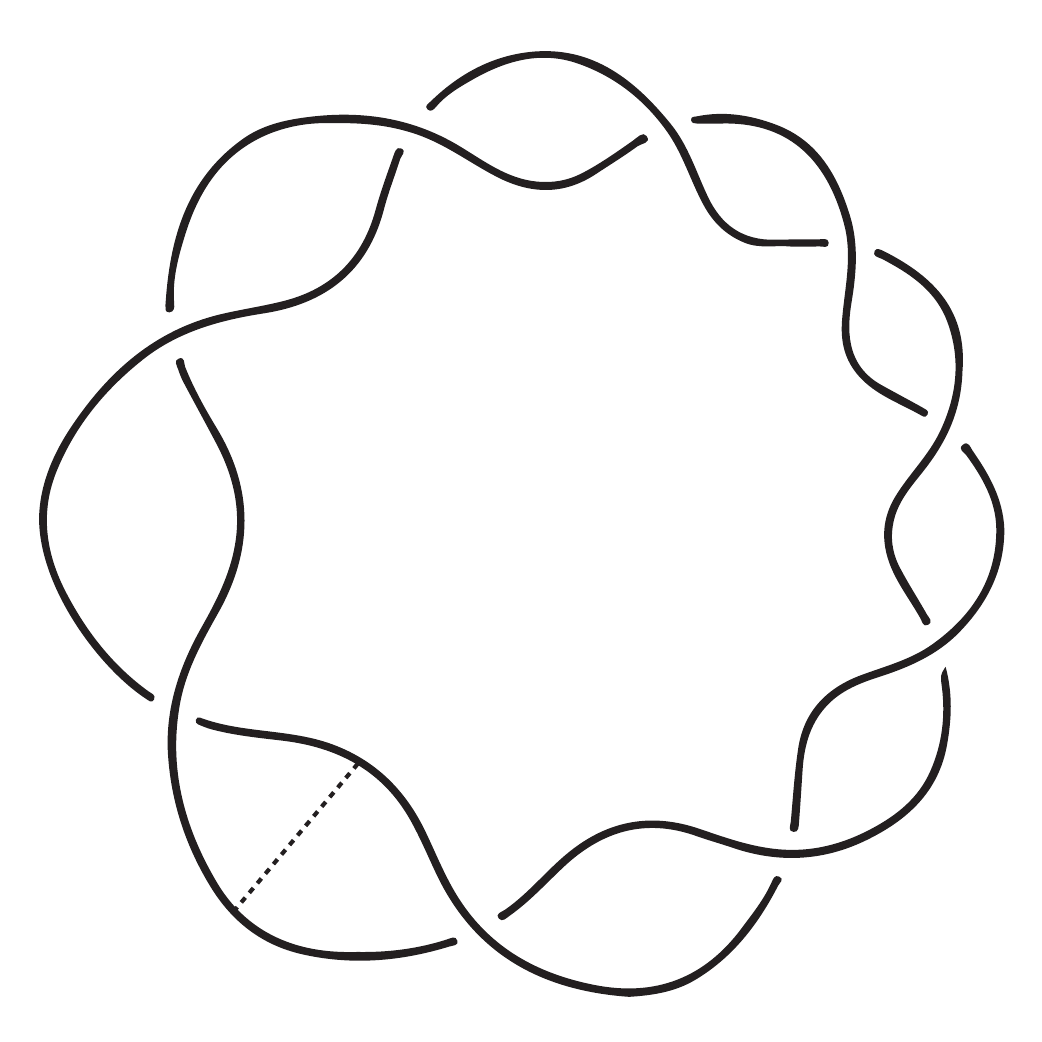}
		\caption{$9_{1}\stackrel{0}{\longrightarrow} 0_1$}
		\label{FigureFor9_1}
	\end{subfigure}
	~
	\begin{subfigure}[b]{0.28\textwidth}
		\includegraphics[width=\textwidth]{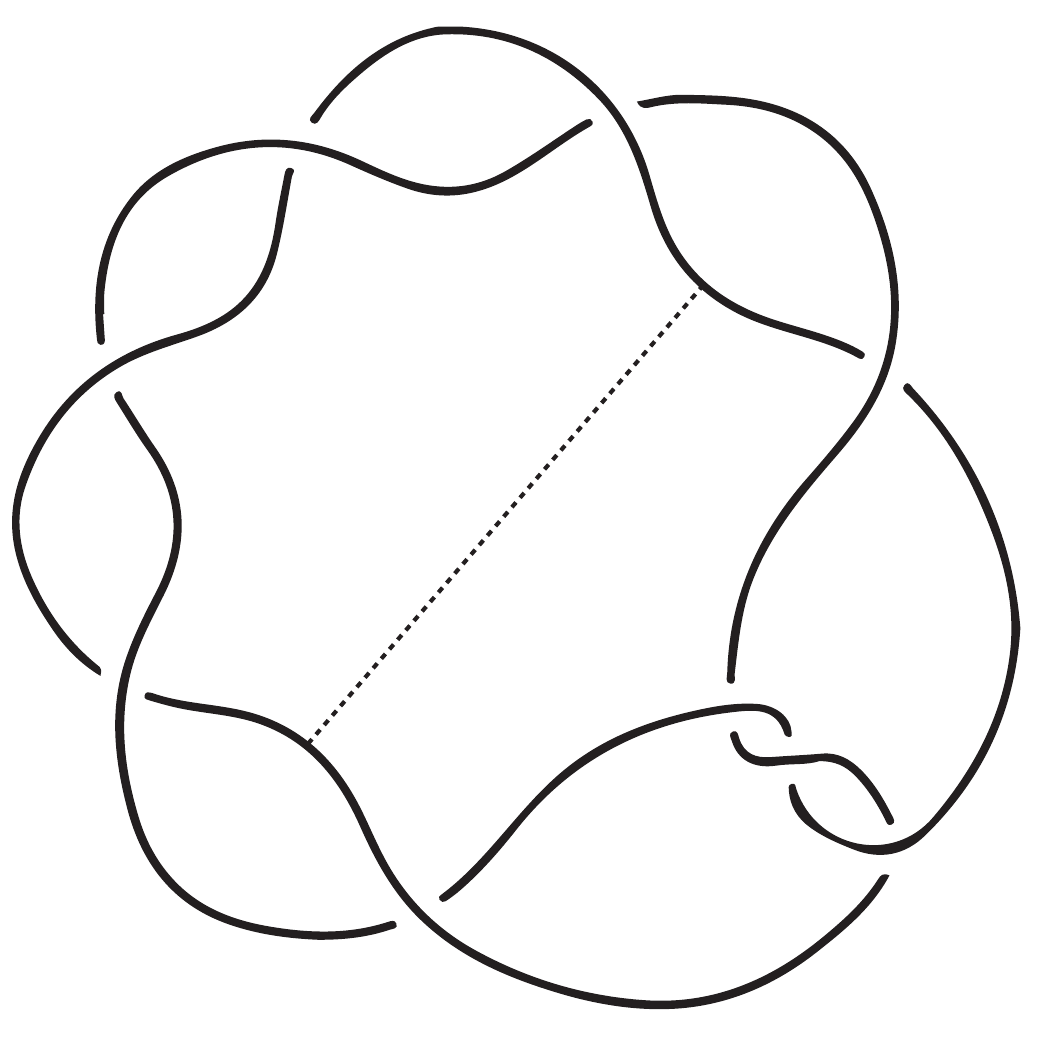}
		\caption{$9_{3}\stackrel{-1\phantom{i}}{\longrightarrow} 6_1$}
		\label{FigureFor9_3}
	\end{subfigure}
	~
	\begin{subfigure}[b]{0.28\textwidth}
		\includegraphics[width=\textwidth]{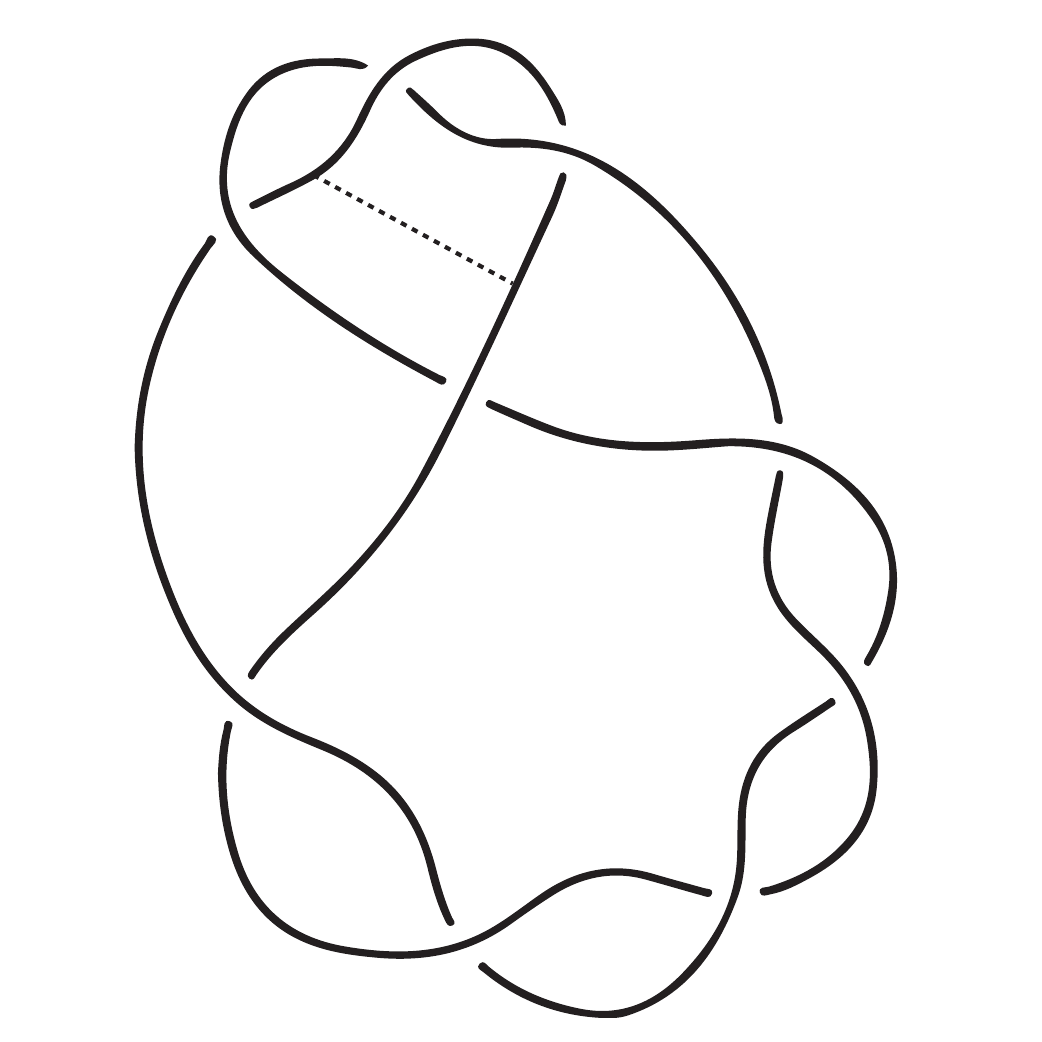}
		\caption{$9_{5}\stackrel{1}{\longrightarrow} 0_1$}
		\label{FigureFor9_5}
	\end{subfigure}
	\vskip3mm
	\begin{subfigure}[b]{0.28\textwidth}
		\includegraphics[width=\textwidth]{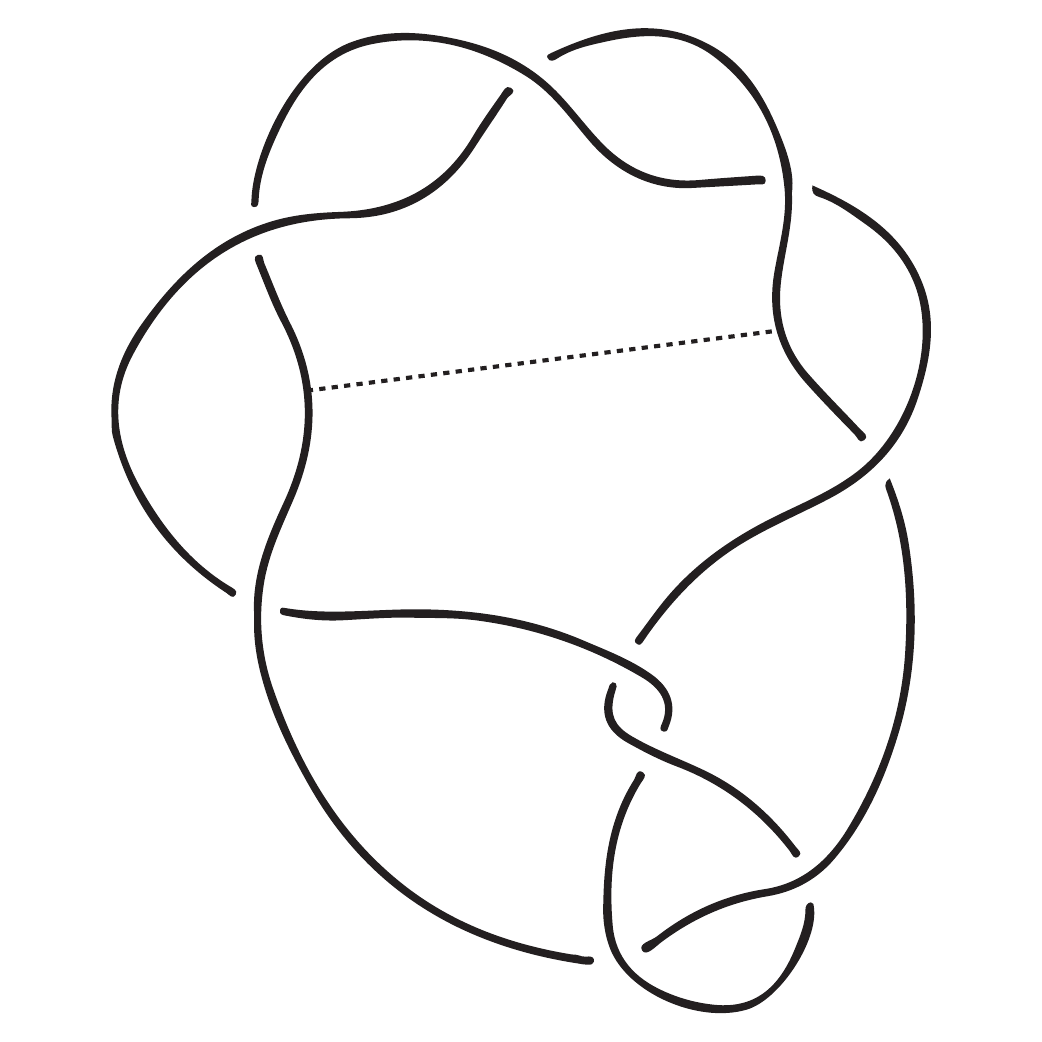}
		\caption{$9_{6}\stackrel{-1\phantom{i}}{\longrightarrow} 6_1$}
		\label{FigureFor9_6}
	\end{subfigure}
	~
	\begin{subfigure}[b]{0.28\textwidth}
		\includegraphics[width=\textwidth]{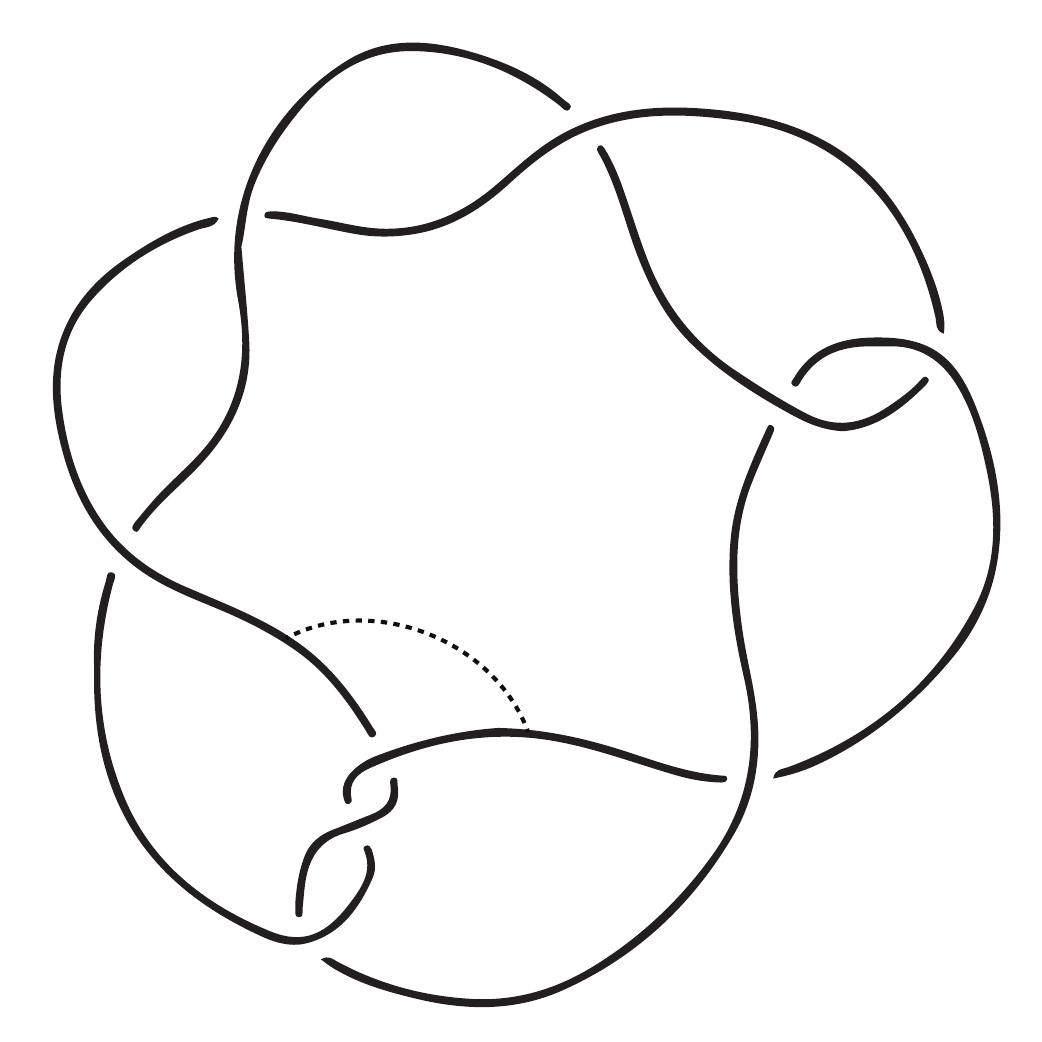}
		\caption{$9_{7}\stackrel{0}{\longrightarrow} 6_1$}
		\label{FigureFor9_7}
	\end{subfigure}
	~
	\begin{subfigure}[b]{0.3\textwidth}
		\includegraphics[width=\textwidth]{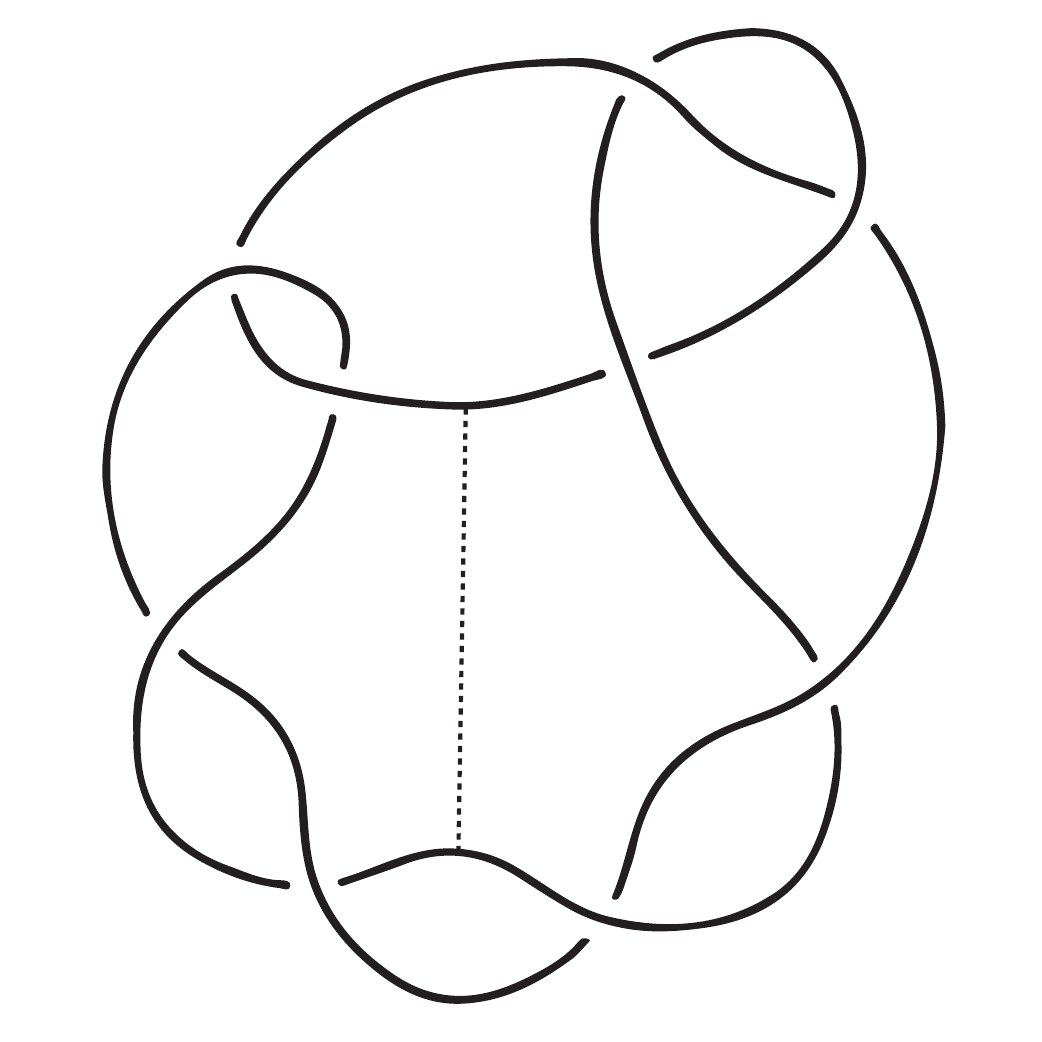}
		\caption{$9_{8}\stackrel{-1\phantom{i}}{\longrightarrow} 6_1$}
		\label{FigureFor9_8}
	\end{subfigure}
	\vskip3mm
	\begin{subfigure}[b]{0.3\textwidth}
		\includegraphics[width=\textwidth]{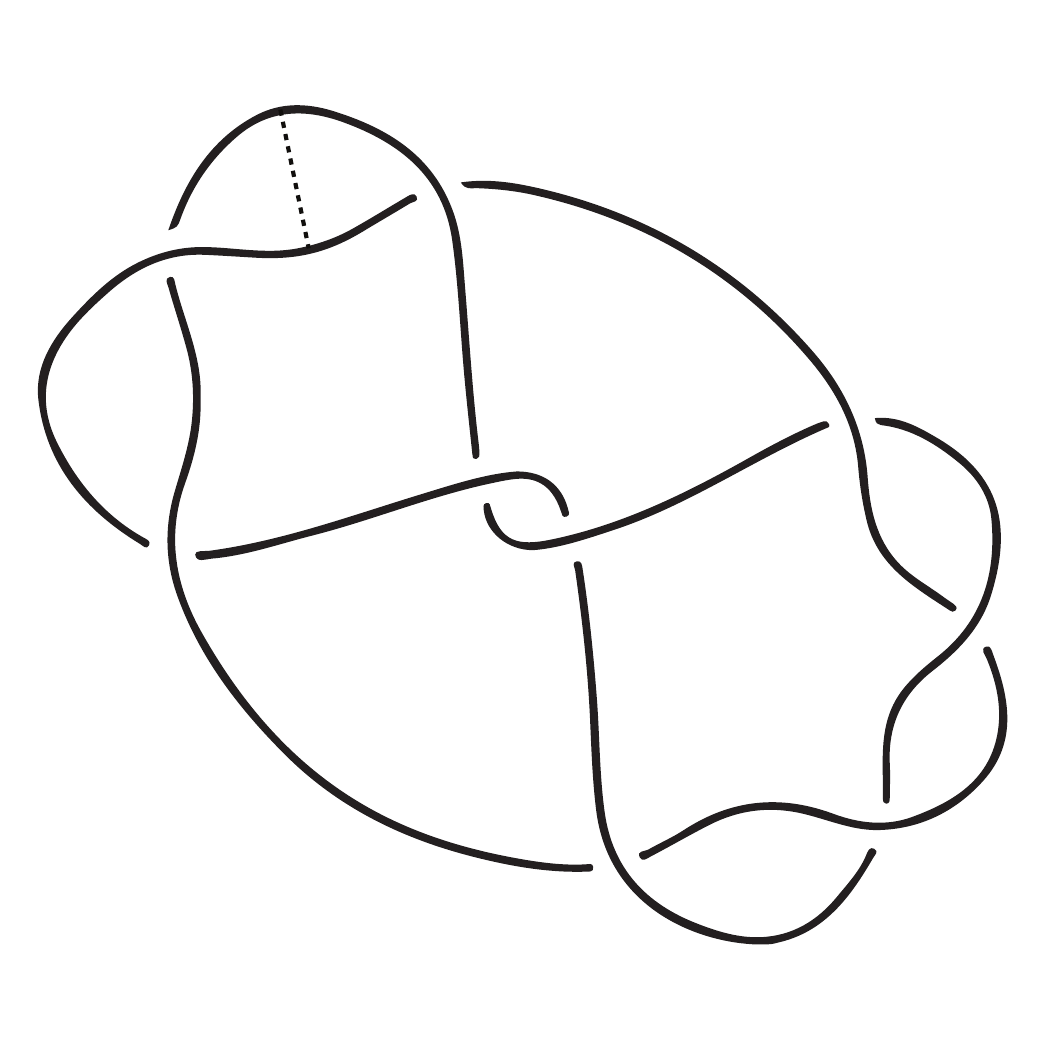}
		\caption{$9_{9}\stackrel{0}{\longrightarrow} 6_1$}
		\label{FigureFor9_9}
	\end{subfigure}
	~
	\begin{subfigure}[b]{0.27\textwidth}
		\includegraphics[width=\textwidth]{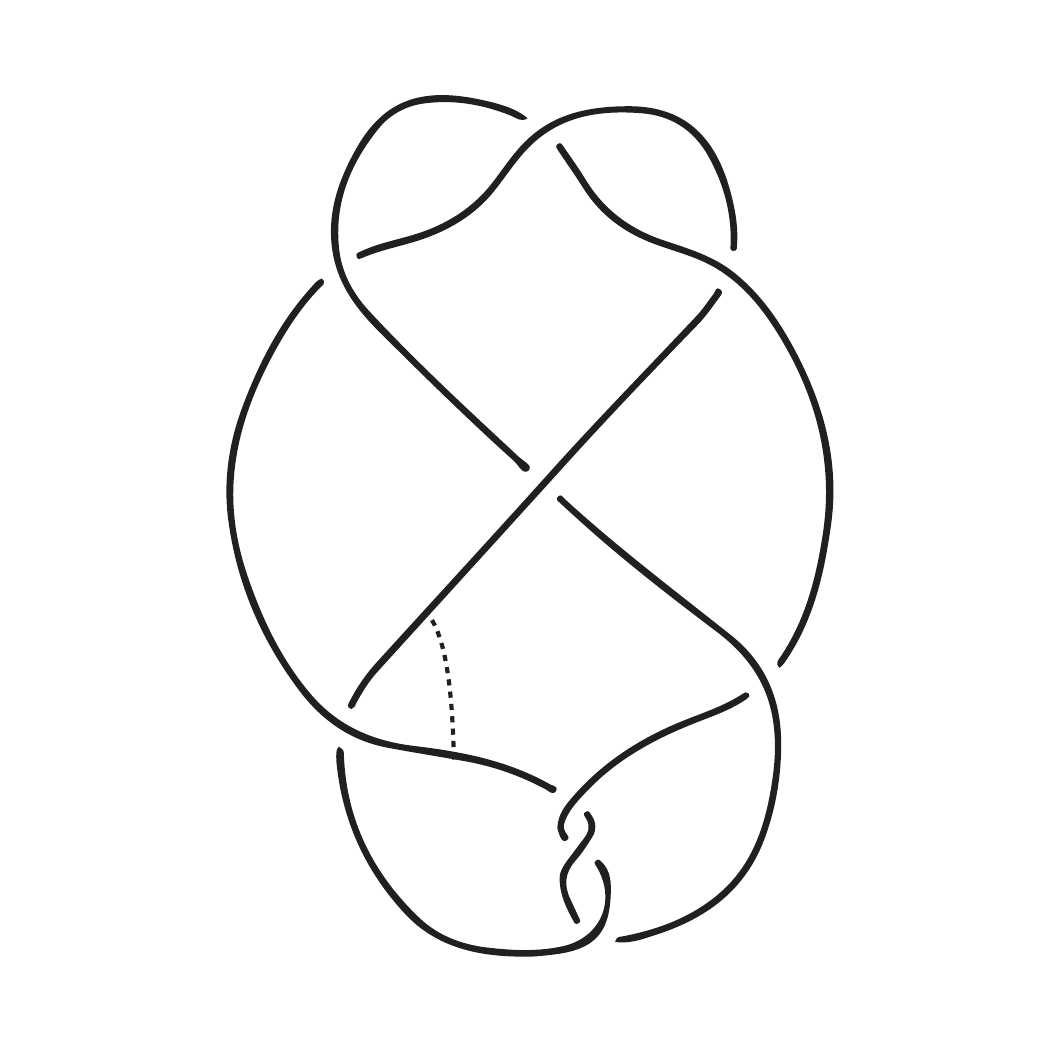}
		\caption{$9_{13}\stackrel{0}{\longrightarrow} 8_9$}
		\label{FigureFor9_13}
	\end{subfigure}
	~
	\begin{subfigure}[b]{0.27\textwidth}
		\includegraphics[width=\textwidth]{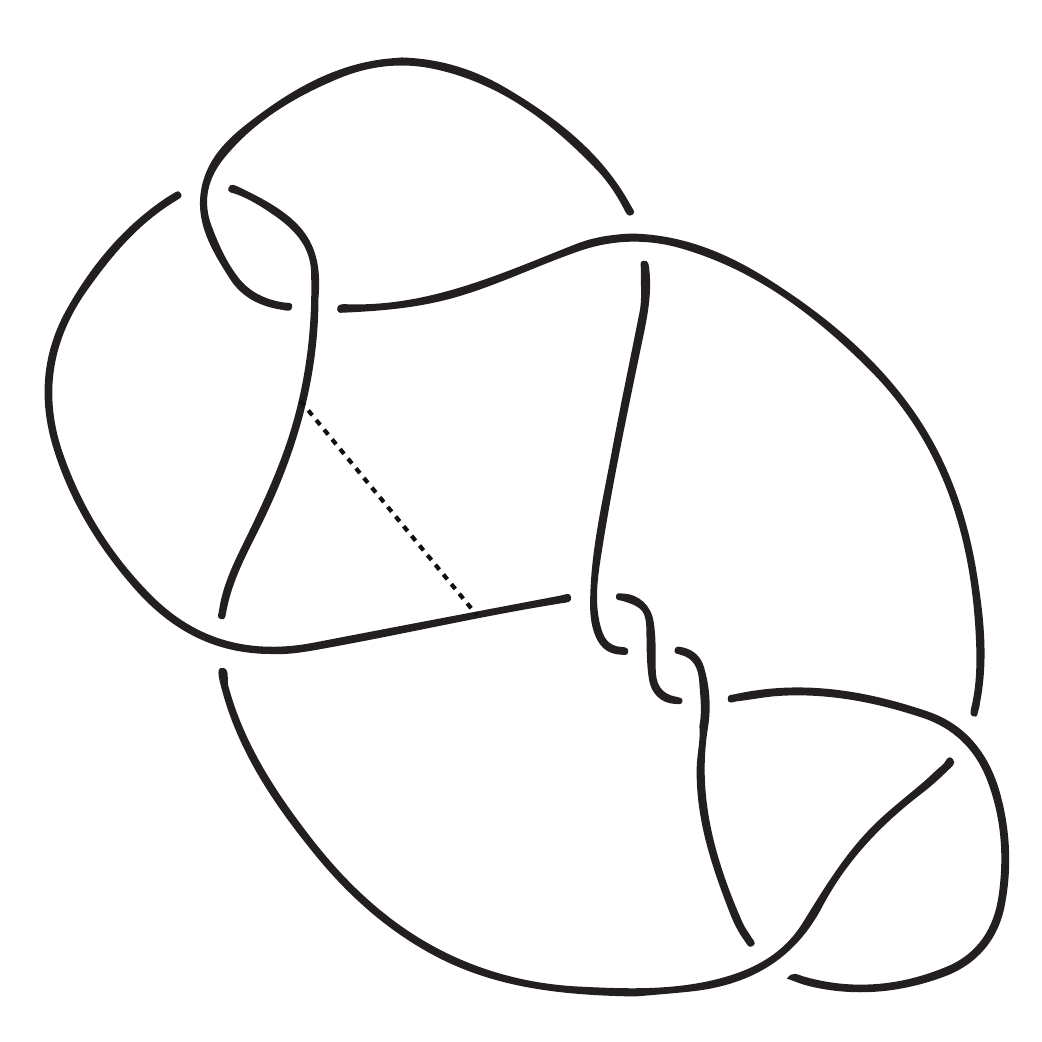}
		\caption{$9_{15}\stackrel{0}{\longrightarrow} 8_8$}
		\label{FigureFor9_15}
	\end{subfigure}
	\vskip3mm
	\caption{Non-oriented band moves from the knots $9_1$, $9_3$, $9_{5}$, $9_{6}$, $9_{7}$, $9_{8}$, $9_{9}$, $9_{13}$, $9_{15}$ to smoothly slice knots.}\label{Figure2ToSliceKnots}
\end{figure}
\begin{figure}[!htbp]
	\centering
	\begin{subfigure}[b]{0.28\textwidth}
		\includegraphics[width=\textwidth]{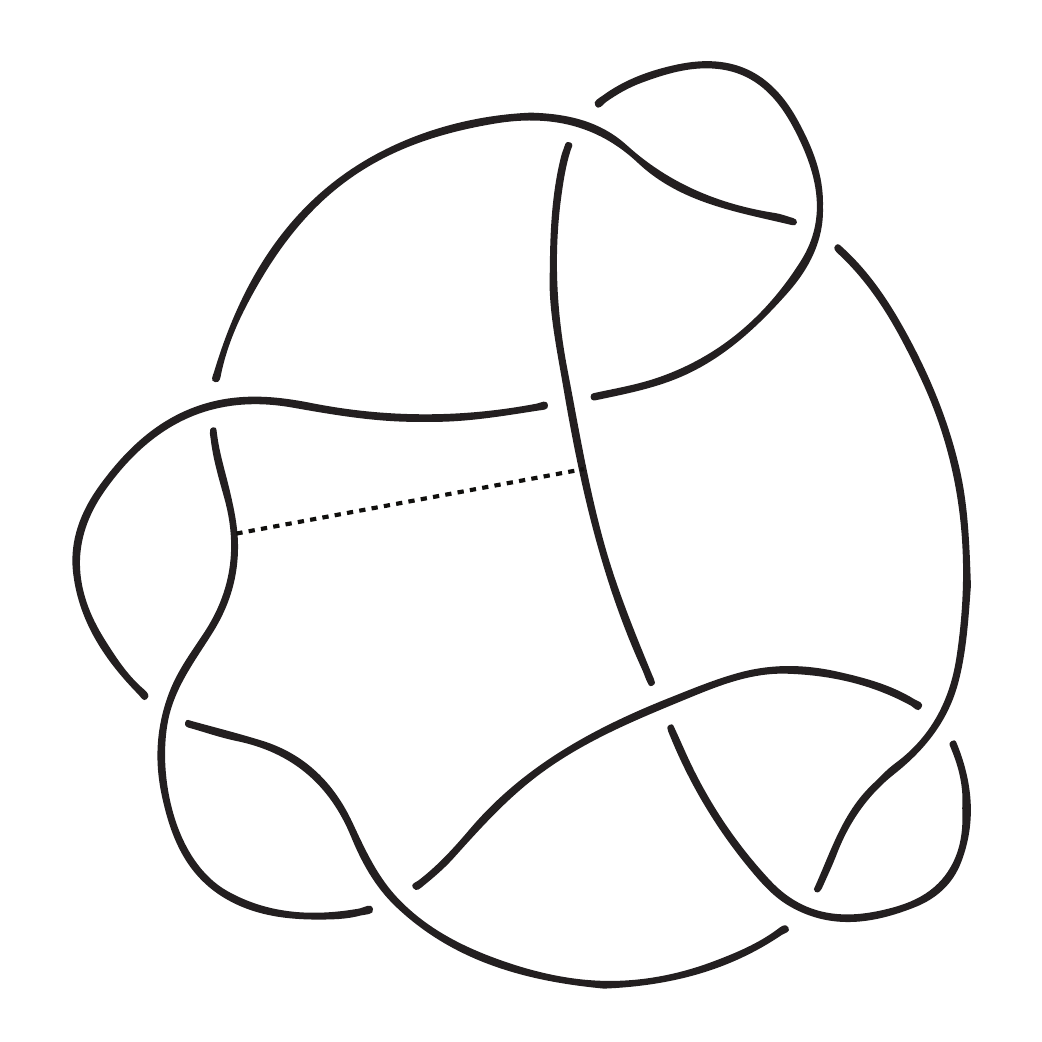}
		\caption{$9_{17}\stackrel{-1\phantom{i}}{\longrightarrow} 0_1$}
		\label{FigureFor9_17}
	\end{subfigure}
	~
	\begin{subfigure}[b]{0.28\textwidth}
		\includegraphics[width=\textwidth]{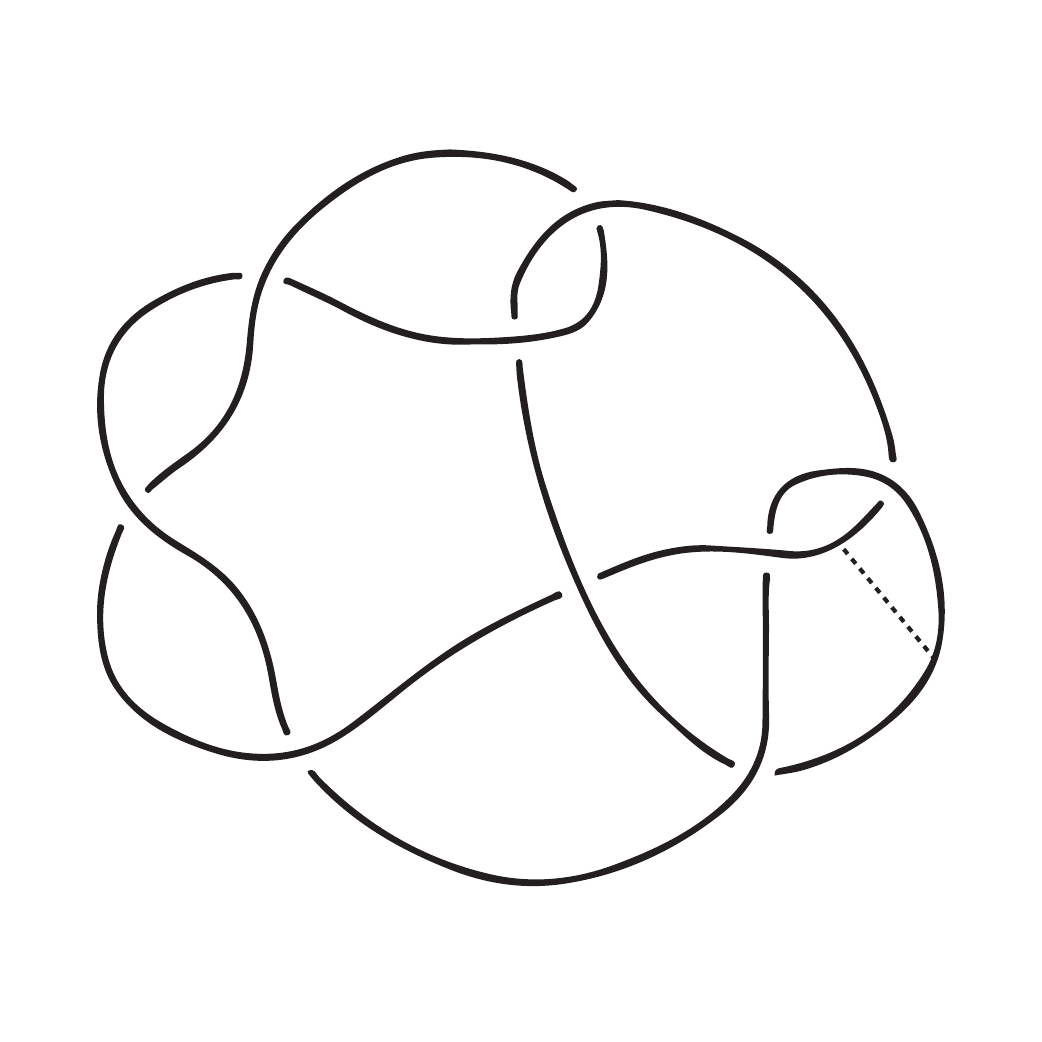}
		\caption{$9_{19}\stackrel{0}{\longrightarrow} 8_8$}
		\label{FigureFor9_19}
	\end{subfigure}
	~
	\begin{subfigure}[b]{0.28\textwidth}
		\includegraphics[width=\textwidth]{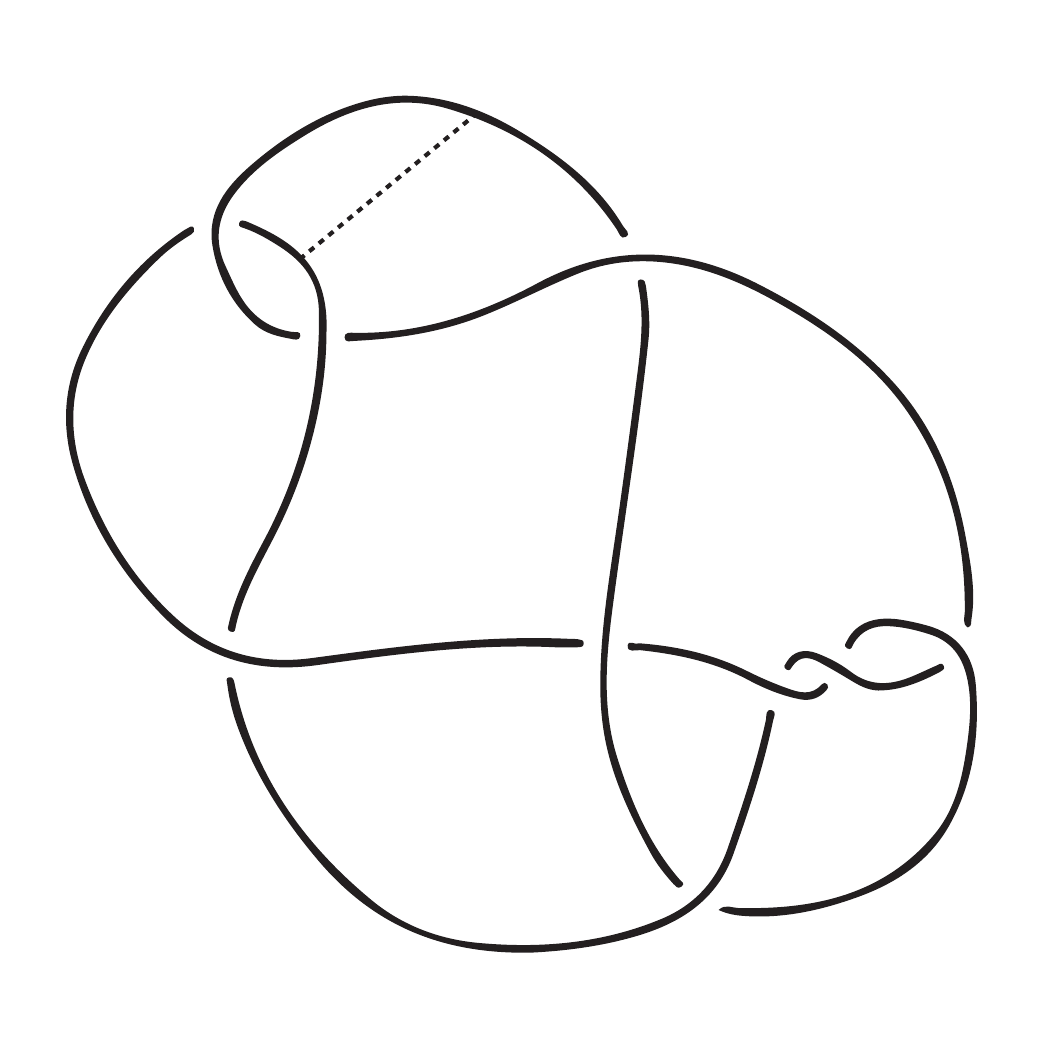}
		\caption{$9_{21}\stackrel{0}{\longrightarrow} 8_9$}
		\label{FigureFor9_21}
	\end{subfigure}
	\vskip3mm
	\begin{subfigure}[b]{0.28\textwidth}
		\includegraphics[width=\textwidth]{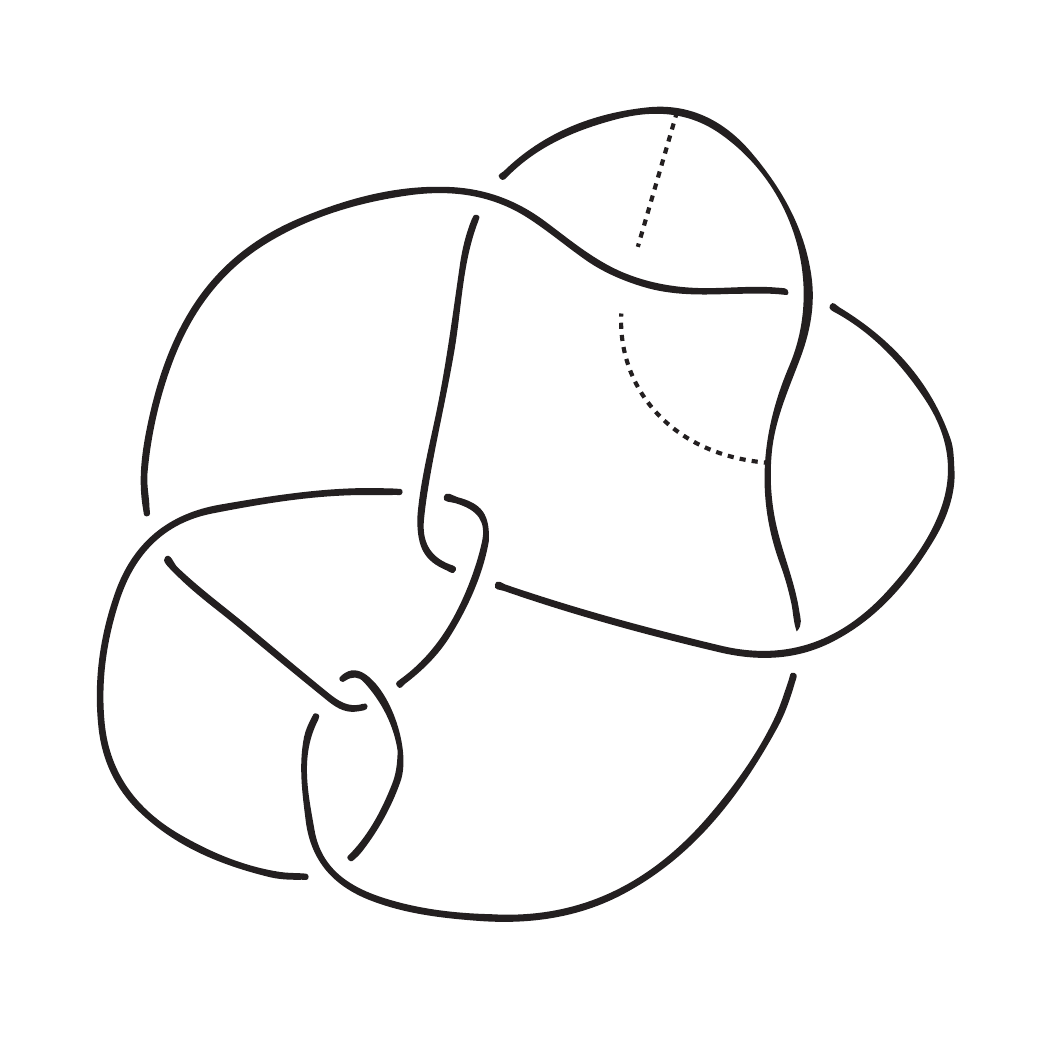}
		\caption{$9_{22}\stackrel{-1\phantom{i}}{\longrightarrow} 0_1$}
		\label{FigureFor9_22}
	\end{subfigure}
	~
	\begin{subfigure}[b]{0.28\textwidth}
		\includegraphics[width=\textwidth]{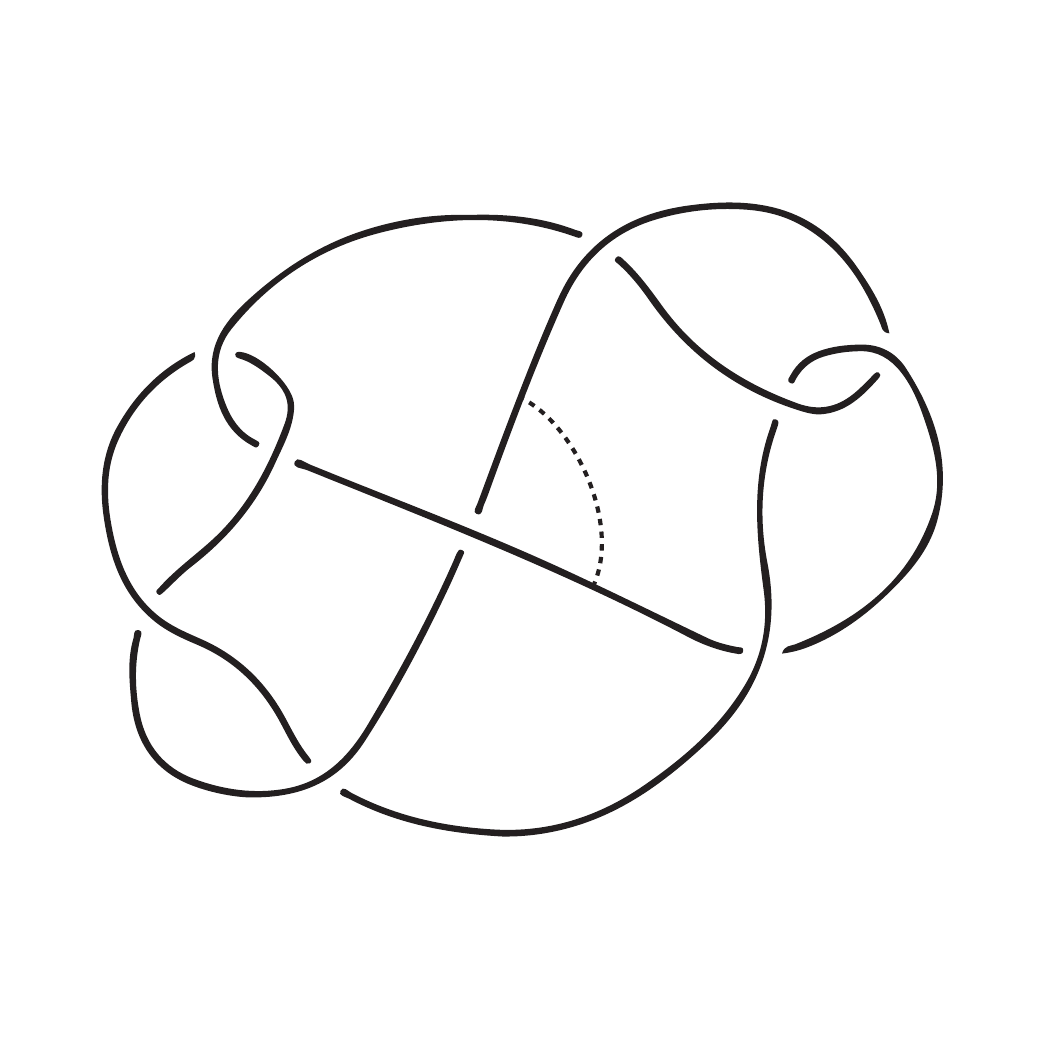}
		\caption{$9_{23}\stackrel{0}{\longrightarrow} 4_1\#4_1$}
		\label{FigureFor9_23}
	\end{subfigure}
	~
	\begin{subfigure}[b]{0.28\textwidth}
		\includegraphics[width=\textwidth]{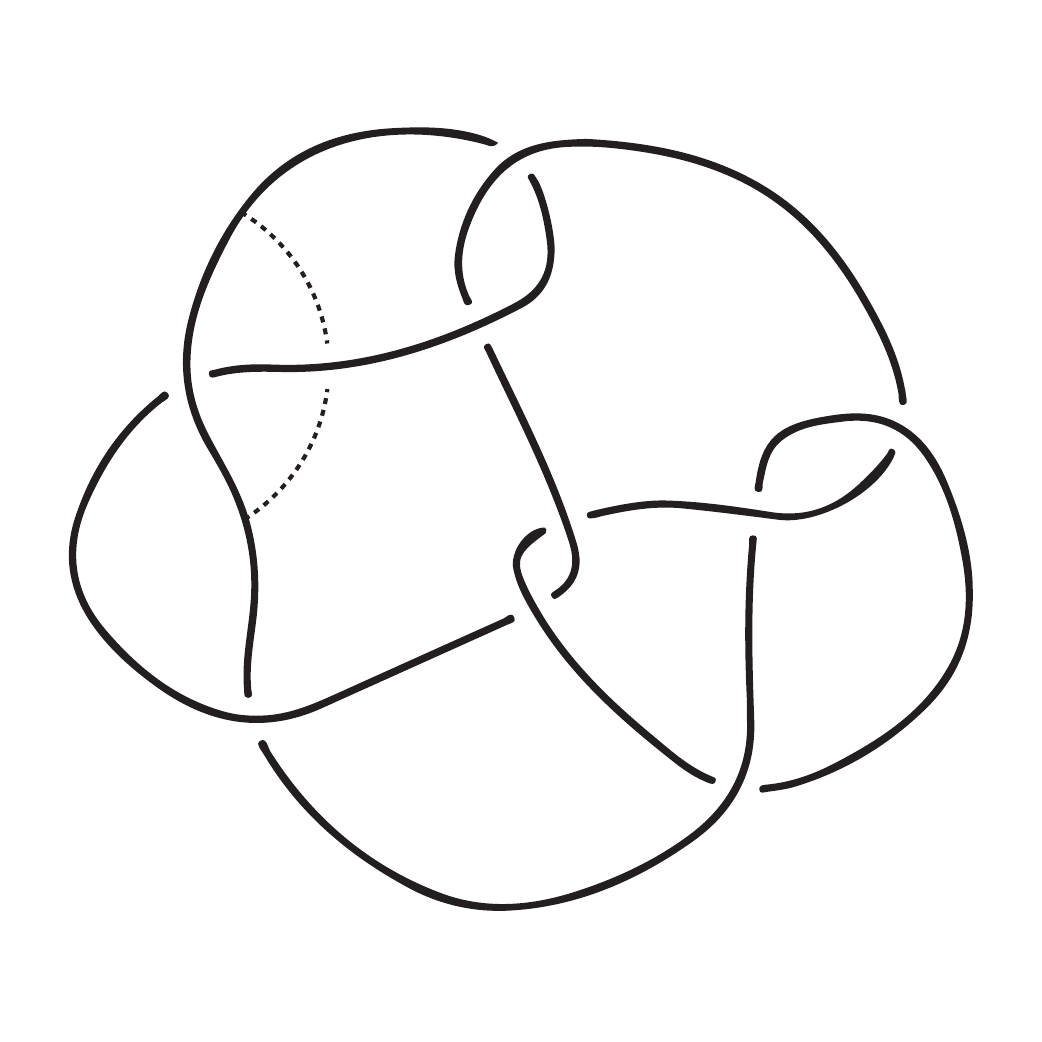}
		\caption{$9_{25}\stackrel{1}{\longrightarrow} 8_{20}$}
		\label{FigureFor9_25}
	\end{subfigure}
	\vskip3mm
	\begin{subfigure}[b]{0.25\textwidth}
		\includegraphics[width=\textwidth]{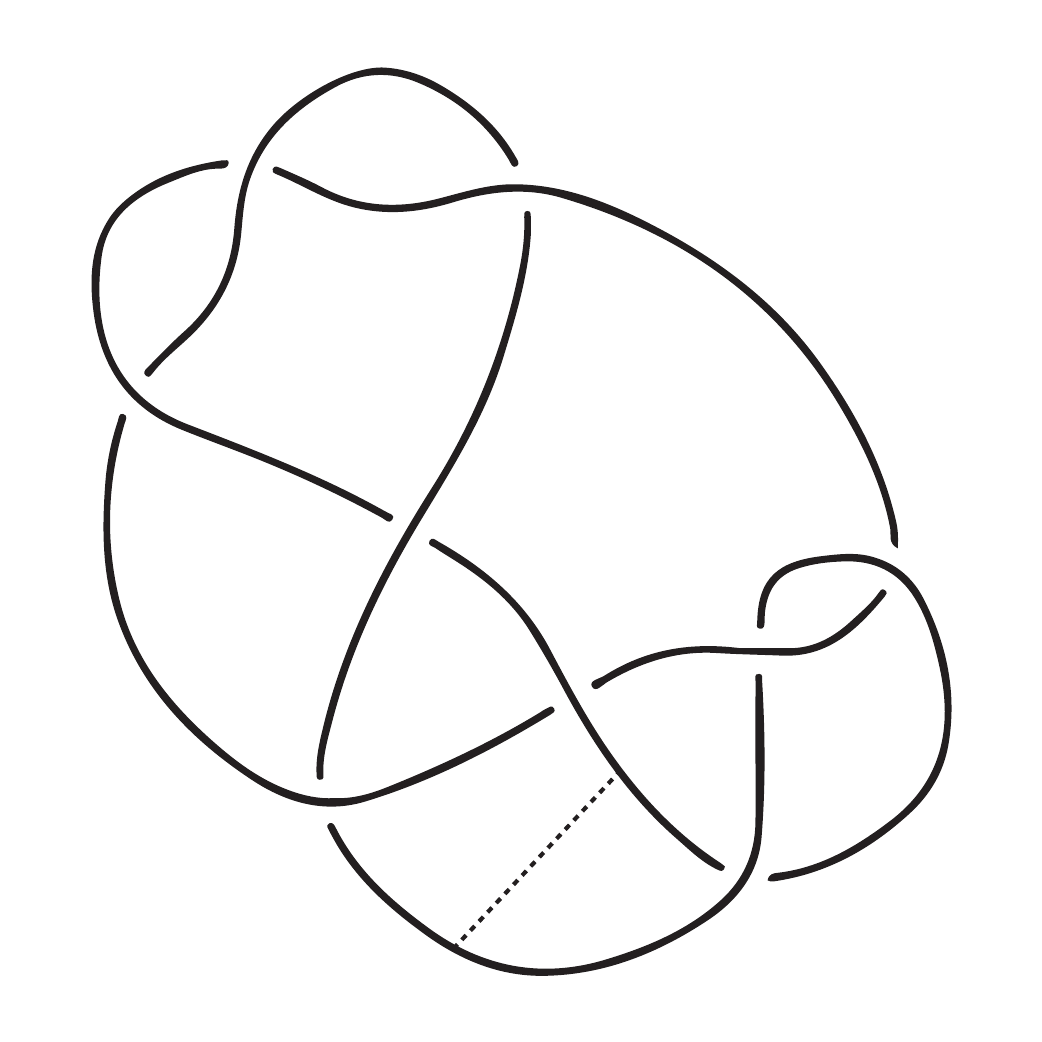}
		\caption{$9_{26}\stackrel{-1\phantom{i}}{\longrightarrow} 8_9$}
		\label{FigureFor9_26}
	\end{subfigure}
	~
	\begin{subfigure}[b]{0.28\textwidth}
		\includegraphics[width=\textwidth]{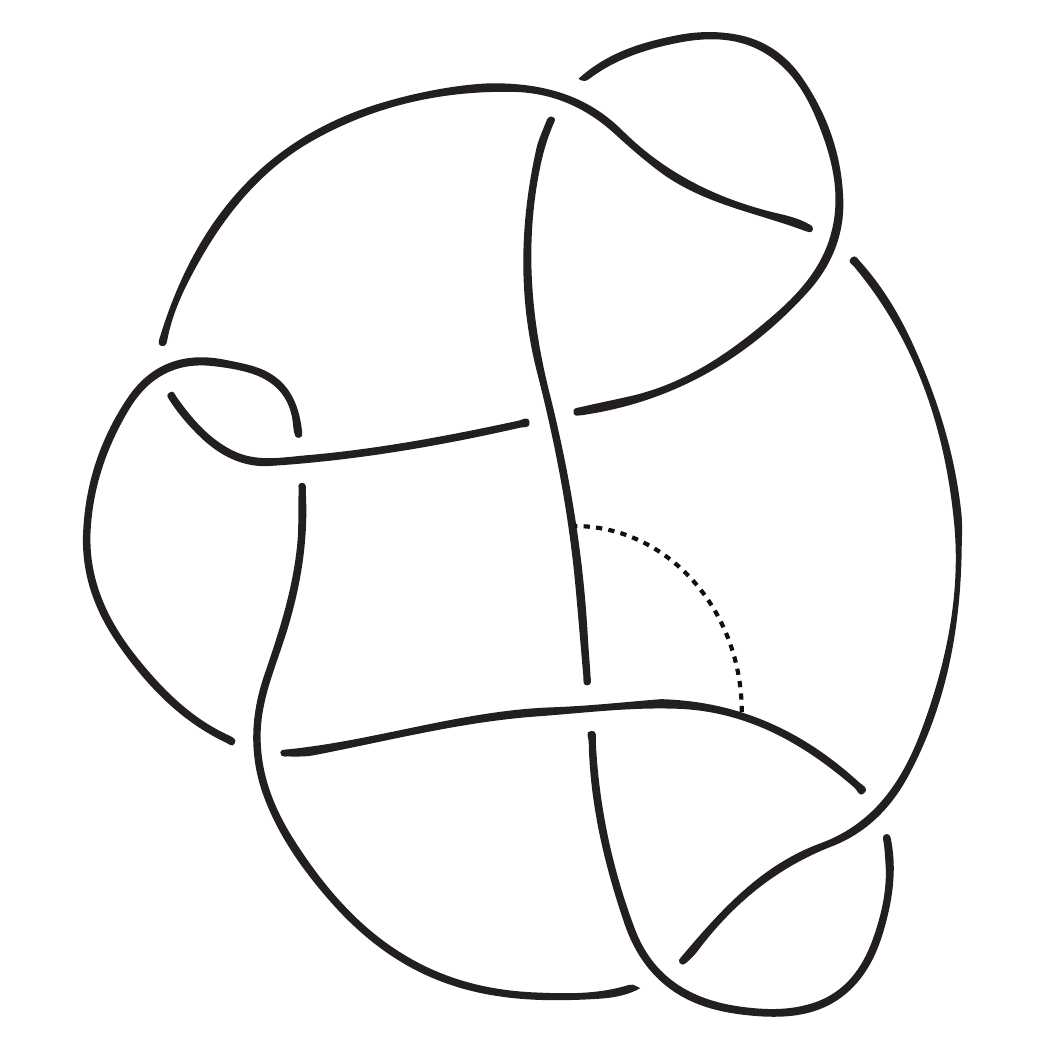}
		\caption{$9_{28}\stackrel{0}{\longrightarrow} 8_8$}
		\label{FigureFor9_28}
	\end{subfigure}
	~
	\begin{subfigure}[b]{0.28\textwidth}
		\includegraphics[width=\textwidth]{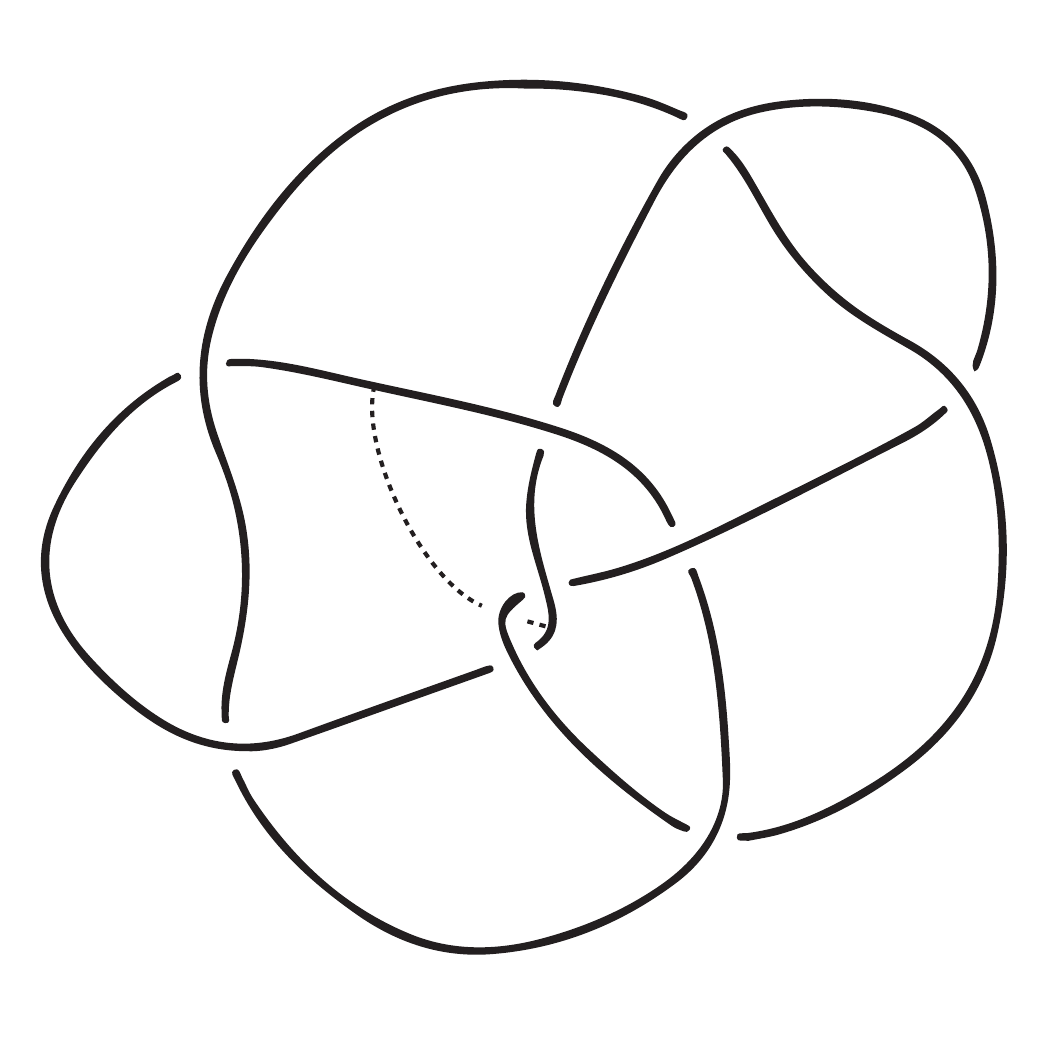}
		\caption{$9_{29}\stackrel{1}{\longrightarrow} 9_{46}$}
		\label{FigureFor9_29}
	\end{subfigure}
	\vskip3mm
	\caption{Non-oriented band moves from the knots $9_{17}$, $9_{19}$, $9_{21}$, $9_{22}$, $9_{23}$, $9_{25}$, $9_{26}$, $9_{28}$, $9_{29}$ to smoothly slice knots.}\label{Figure3ToSliceKnots}
\end{figure}
\begin{figure}[!htbp]
	\centering
	\begin{subfigure}[b]{0.27\textwidth}
		\includegraphics[width=\textwidth]{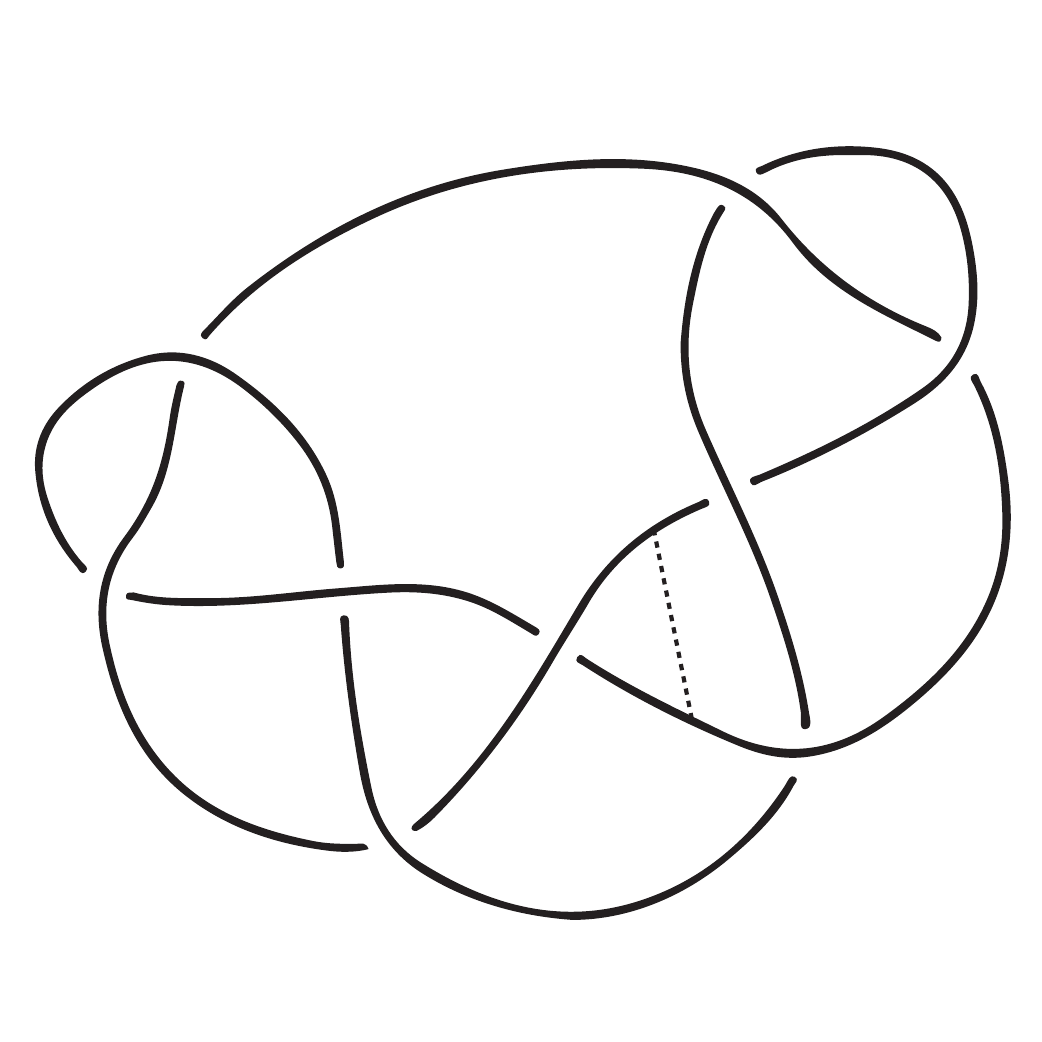}
		\caption{$9_{31}\stackrel{0}{\longrightarrow} 4_1\# 4_1$}
		\label{FigureFor9_31}
	\end{subfigure}
	~
	\begin{subfigure}[b]{0.26\textwidth}
		\includegraphics[width=\textwidth]{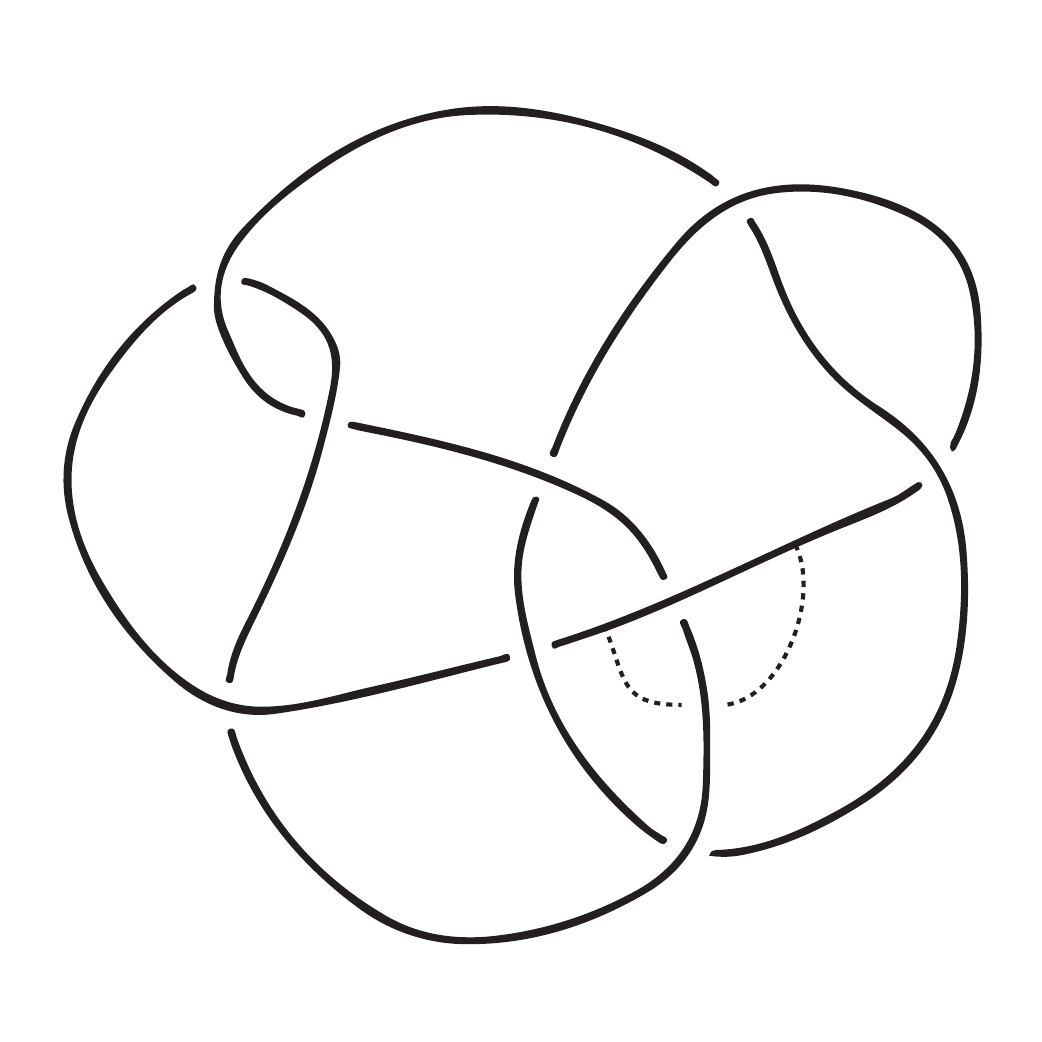}
		\caption{$9_{32}\stackrel{1}{\longrightarrow} 11n_4$}
		\label{FigureFor9_32}
	\end{subfigure}
	~
	\begin{subfigure}[b]{0.27\textwidth}
		\includegraphics[width=\textwidth]{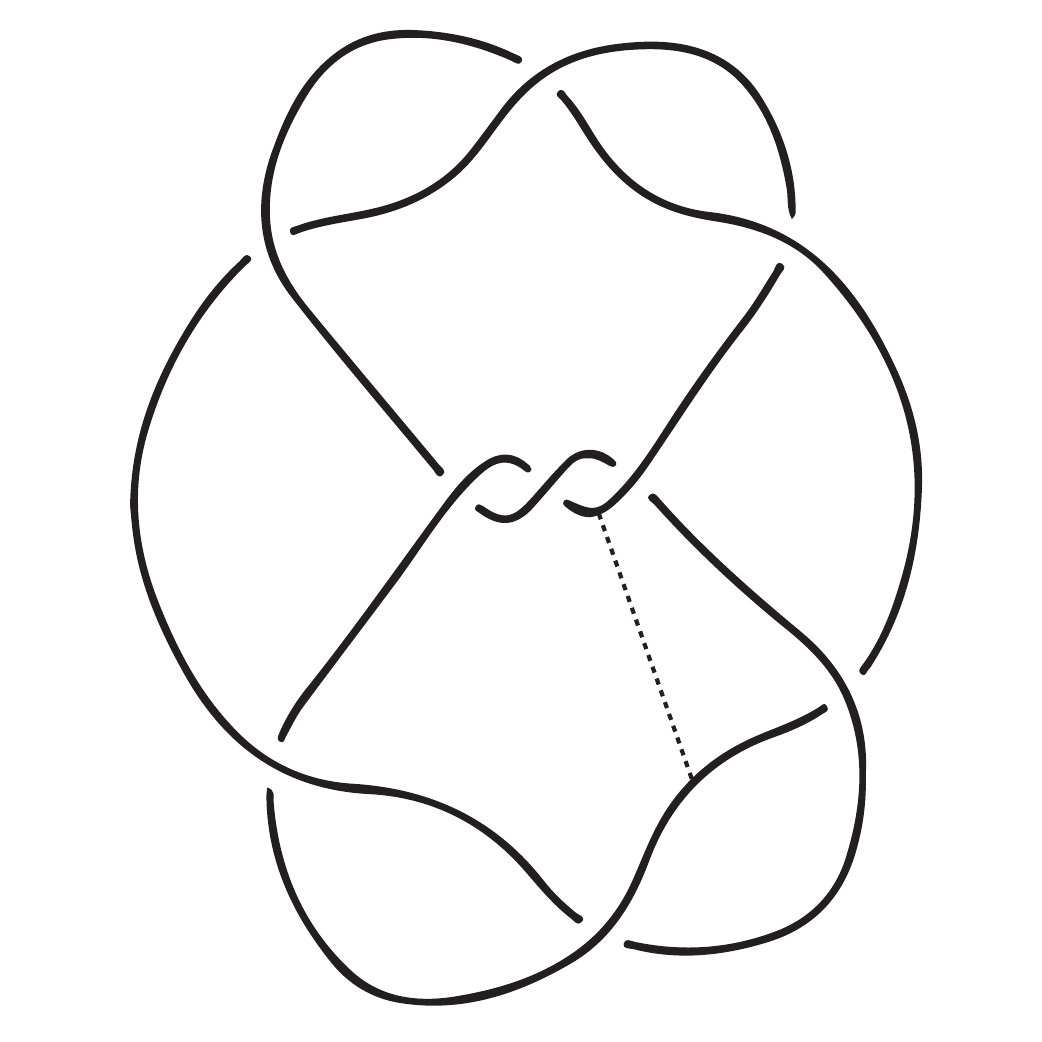}
		\caption{$9_{35}\stackrel{1}{\longrightarrow} 6_1$}
		\label{FigureFor9_35}
	\end{subfigure}
	\vskip3mm
	\begin{subfigure}[b]{0.27\textwidth}
		\includegraphics[width=\textwidth]{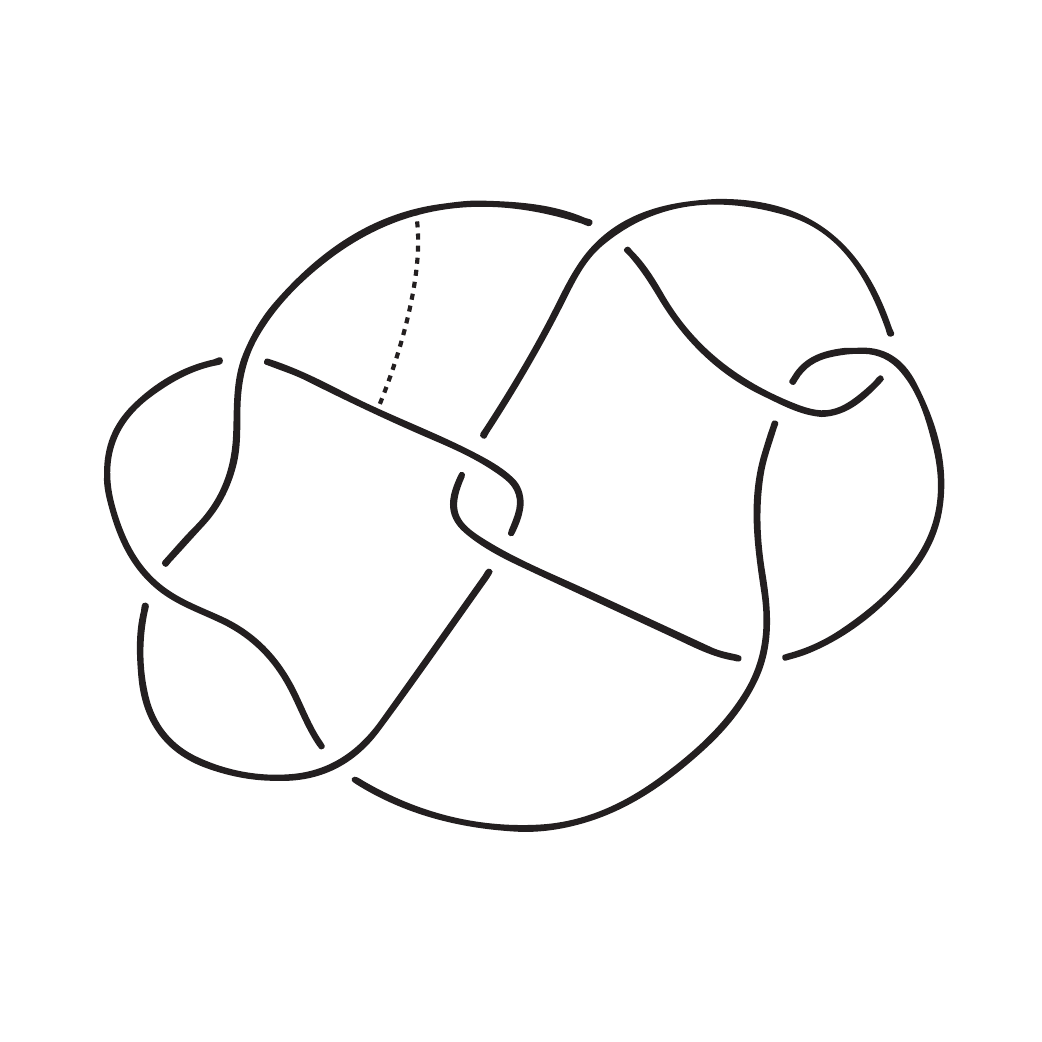}
		\caption{$9_{36}\stackrel{0}{\longrightarrow} 6_1$}
		\label{FigureFor9_36}
	\end{subfigure}
	~
	\begin{subfigure}[b]{0.25\textwidth}
		\includegraphics[width=\textwidth]{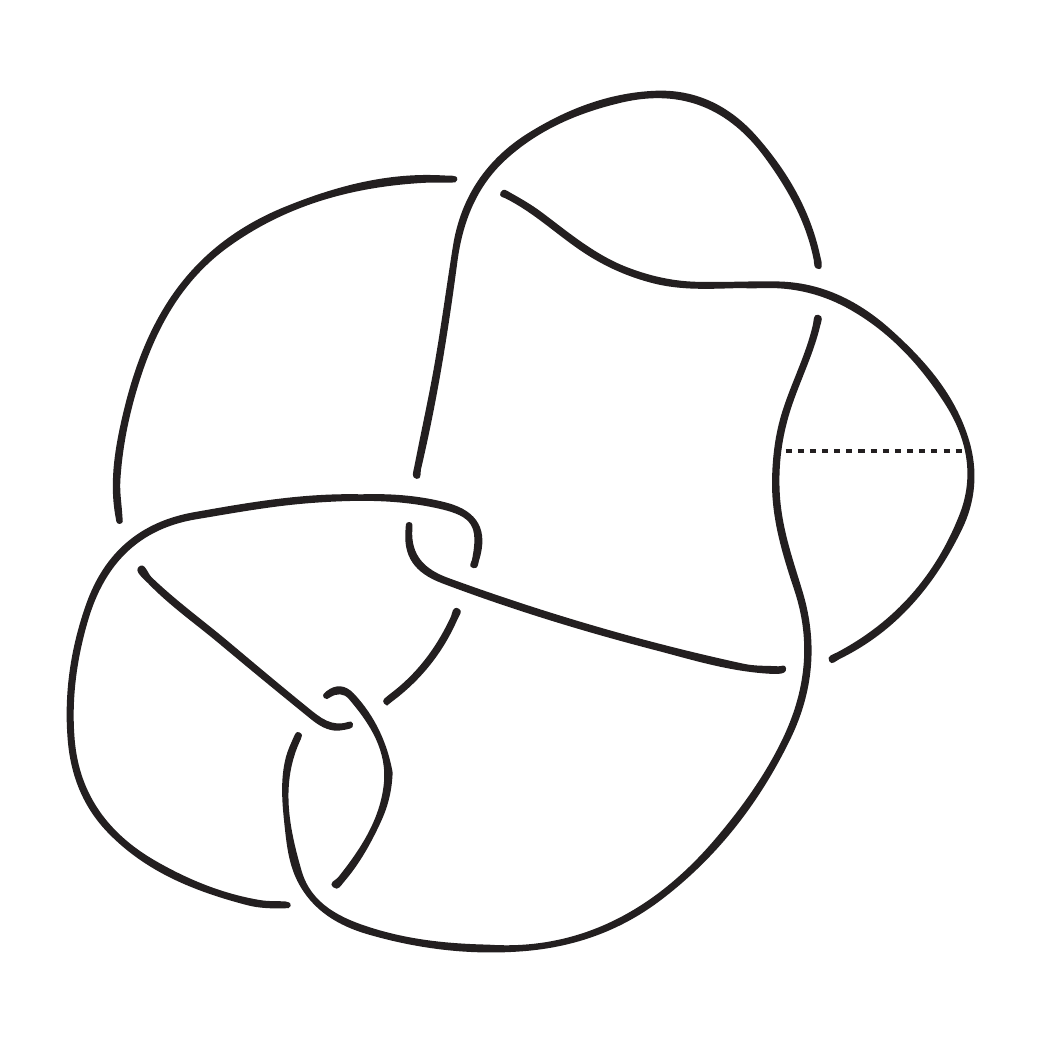}
		\caption{$9_{42}\stackrel{0}{\longrightarrow} 0_1$}
		\label{FigureFor9_42}
	\end{subfigure}
	~
	\begin{subfigure}[b]{0.25\textwidth}
		\includegraphics[width=\textwidth]{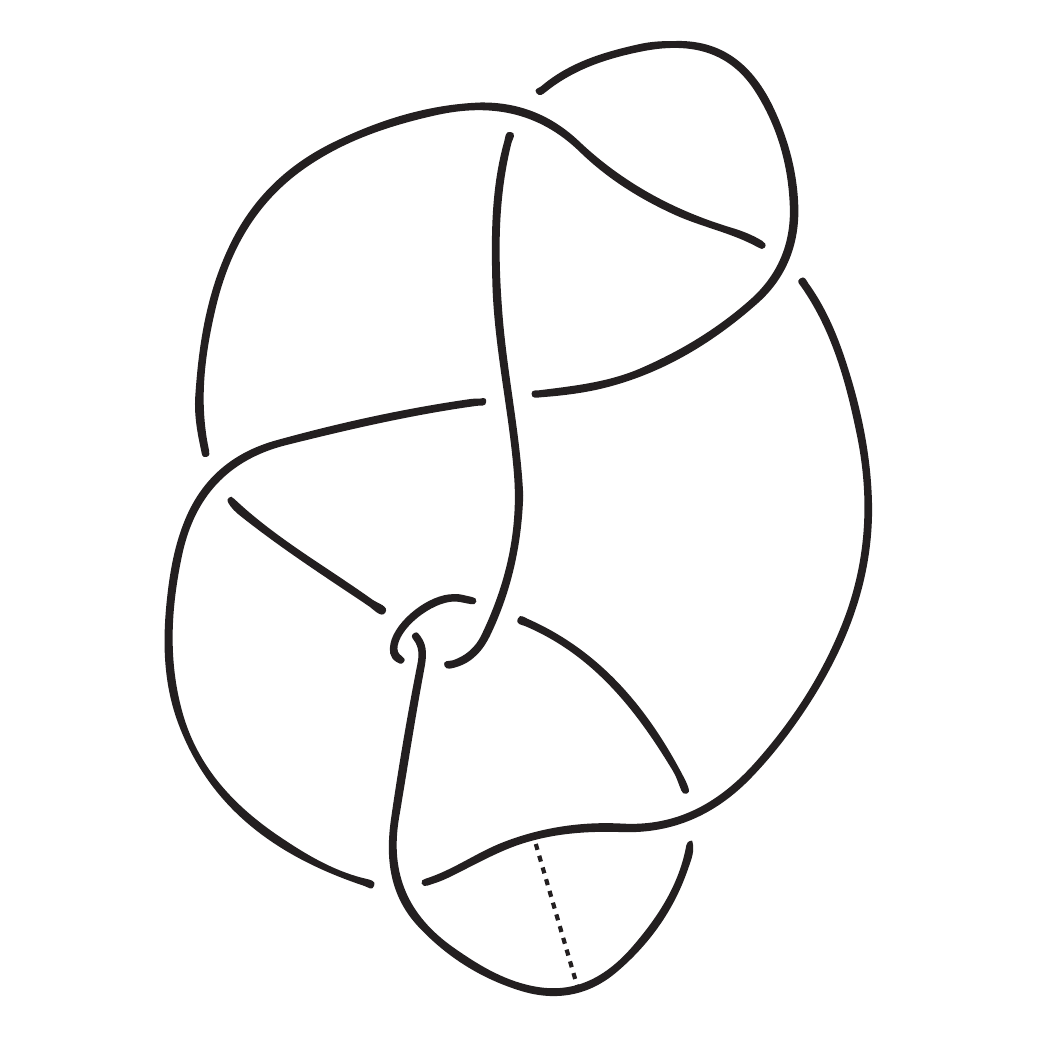}
		\caption{$9_{43}\stackrel{0}{\longrightarrow} 0_{1}$}
		\label{FigureFor9_43}
	\end{subfigure}
	\vskip3mm
	\begin{subfigure}[b]{0.25\textwidth}
		\includegraphics[width=\textwidth]{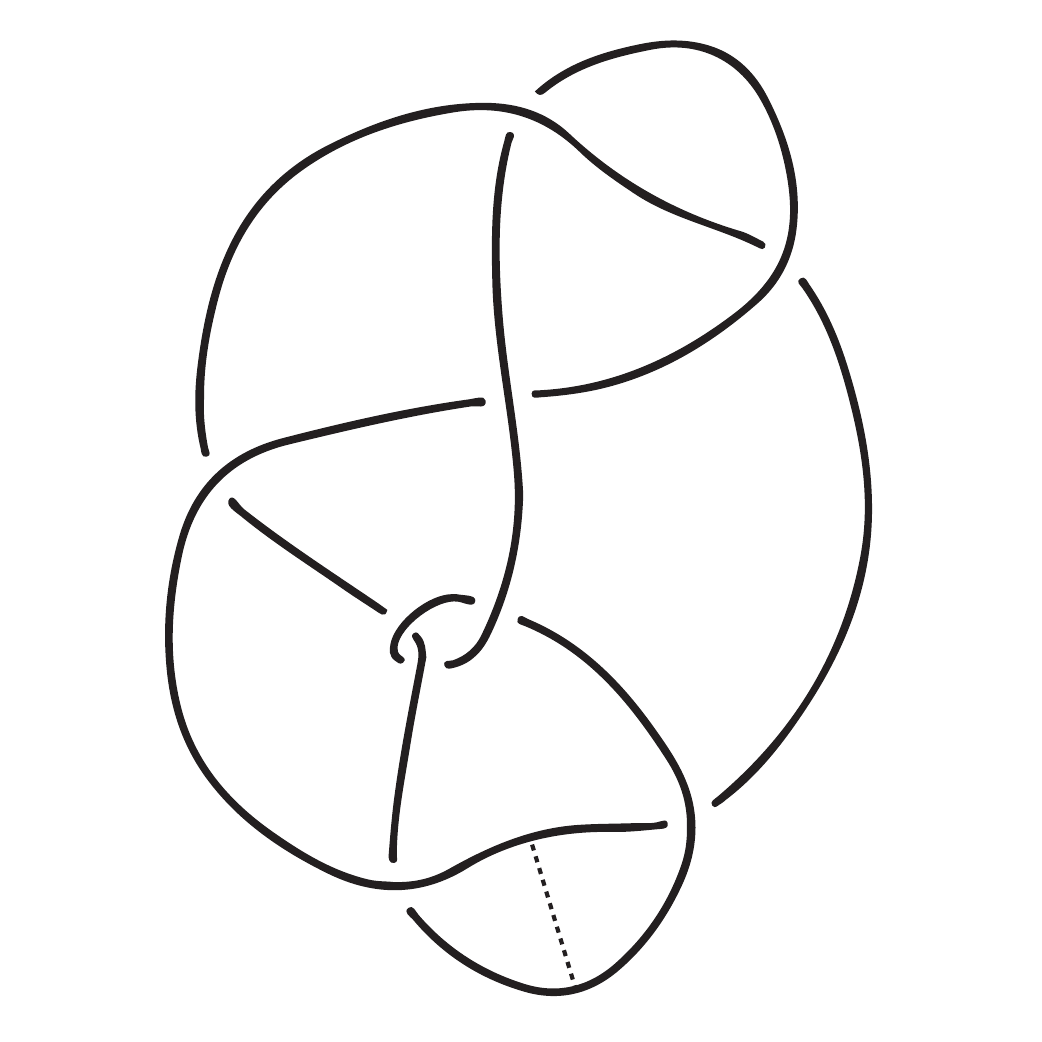}
		\caption{$9_{44}\stackrel{0}{\longrightarrow} 0_{1}$}
		\label{FigureFor9_44}
	\end{subfigure}
	~
	\begin{subfigure}[b]{0.25\textwidth}
		\includegraphics[width=\textwidth]{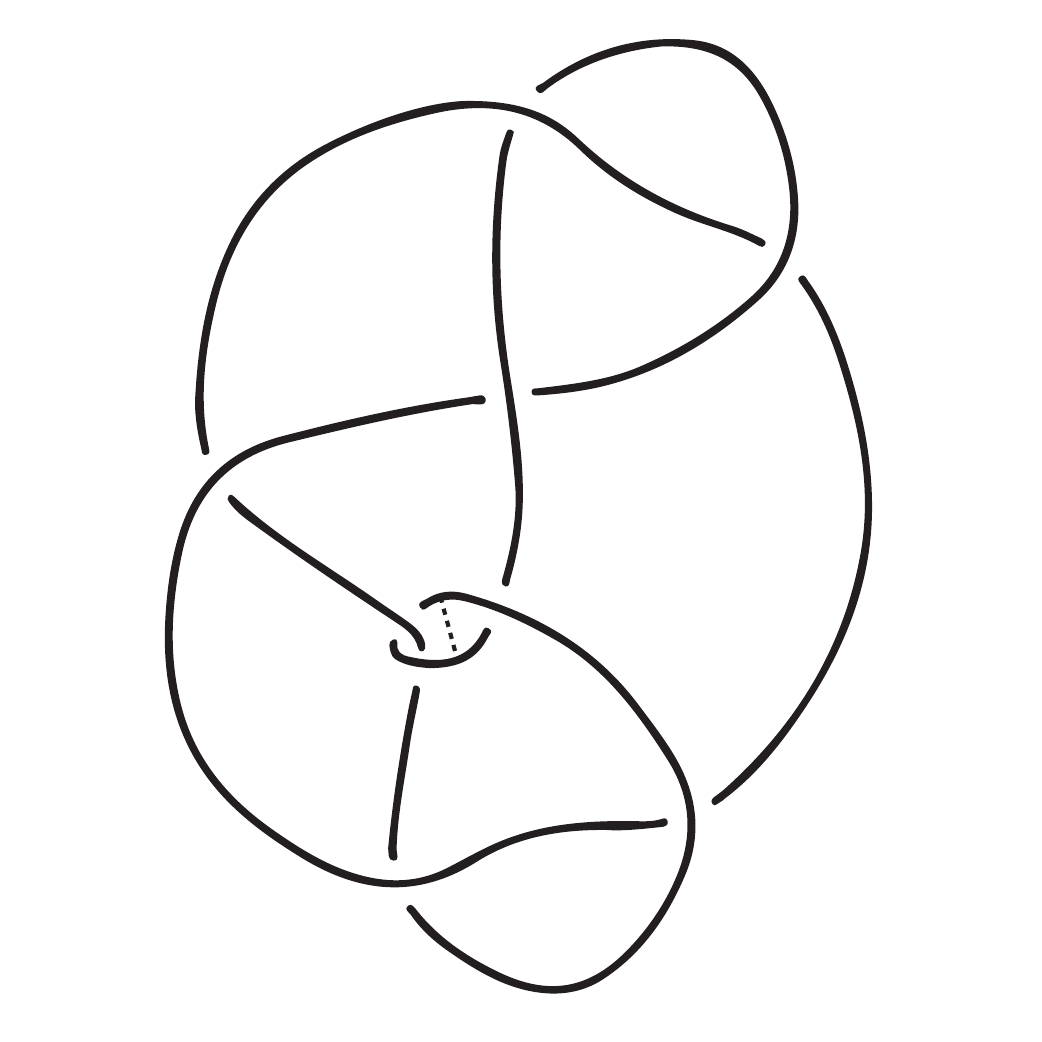}
		\caption{$9_{45}\stackrel{-1\phantom{i}}{\longrightarrow} 10_{137}$}
		\label{FigureFor9_45}
	\end{subfigure}
	~
	\begin{subfigure}[b]{0.24\textwidth}
		\includegraphics[width=\textwidth]{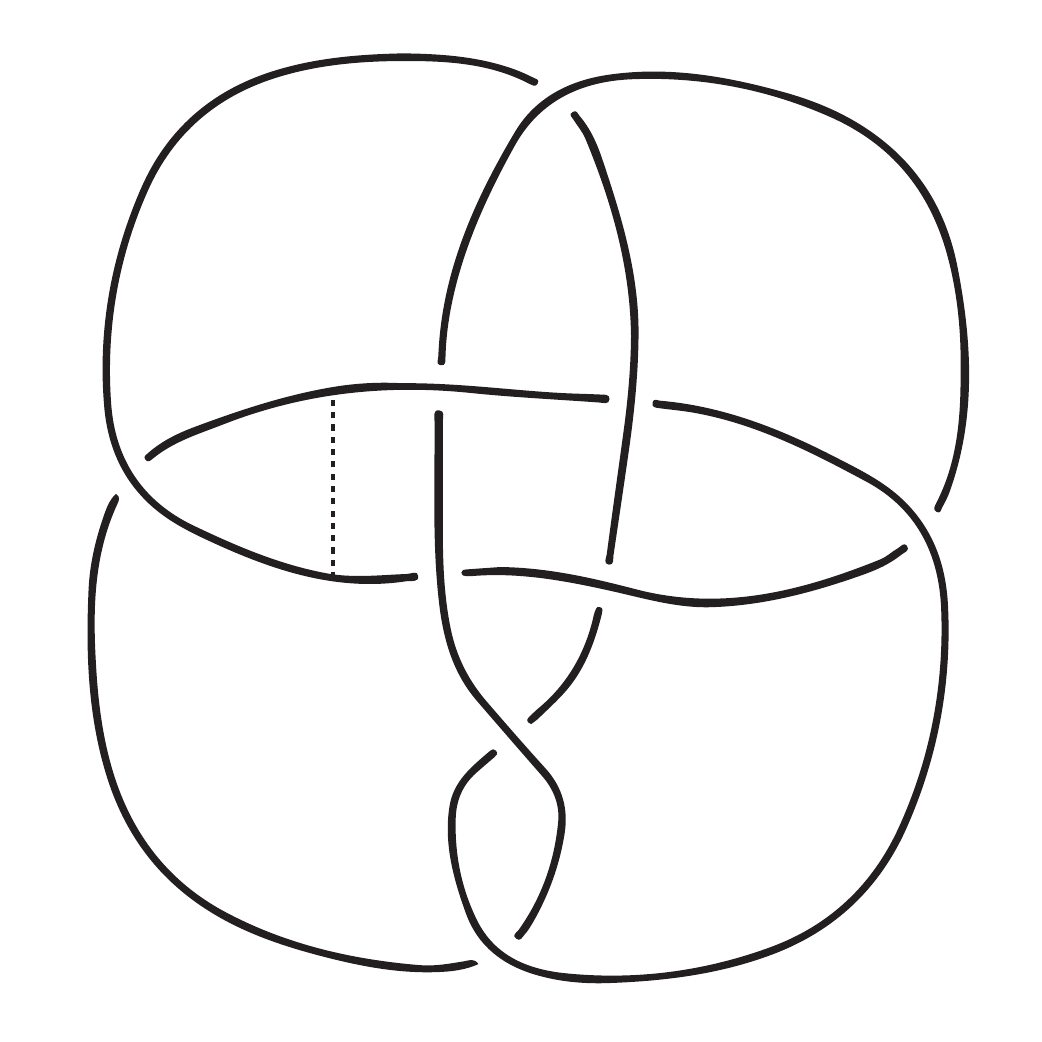}
		\caption{$9_{47}\stackrel{0}{\longrightarrow} 8_{20}$}
		\label{FigureFor9_47}
	\end{subfigure}
	~
	\begin{subfigure}[b]{0.25\textwidth}
		\includegraphics[width=\textwidth]{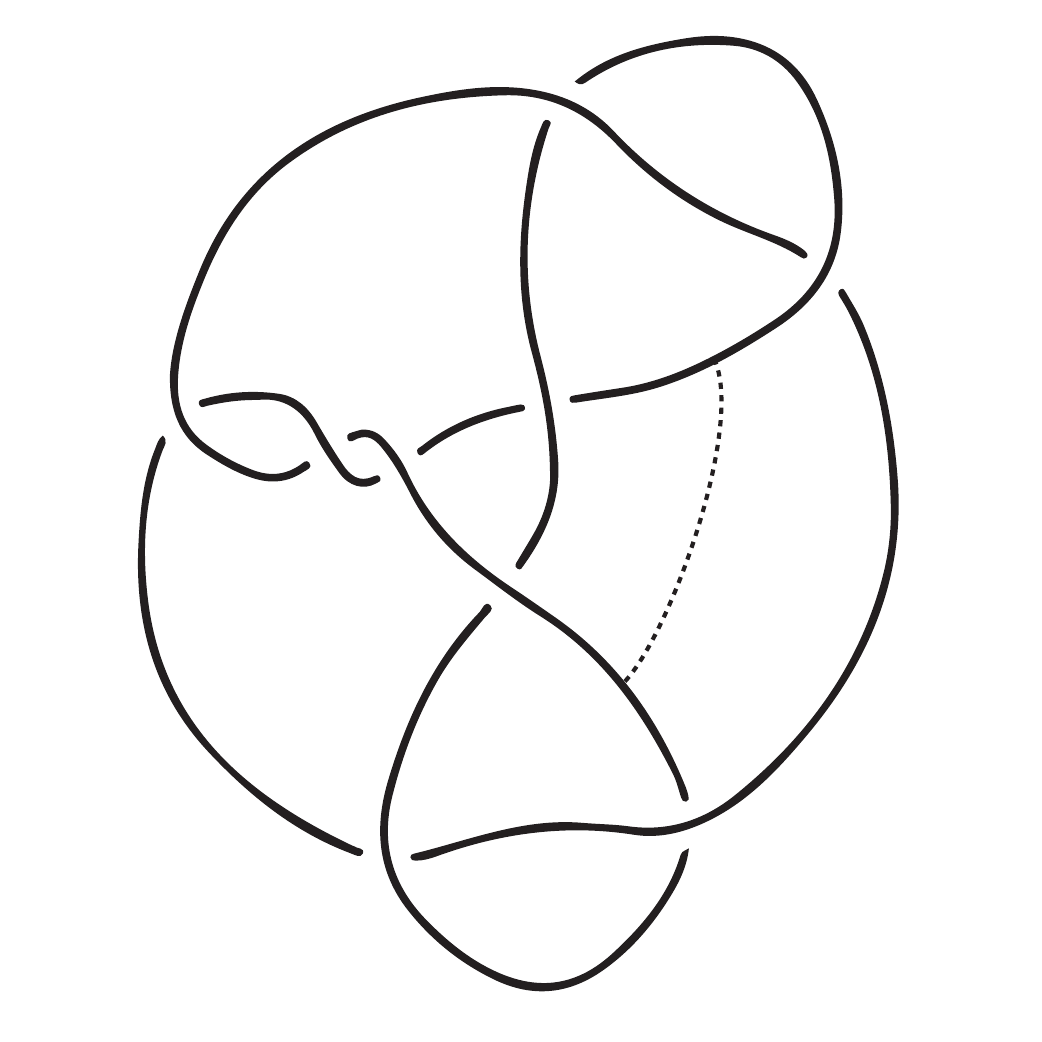}
		\caption{$9_{48}\stackrel{1}{\longrightarrow} 6_{1}$}
		\label{FigureFor9_48}
	\end{subfigure}
	\vskip3mm
	\caption{Non-oriented band moves from the knots $9_{31}$, $9_{32}$, $9_{35}$, $9_{36}$, $9_{42}$, $9_{43}$, $9_{44}$, $9_{45}$, $9_{47}$, $9_{48}$ to smoothly slice knots. }\label{Figure4ToSliceKnots}
	%
\end{figure}
\begin{figure}[!htbp]
	\centering
	\begin{subfigure}[b]{0.27\textwidth}
		\includegraphics[width=\textwidth]{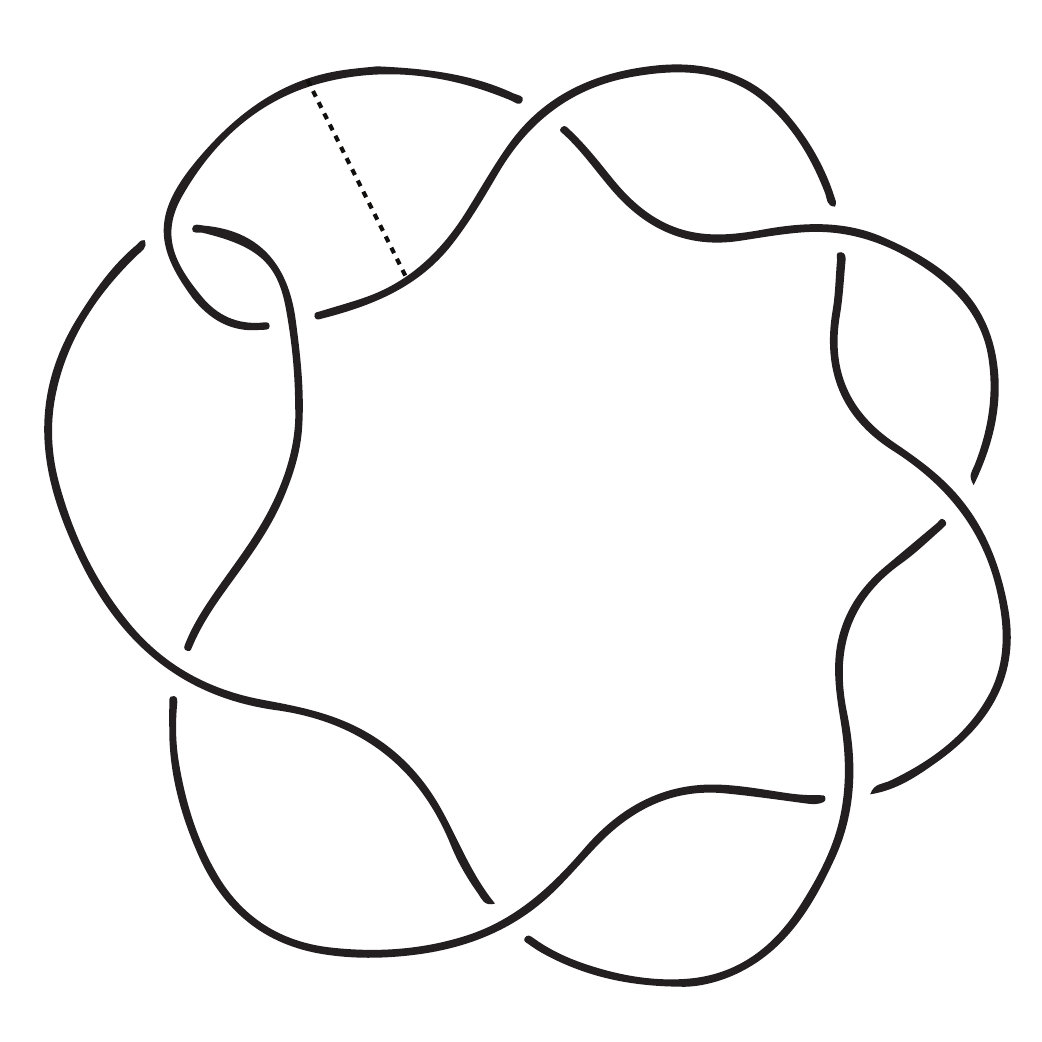}
		\caption{$8_{1}\stackrel{-1\phantom{i}}{\longrightarrow} 7_2$}
		\label{FigureFor8_1}
	\end{subfigure}
	~ 
	\begin{subfigure}[b]{0.28\textwidth}
		\includegraphics[width=\textwidth]{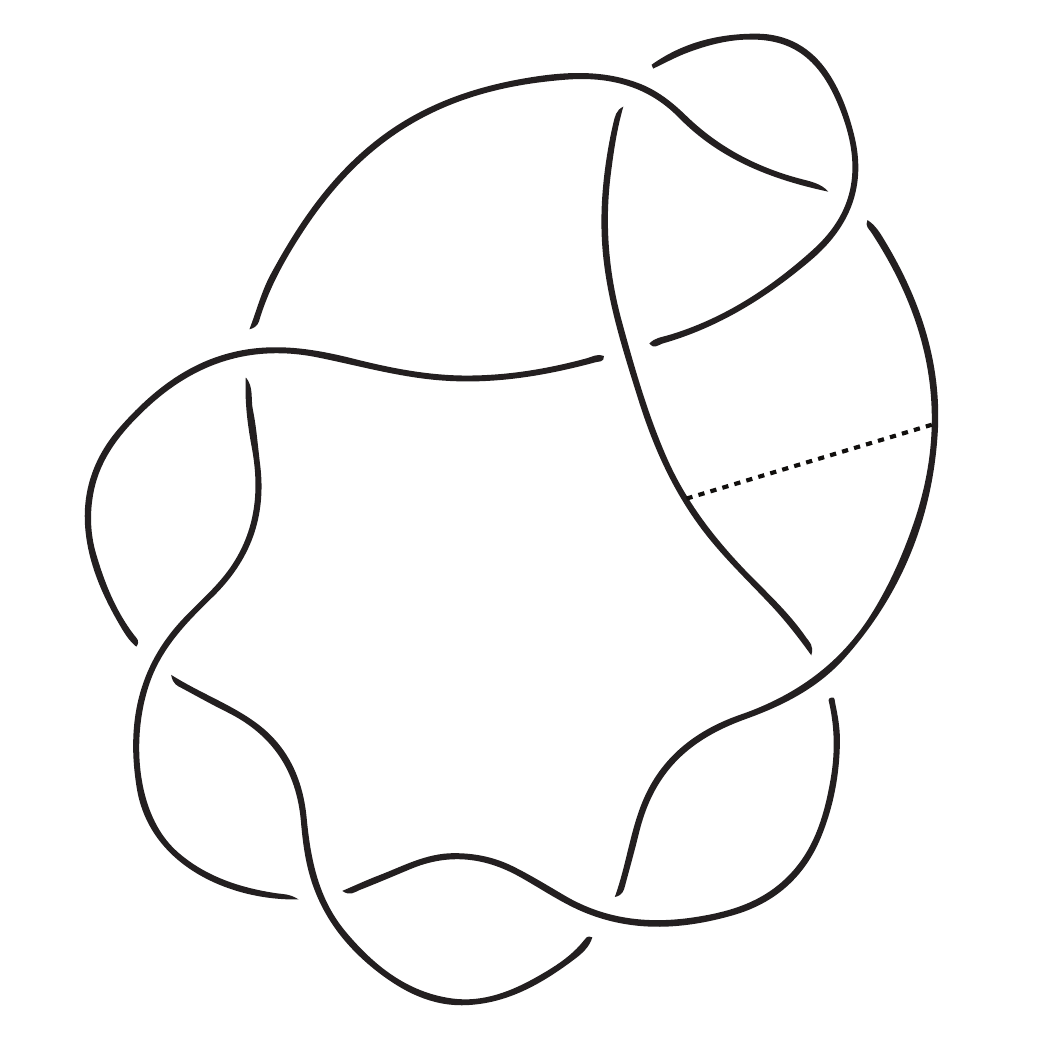}
		\caption{$8_{2}\stackrel{0}{\longrightarrow} 3_1$}
		\label{FigureFor8_2}
	\end{subfigure}
	~ 
	\begin{subfigure}[b]{0.28\textwidth}
		\includegraphics[width=\textwidth]{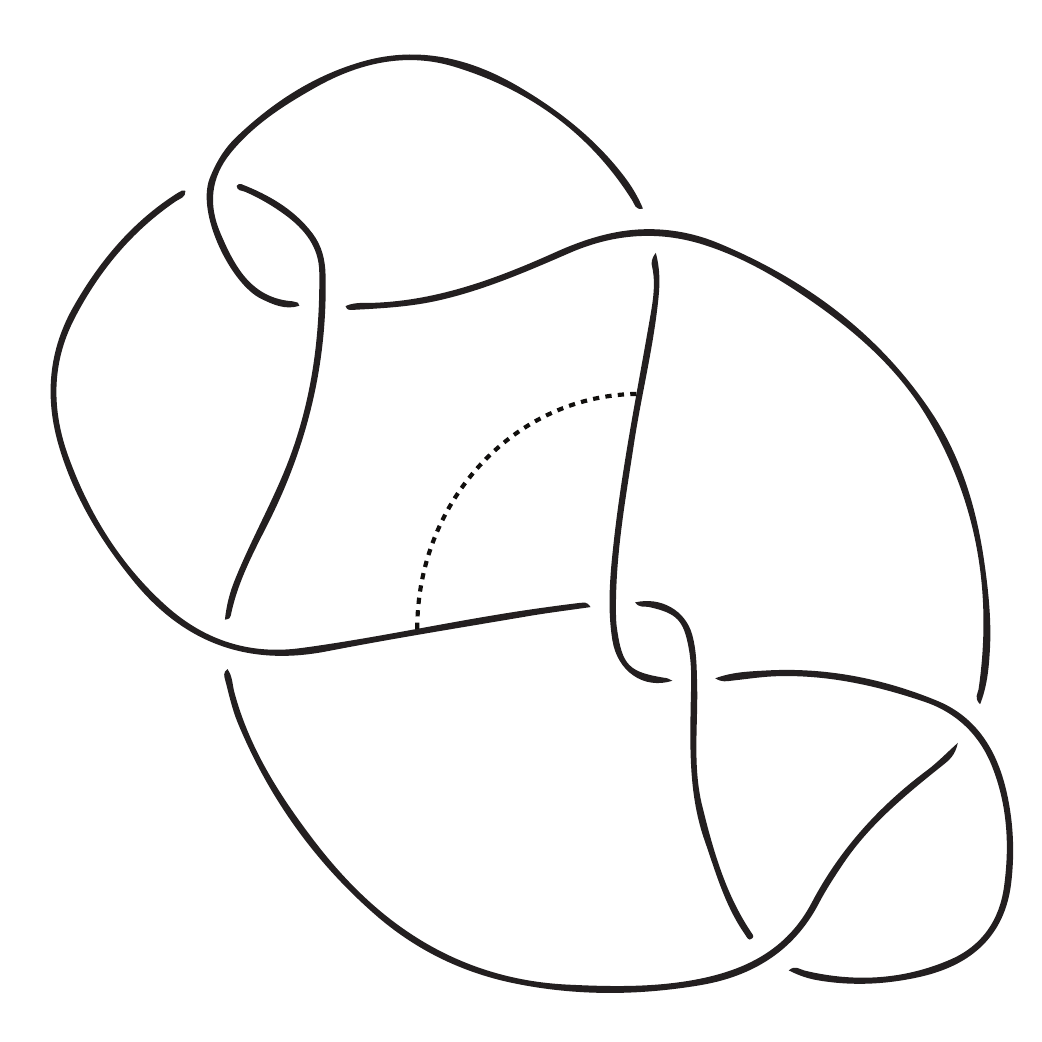}
		\caption{$8_{12}\stackrel{1}{\longrightarrow} 7_6$}
		\label{FigureFor8_12}
	\end{subfigure}
	\vskip3mm
	\begin{subfigure}[b]{0.27\textwidth}
		\includegraphics[width=\textwidth]{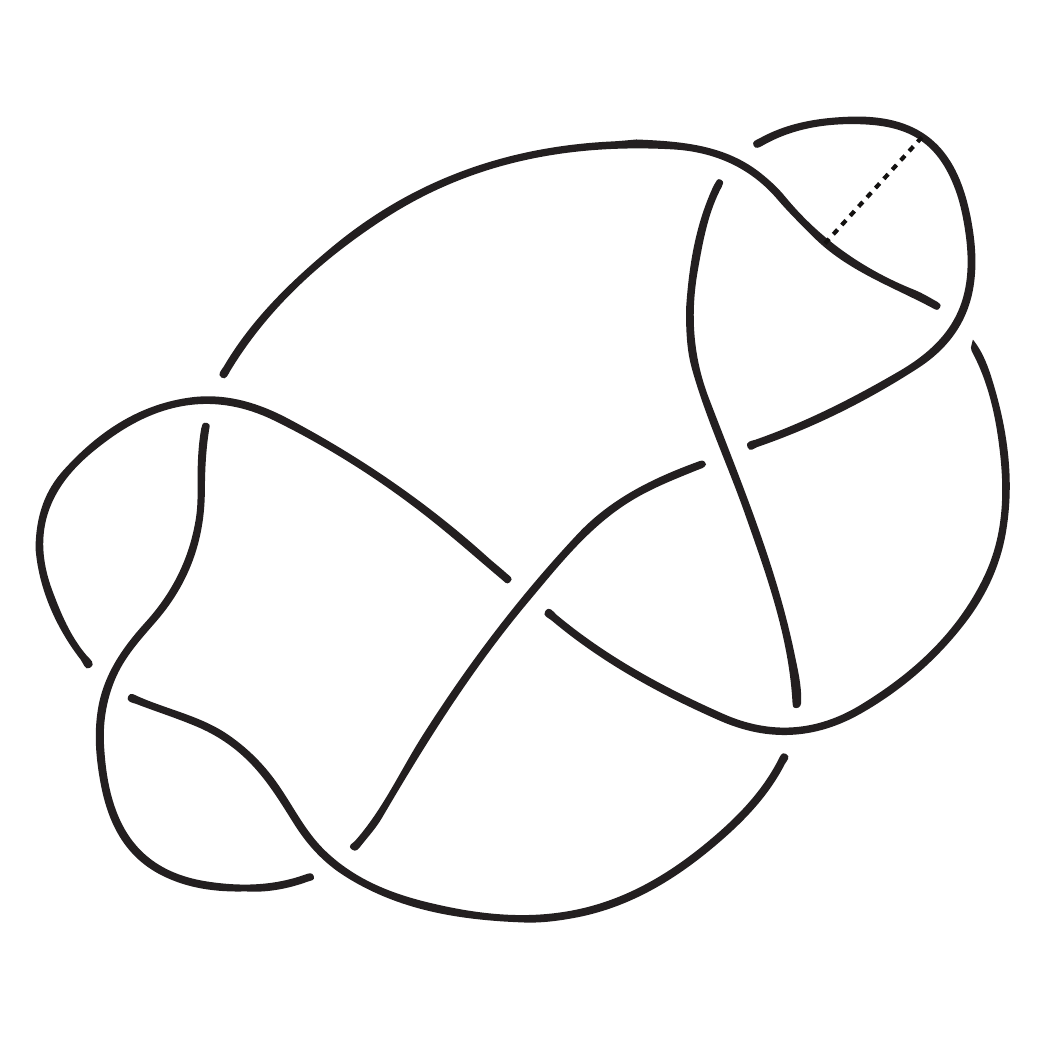}
		\caption{$8_{13}\stackrel{0}{\longrightarrow} 6_2$}
		\label{FigureFor8_13}
	\end{subfigure}
	~ 
	\begin{subfigure}[b]{0.3\textwidth}
		\includegraphics[width=\textwidth]{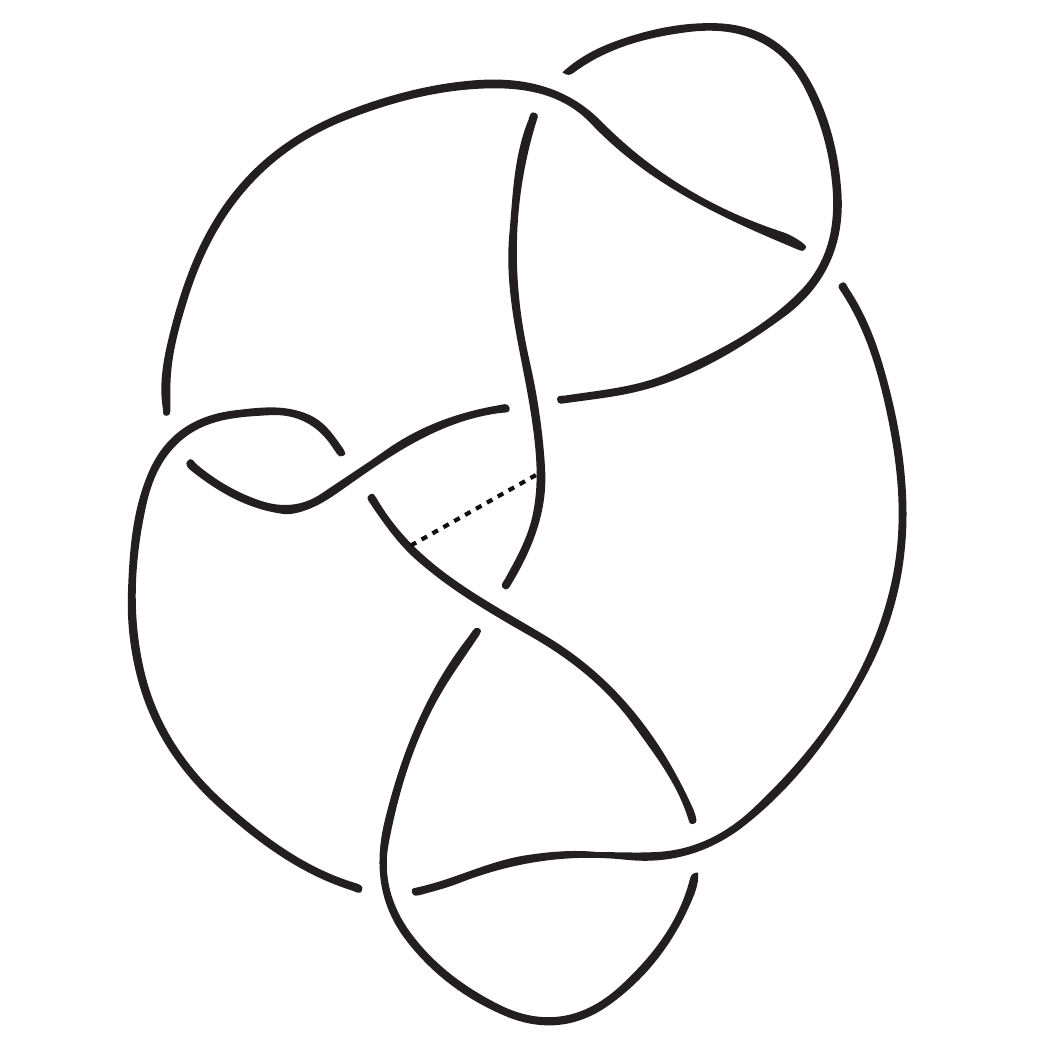}
		\caption{$8_{15}\stackrel{-1\phantom{i}}{\longrightarrow} 7_6$}
		\label{FigureFor8_15}
	\end{subfigure}
	~
	\begin{subfigure}[b]{0.27\textwidth}
		\includegraphics[width=\textwidth]{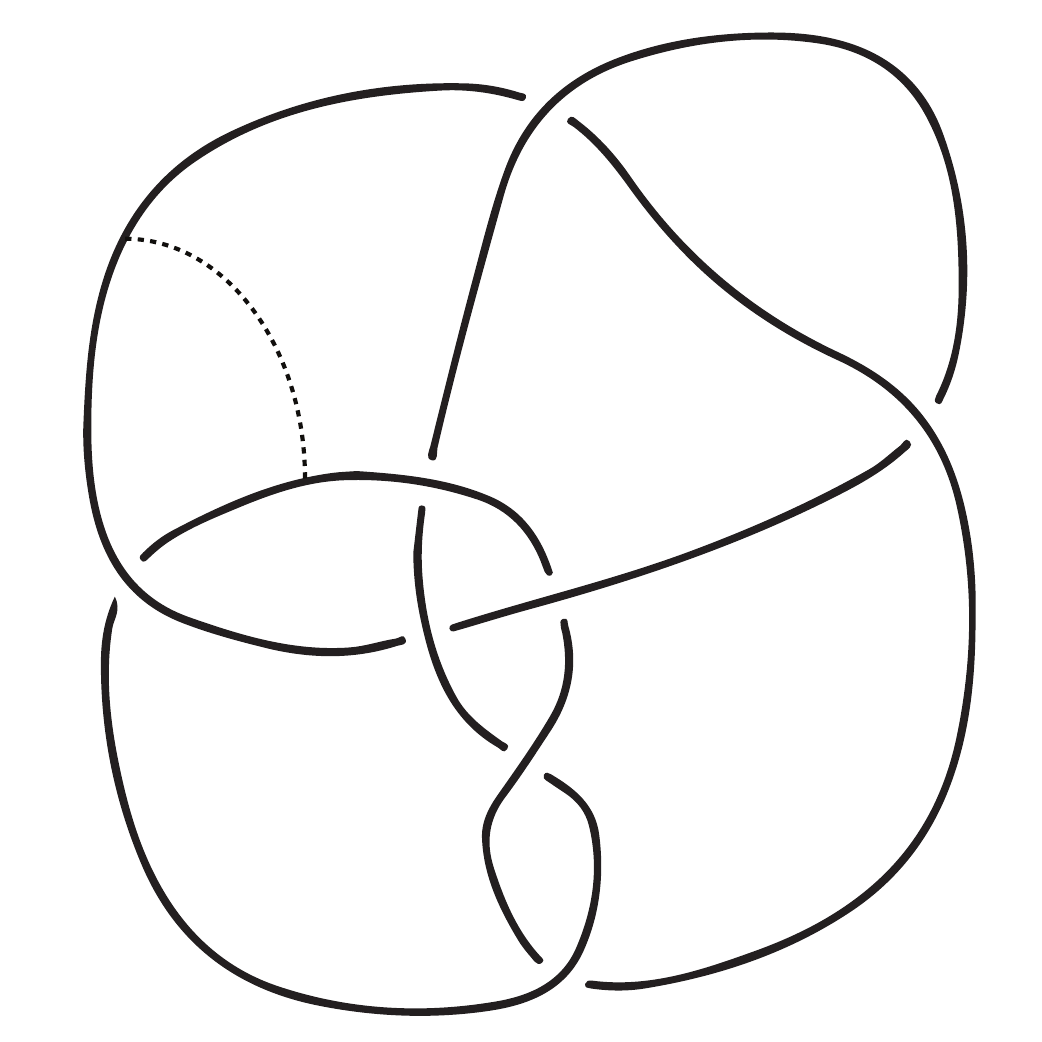}
		\caption{$8_{17}\stackrel{-1\phantom{i}}{\longrightarrow} 7_6$}
		\label{FigureFor8_17}
	\end{subfigure}
	\vskip3mm
	\begin{subfigure}[b]{0.3\textwidth}
		\includegraphics[width=\textwidth]{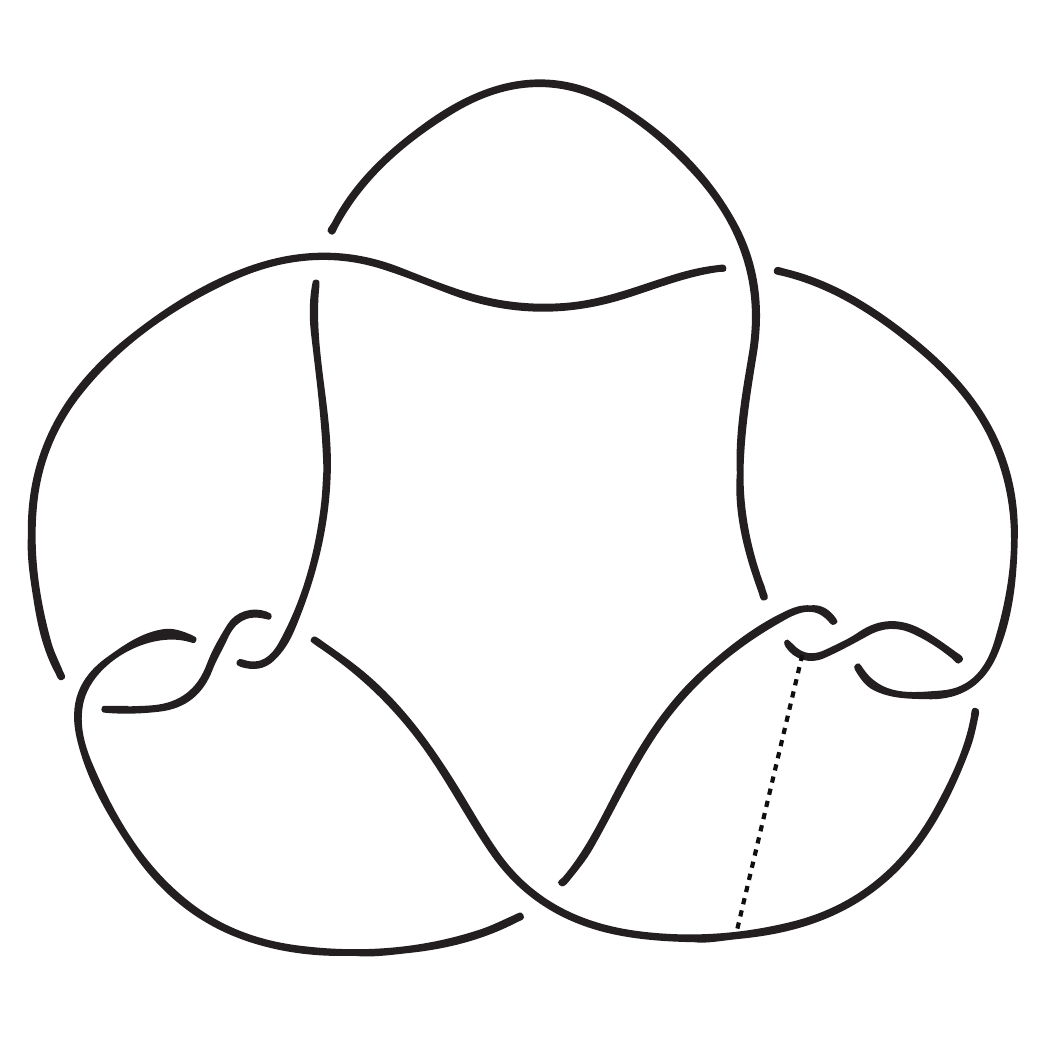}
		\caption{$9_{10}\stackrel{1}{\longrightarrow} 5_2$}
		\label{FigureFor9_10}
	\end{subfigure}
	~ 
	\begin{subfigure}[b]{0.28\textwidth}
		\includegraphics[width=\textwidth]{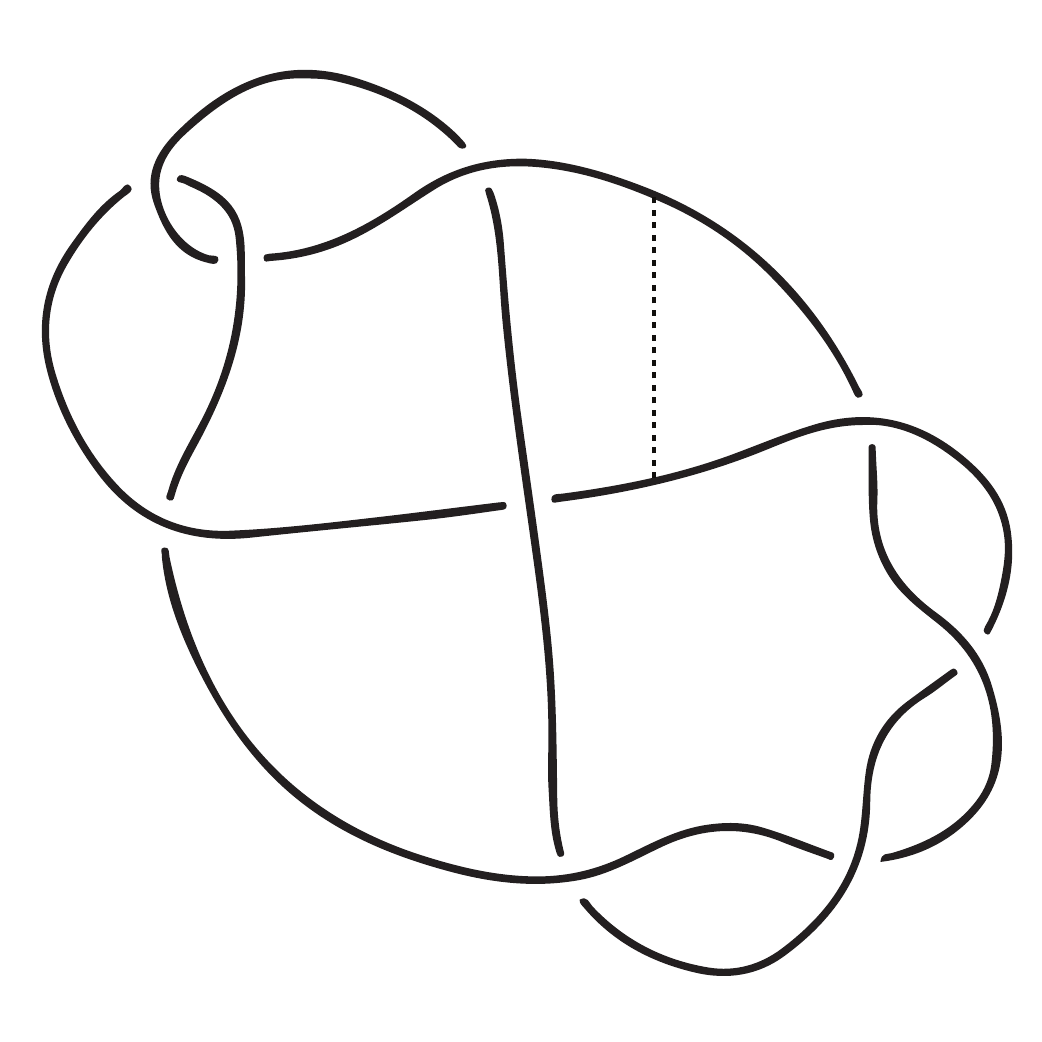}
		\caption{$9_{11}\stackrel{0}{\longrightarrow} 5_2$}
		\label{FigureFor9_11}
	\end{subfigure}
	~ 
	\begin{subfigure}[b]{0.3\textwidth}
		\includegraphics[width=\textwidth]{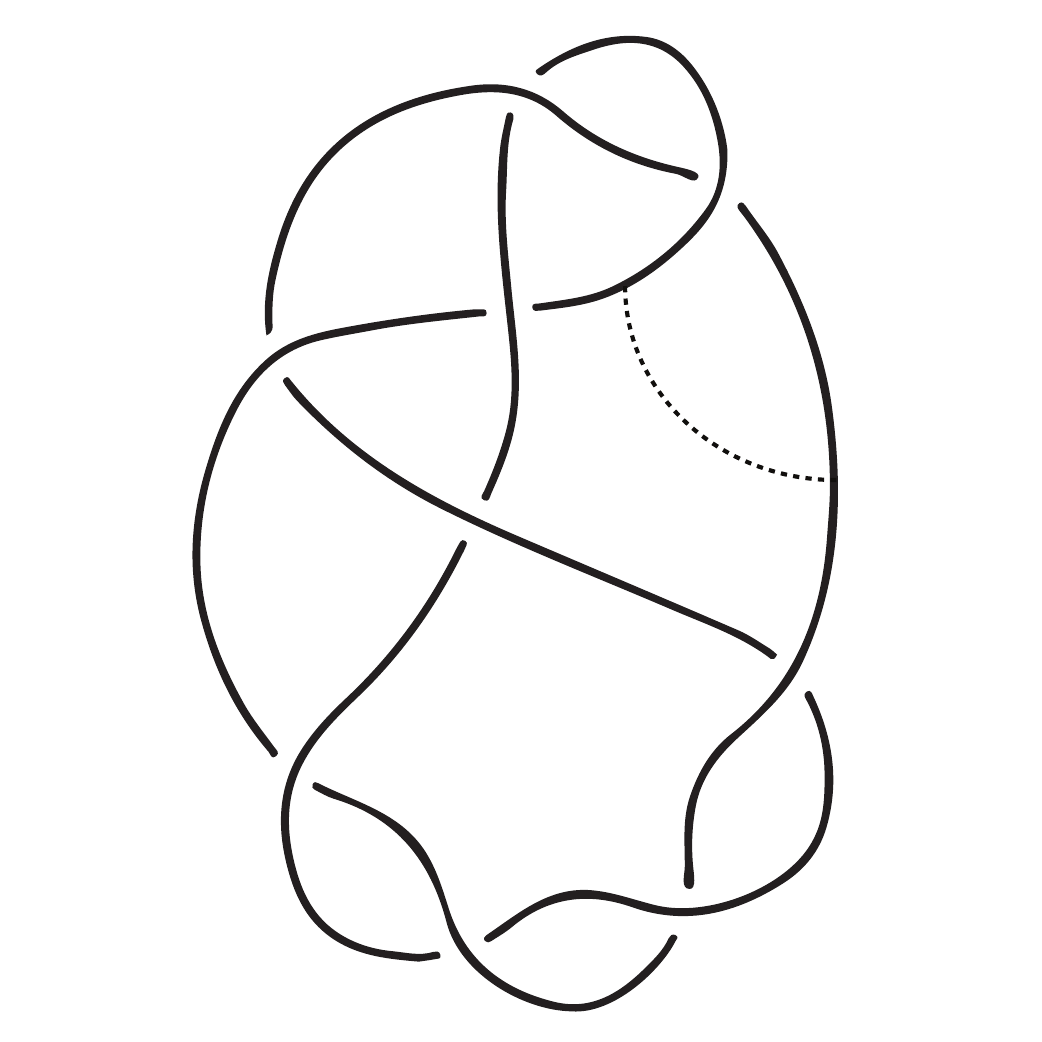}
		\caption{$9_{14}\stackrel{1}{\longrightarrow} 8_7$}
		\label{FigureFor9_14}
	\end{subfigure}
	\vskip3mm
	\caption{Non-oriented band moves from the knots $8_{1}$, $8_{2}$, $8_{12}$, $8_{13}$, $8_{15}$, $8_{17}$, $9_{10}$, $9_{11}$, $9_{14}$ to  knots with $\gamma_4$ equal to 1.}\label{8_1, 8_2, 8_12,8_13,8_15,8_17,9_10,9_11,9_14}
\end{figure}
\begin{figure}[!htbp]
	\centering
	\begin{subfigure}[b]{0.29\textwidth}
		\includegraphics[width=\textwidth]{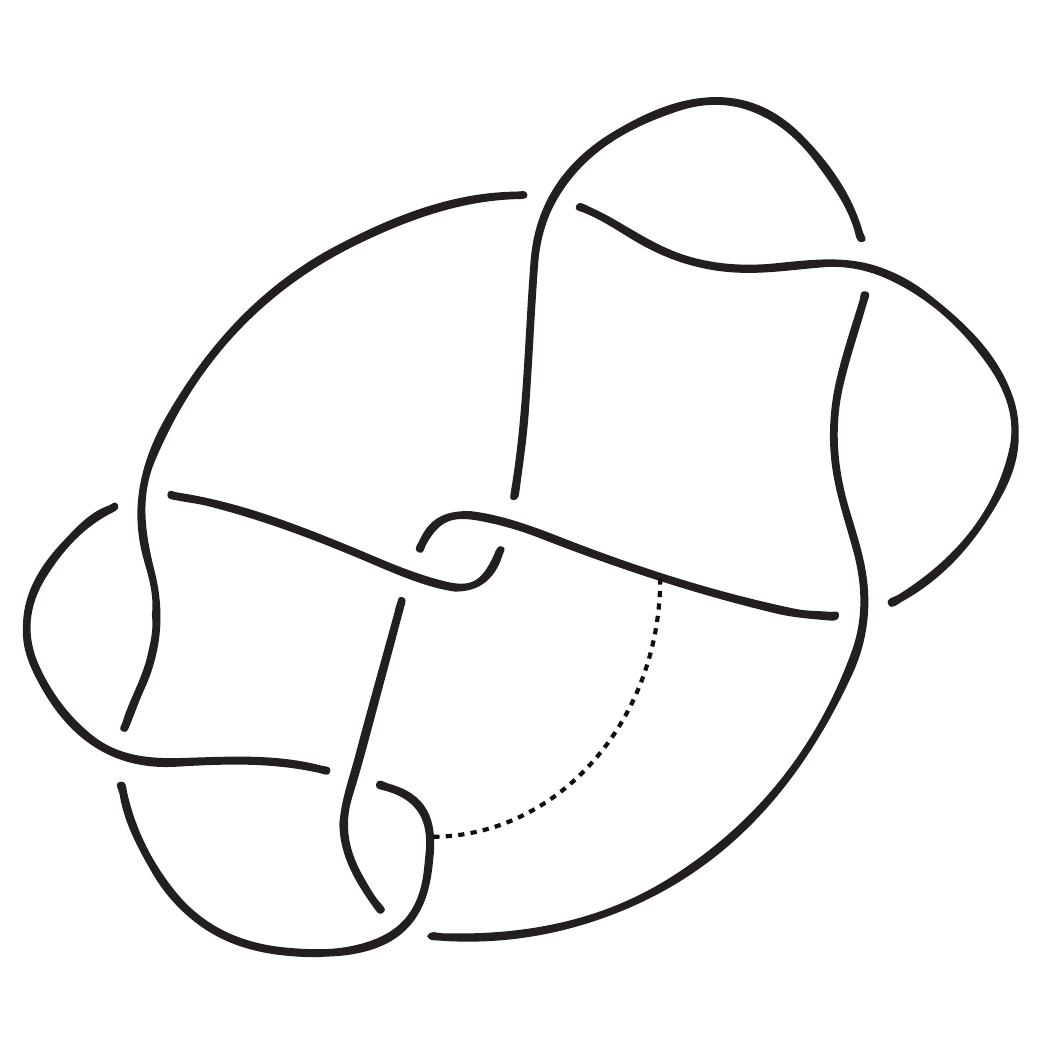}
		\caption{$9_{18}\stackrel{-1\phantom{i}}{\longrightarrow} 5_2$}
		\label{FigureFor9_18}
	\end{subfigure}
	~ 
	\begin{subfigure}[b]{0.29\textwidth}
		\includegraphics[width=\textwidth]{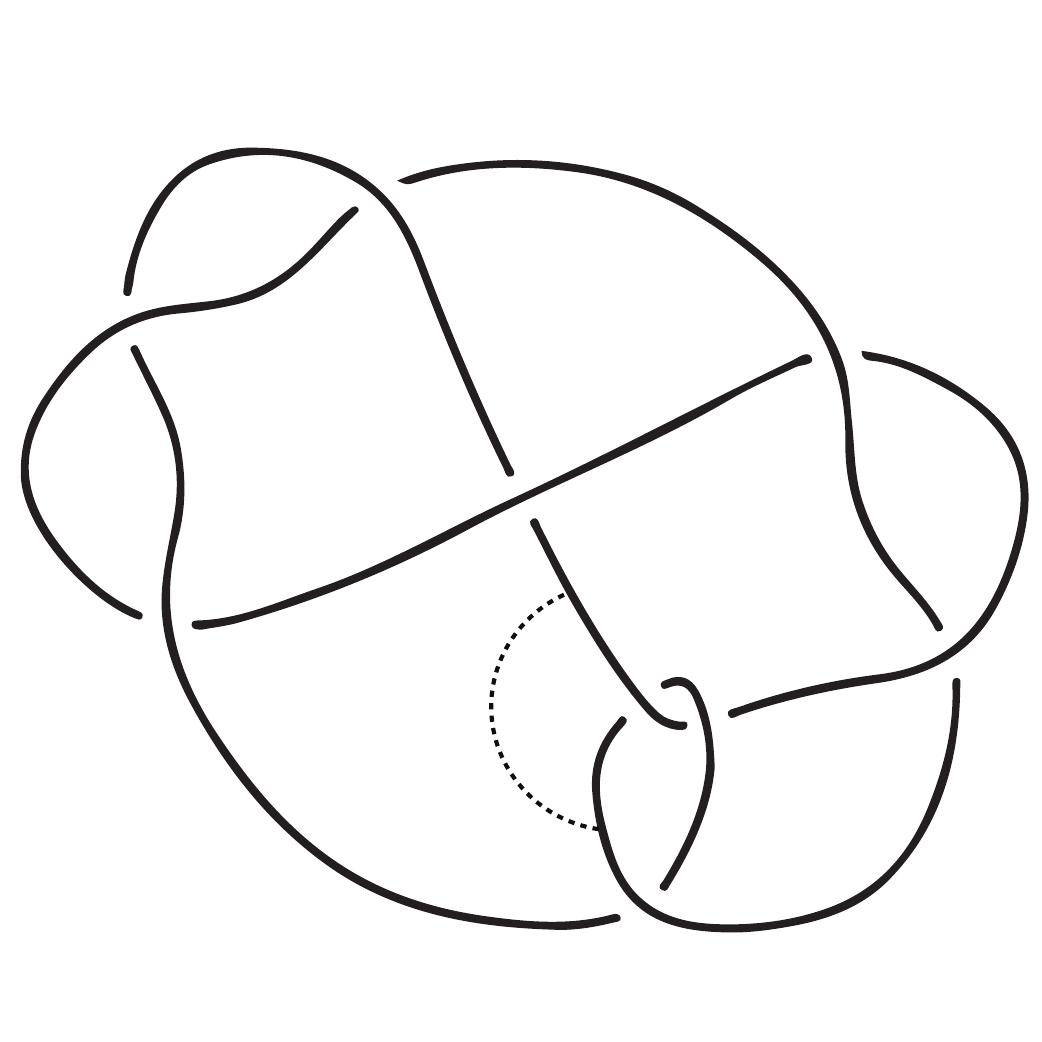}
		\caption{$9_{20}\stackrel{0}{\longrightarrow} 7_4$}
		\label{FigureFor9_20}
	\end{subfigure}
	~ 
	\begin{subfigure}[b]{0.28\textwidth}
		\includegraphics[width=\textwidth]{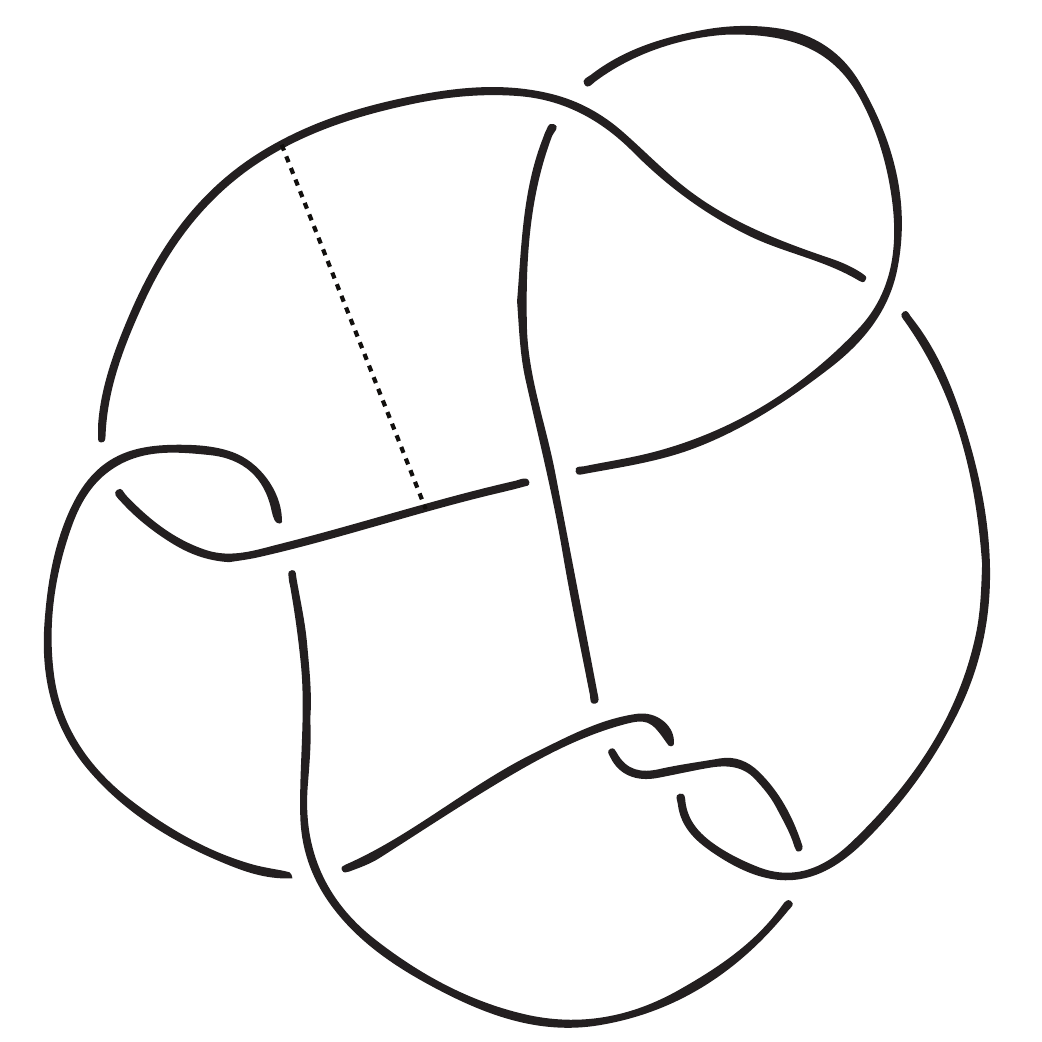}
		\caption{$9_{24}\stackrel{1}{\longrightarrow} 8_{10}$}
		\label{FigureFor9_24}
	\end{subfigure}
	\vskip3mm
	\begin{subfigure}[b]{0.3\textwidth}
		\includegraphics[width=\textwidth]{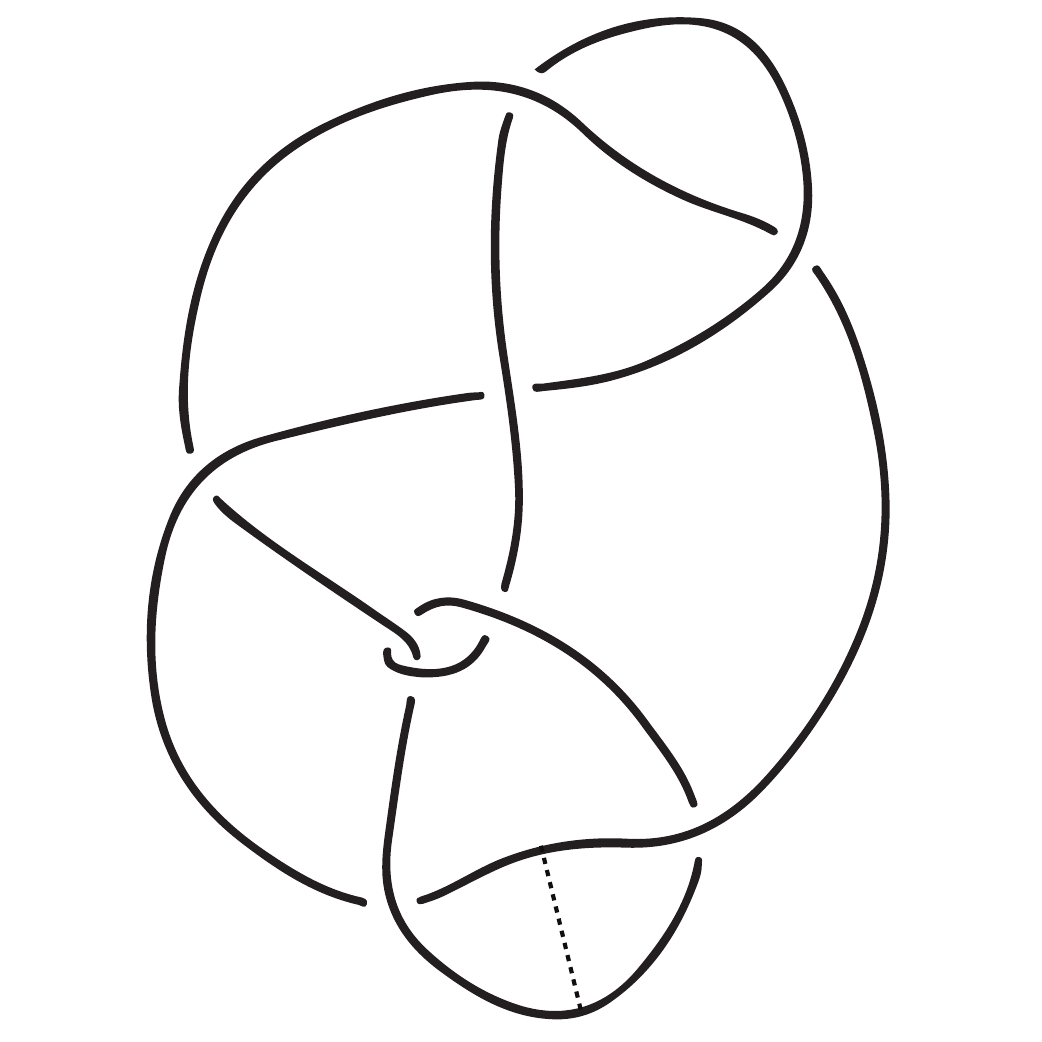}
		\caption{$9_{30}\stackrel{0}{\longrightarrow} 7_6$}
		\label{FigureFor9_30}
	\end{subfigure}
	~ 
	\begin{subfigure}[b]{0.3\textwidth}
		\includegraphics[width=\textwidth]{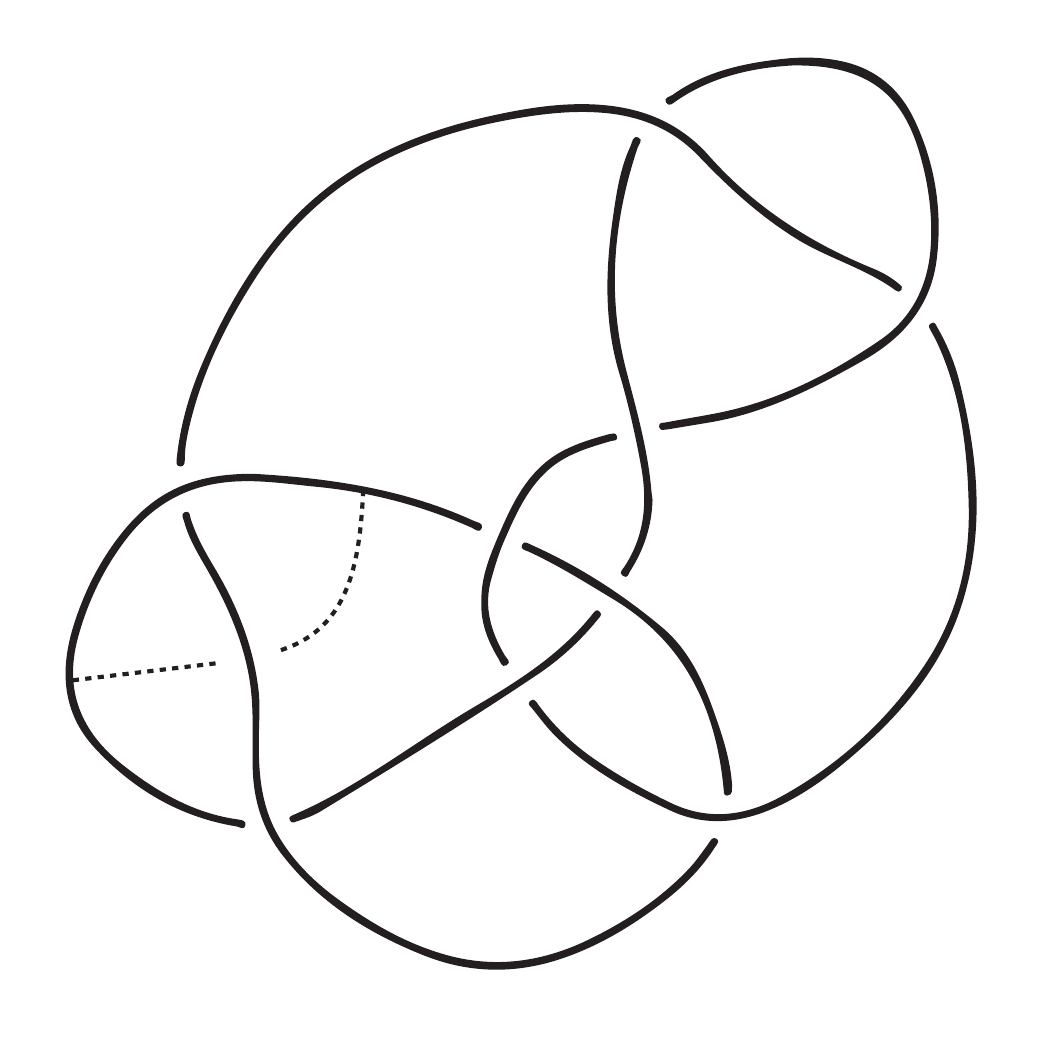}
		\caption{$9_{33}\stackrel{-1\phantom{i}}{\longrightarrow} 9_{45}$}
		\label{FigureFor9_33}
	\end{subfigure}
	~ 
	\begin{subfigure}[b]{0.29\textwidth}
		\includegraphics[width=\textwidth]{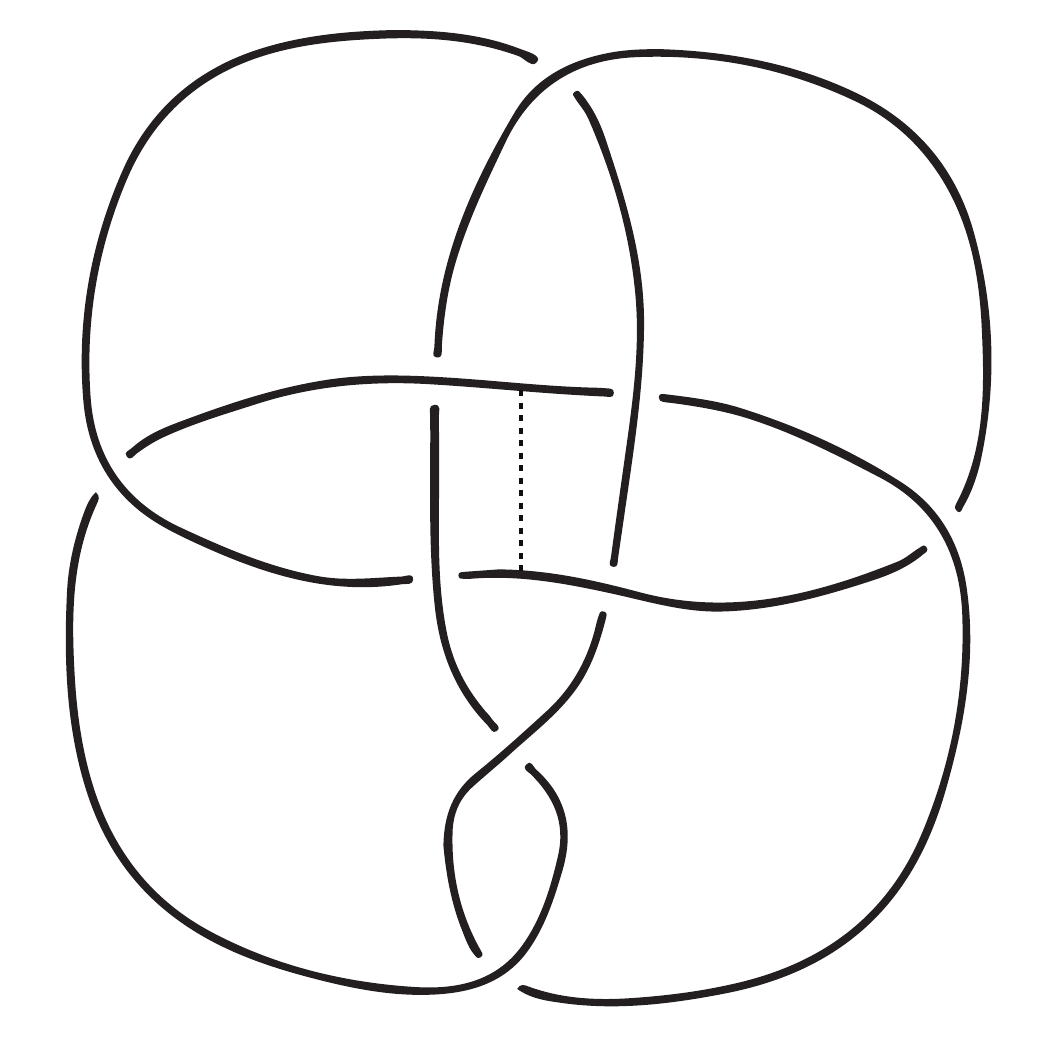}
		\caption{$9_{34}\stackrel{0}{\longrightarrow} 9_{28}$}
		\label{FigureFor9_34}
	\end{subfigure}
	\vskip3mm
	\begin{subfigure}[b]{0.29\textwidth}
		\includegraphics[width=\textwidth]{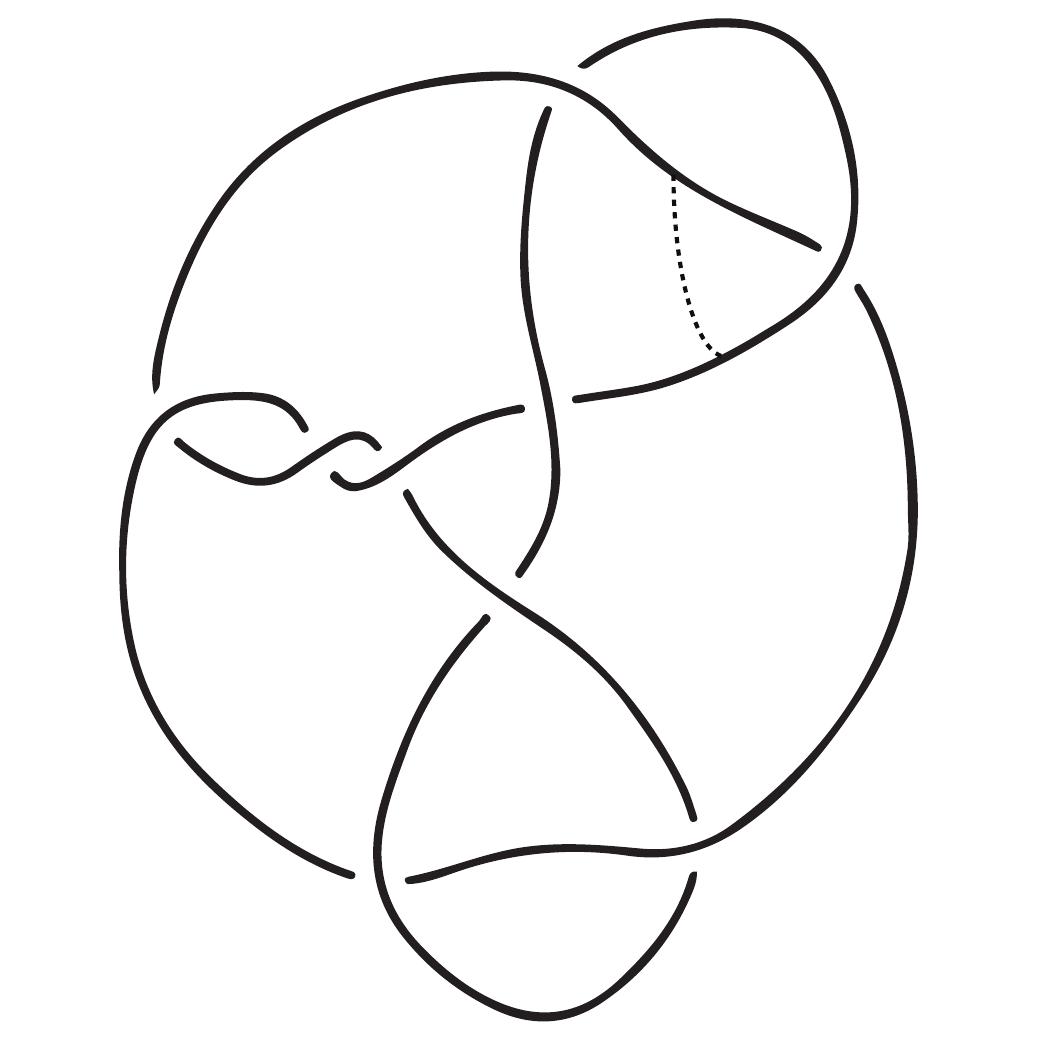}
		\caption{$9_{37}\stackrel{0}{\longrightarrow} 8_{10}$}
		\label{FigureFor9_37}
	\end{subfigure}
	~
	\begin{subfigure}[b]{0.3\textwidth}
		\includegraphics[width=\textwidth]{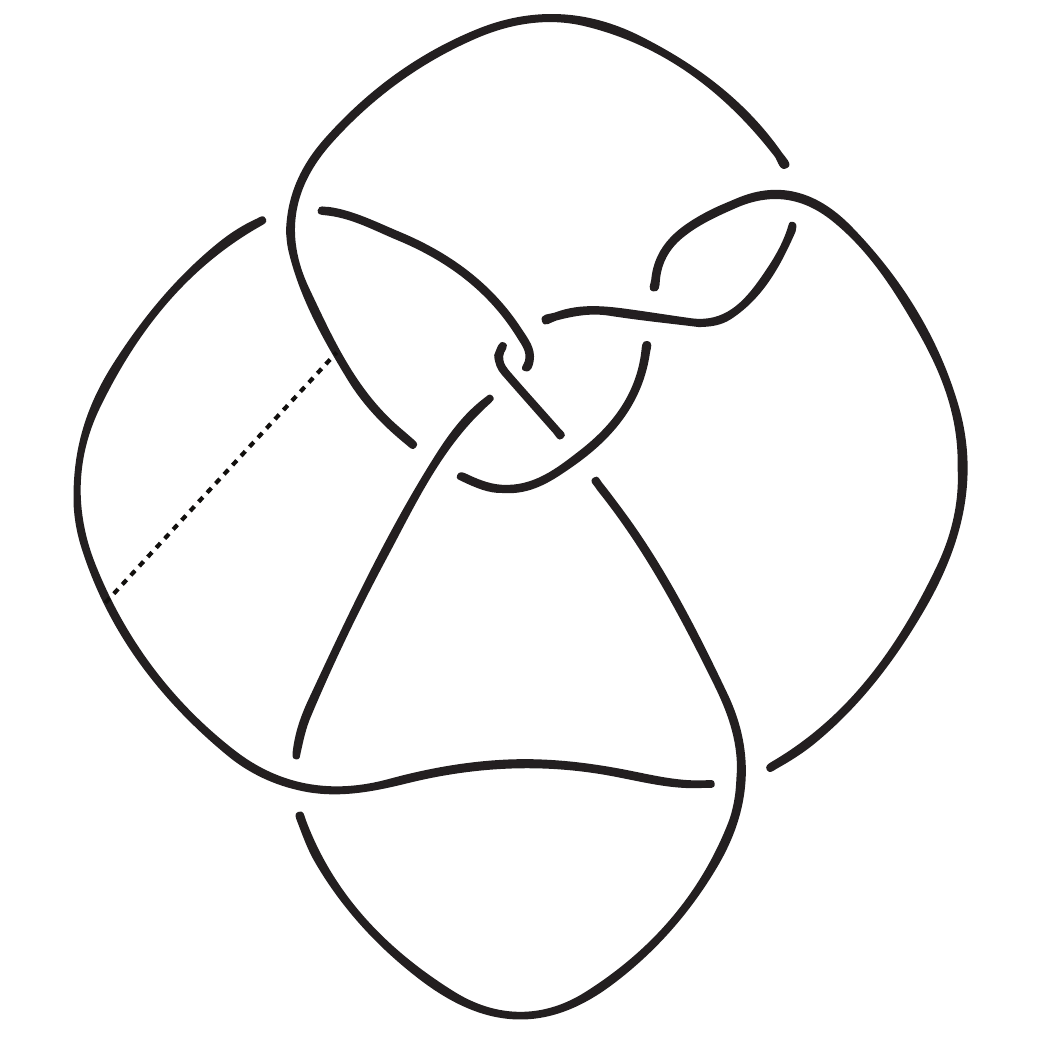}
		\caption{$9_{38}\stackrel{-1\phantom{i}}{\longrightarrow} 8_{14}$}
		\label{FigureFor9_38}
	\end{subfigure}
	~
	\begin{subfigure}[b]{0.28\textwidth}
		\includegraphics[width=\textwidth]{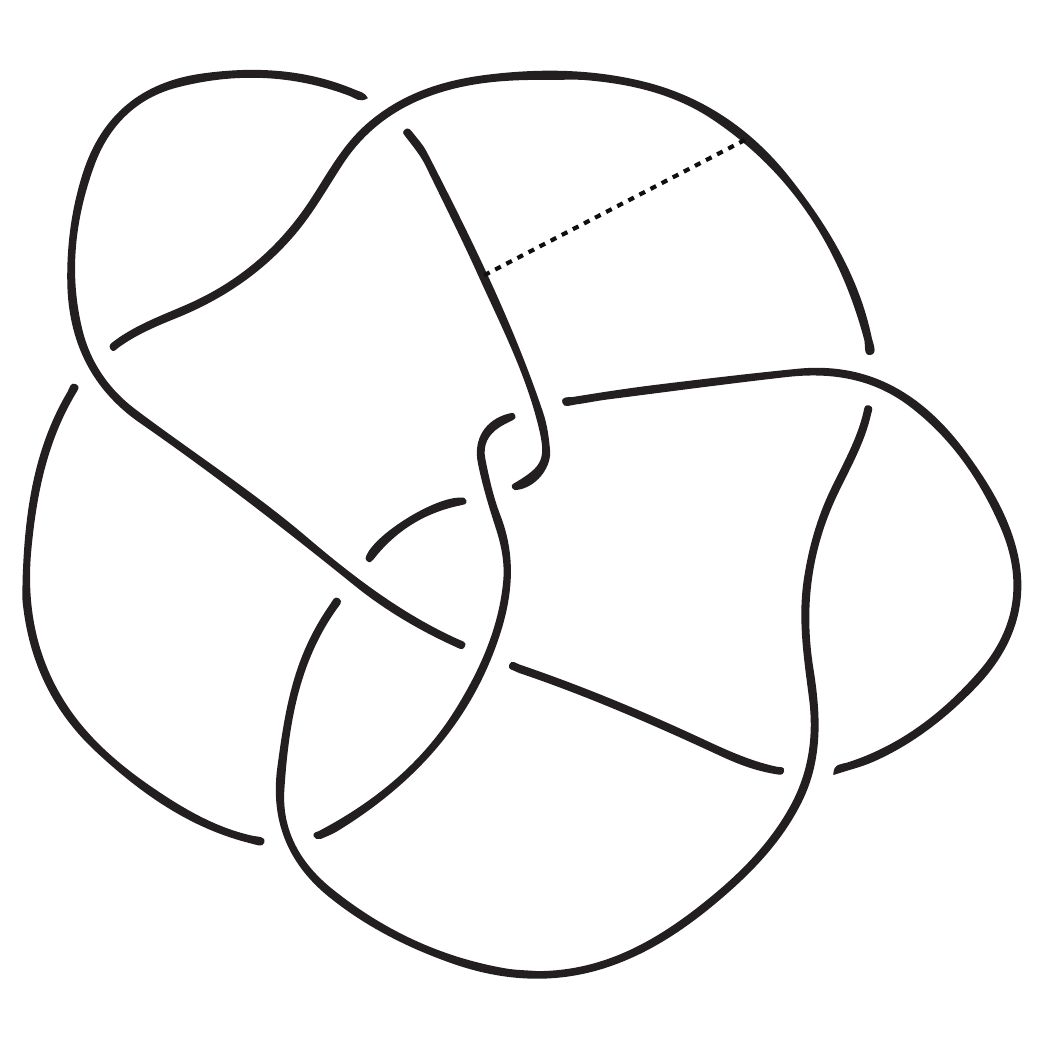}
		\caption{$9_{49}\stackrel{-1\phantom{i}}{\longrightarrow} 8_{21}$}
		\label{FigureFor9_49}
	\end{subfigure}
	\vskip3mm
	\caption{Non-oriented band moves from the knots $9_{18}$, $9_{20}$, $9_{24}$, $9_{30}$, $9_{33}$, $9_{34}$, $9_{37}$, $9_{38}$ to  knots with $\gamma_4$ equal to 1 and from $9_{49}$ to a knot with $\gamma_4$ equal to 2.}\label{9_18,9_20,9_24,9_30,9_34,9_37,9_38,9_49}
\end{figure}

All of these knots have already been considered in Section \ref{SectionOnSliceAndConcordantKnots}, with the exception of $K=9_4$ for which Figure \ref{HandleLabelingConvention} shows a band move to the slice knot $10_3$, demonstrating that $\gamma_4(9_4)=1$. 
%
%
\section{Concluding remarks} \label{SectionOnConcludingRemarks}
Murakami and Yasuhara \cite{MurakamiYasuhara2000} proved that 
\[
\gamma_4(K) \le \left\lfloor\frac{c(K)}{2} \right\rfloor
\]
where $c(K)$ is the crossing number of the knot $K$, and where $x\mapsto \lfloor x \rfloor$ is the ``floor function,'' giving the largest integer $n$ less than or equal to the real number $x$. For the case of a knot $K$ with $c(K) =8$ or $c(K)=9$ this inequality becomes $\gamma_4(K) \le 4$, an inequality which is strict for all such knots as demonstrated by Theorems \ref{8CrossingKnots} and \ref{9CrossingKnots}. The known values of $\gamma_4$ from \cite{Knotinfo} show that this inequality remains strict for all knots $K$ with $c(K)$ equal to either 3, 5, 6 or 7. However the above inequality does become an equality for $K=4_1$. These observations prompt the following question.
\begin{question}
Does there exist a knot $K$ with $c(K)>4$ and with $\gamma_4(K) = \left\lfloor\frac{c(K)}{2} \right\rfloor$? 
\end{question}

The knot $K=8_{18}$ is special among 8- and 9-crossing knots, being the only knot that maximizes $\gamma_4$, with a maximal value of 3. We note that $8_{18}$ is also special among this set of knots as it has the ``largest'' full symmetry group, namley the dihedral group $D_8$ (see \cite{Knotinfo}). Other knots with 8 or 9 crossings have smaller full symmetry groups, given by $\mathbb Z_i$ and $D_j$ with $i=1,2$ and $j=1, 2, 3, 4, 6$. The group $D_8$ does not appear again as the full symmetry group for any knot $K$ with $c(K) \le 11$, and only the knot $10_{123}$ has larger full symmetry group, namely $D_{10}$. However $10_{123}$ is slice and so $\gamma_4(10_{123}) = 1$. 
\begin{question}
Is there a connection between $\gamma_4(K)$ and the full symmetry group of a non-slice knot $K$?
\end{question} 
A beautiful result of Edmonds' \cite{Edmonds} stipulates that a $p$-periodic knot $K$ possesses a Seifert surface $S\subset S^3$ of genus $g_3(K)$ that is preserved under the $\mathbb Z_p$-action on $S^3$, making the connection between symmetries of a knot and its various genera plausible.

\clearpage


\end{document}